\documentclass[12pt]{amsart}

\makeatletter
\def\l@section{\@tocline{1}{10pt}{1pc}{}{}}
\def\l@subsection{\@tocline{2}{0pt}{1pc}{4.6em}{}}
\def\l@subsubsection{\@tocline{3}{0pt}{1pc}{7.6em}{}}
\renewcommand{\tocsection}[3]{%
  \indentlabel{\@ifnotempty{#2}{\makebox[2.3em][l]{%
    \ignorespaces#1 #2.\hfill}}}\textbf{#3}}
\renewcommand{\tocsubsection}[3]{%
  \indentlabel{\@ifnotempty{#2}{\hspace*{2.3em}\makebox[2.3em][l]{%
    \ignorespaces#1 #2.\hfill}}}#3}
\renewcommand{\tocsubsubsection}[3]{%
  \indentlabel{\@ifnotempty{#2}{\hspace*{4.6em}\makebox[3em][l]{%
    \ignorespaces#1 #2.\hfill}}}#3}
\makeatother  

\usepackage{amsmath,amssymb,amscd,amsthm,graphicx,enumerate,tikz-cd}

\usepackage{hyperref}
\hypersetup{breaklinks=true}

\usepackage{enumitem} 
\newlist{condenum}{enumerate}{1} 
\setlist[condenum]{label=\bfseries Condition \arabic*.,  ref=\arabic*, wide}

\numberwithin{equation}{section}
\theoremstyle{plain}



\usepackage[utf8]{inputenc}
\usepackage{amssymb}
\usepackage{amsthm}
\usepackage{mathrsfs}
\usepackage{amsfonts}
\usepackage{amsmath}
\usepackage{mathtools}
\usepackage{makecell}
\usepackage{xcolor}
\usepackage[margin=1.25in]{geometry}

\usepackage{lipsum}
\makeatletter
\def\ps@pprintTitle{%
 \let\@oddhead\@empty
 \let\@evenhead\@empty
 \def\@oddfoot{}%
 \let\@evenfoot\@oddfoot}
\makeatother

\newcommand{\intde}{\int_0^{\delta(\tau)^{-\frac 1{100}}}}
\numberwithin{equation}{section}

\newtheorem{theorem}{Theorem}[section]
\newtheorem{lem}[theorem]{Lemma}
\newtheorem{remark}[theorem]{Remark}
\newtheorem{prop}[theorem]{Proposition}
\newtheorem{claim}[theorem]{Claim}
\newtheorem{cor}[theorem]{Corollary} 

\theoremstyle{definition}
\newtheorem{defn}[theorem]{Definition}
\newtheorem*{theorem*}{Theorem}

\usepackage{xpatch}
\makeatletter
\xpatchcmd{\tableofcontents}{\contentsname \@mkboth}{\small\contentsname \@mkboth}{}{}
\xpatchcmd{\listoffigures}{\chapter *{\listfigurename }}{\chapter *{\small\listfigurename }}{}{}
\makeatother

\makeatletter
\def\blfootnote{\xdef\@thefnmark{}\@footnotetext}
\makeatother
\begin{document}

\title{Unique asymptotics of $SO(k)\times SO(n-k+1)$ symmetric ancient ovals of Ricci flow}
\author[Panagiota Daskalopoulos, Wenkui Du, Natasa Sesum, Ziyi Zhao]{Panagiota Daskalopoulos, Wenkui Du,  Natasa Sesum, Ziyi Zhao}
\begin{abstract}
     We obtain the unique asymptotics of $SO(k)\times SO(n-k+1)$-invariant,  compact, simply-connected, {factorwisely non-self-similar} $n$-dimensional $\kappa$-solutions of the  Ricci flow $(M^n, g(t))$, where $n\geq 4$ and $2\leq k\leq n-2$. More precisely, these $\kappa$-solutions are either ancient ovals of the  Ricci flow that are diffeomorphic to the standard sphere $S^n$, having a positive curvature operator metric $g(t)$ and a cylindrical tangent flow at $-\infty$, or they are a Riemannian product of a Perelman’s ancient oval and a
    shrinking round sphere.
     The  metric $g(t)$ of every $SO(k)\times SO(n-k+1)$-invariant ancient oval is represented  in the form  $g(t)=dz\otimes dz + F^2(z,t) g_{S^{k-1}}  + G^2(z,t)g_{S^{n-k}}$ (up to flipping $k-1$ and $n-k$). We obtain results about the blowdown limits of such solutions, establish the unique sharp asymptotics of the profile function $G(z, t)$,  and prove that the  uniqueness of $G(z, t)$ implies the  uniqueness of $F(z, t)$. In particular,  this provides the first instance of a classification result for geometric flows represented by a coupled PDE system, opening new avenues for studying the classification of higher-dimensional $\kappa$-solutions of the Ricci flow.

\end{abstract}
\maketitle

\tableofcontents

\section{Introduction}

The classification of singularity models has been a central driving force in the study of geometric flows largely due to their critical role in modeling the asymptotic geometry of high-curvature singularities.  The $\kappa$-solutions introduced by \cite{perelman2002entropy}  naturally arise in the study of singularity models of Ricci flow. Here, we give the following related definitions needed in our paper. 
\begin{defn}[$\kappa$-solution, cf. {\cite{haslhofer2024kappa, perelman2002entropy}}]\label{kappa solution}
    A $\kappa$-solution of Ricci flow is an ancient solution of the Ricci flow, that is complete with bounded curvature on compact time intervals, and 
 has nonnegative curvature operator, positive scalar curvature, and
is $\kappa$-noncollapsed at all scales.
\end{defn}

\begin{defn}[ancient ovals]\label{ancient oval defn}
For $n\geq 3$, an  $n$-dimensional ancient oval of the Ricci flow $(M^n, g(t))$ is defined as a  compact, simply-connected,  non-self-similar  $\kappa$-solution  with self-shrinking   cylindrical   tangent flow\footnote{For example, see \cite{bamler2020structure}, \cite{perelman2002entropy} and \cite{CMZ-asymptotic-solitons} for the definition of tangent flow at $-\infty$.} (blowdown) $\mathbb{R}^k\times S^{n-k}(\sqrt{2(n-k-1)|t|})$ at $-\infty$ for some $1\leq k\leq n-2$.
\end{defn}

\begin{defn}[factorwisely non-self-similar  Ricci flows]
A Ricci flow is called \emph{factorwisely non-self-similar} if it is
not factorwisely shrinking self-similar. Here, we say that  a Ricci flow
$(M,g(t))_{t\in I}$ is \emph{factorwisely shrinking self-similar} if, after
passing to the universal cover, there exists a time-independent
product decomposition
\[
(\widetilde M,\widetilde g(t))
=
(M_1,g_1(t))\times\cdots\times(M_N,g_N(t))
\]
such that, for each $\alpha\in\{1,\ldots,N\}$, there exist a time
$t_0\in I$, a constant $T_\alpha$ satisfying
\[
t<T_\alpha
\qquad\text{for every }t\in I,
\]
and a smooth family of diffeomorphisms
\[
\Phi_{\alpha,t}:M_\alpha\to M_\alpha,
\qquad
\Phi_{\alpha,t_0}=\operatorname{Id},
\]
for which
\[
g_\alpha(t)
=
\frac{T_\alpha-t}{T_\alpha-t_0}\,
\Phi_{\alpha,t}^{*}g_\alpha(t_0)
\qquad
\text{for every }t\in I.
\]
The constants $T_\alpha$ are not required to be equal. 
Apparently, factorwisely non-self-similar Ricci flow cannot be non-self-similar Ricci flow.
\end{defn}

In dimension three, 
the program of classification of all $\kappa$-solutions of Ricci flow proposed in \cite{haslhofer2024kappa} has culminated in a complete classification for  Ricci flows. For compact case, Angenent, Brendle, Daskalopoulos, and Sesum \cite{ABDS, BDS2021uniqueness} proved that any  compact $\kappa$-ancient $3$D solution is isometric to either a family of shrinking round spheres or the Perelman ancient oval. For noncompact solutions, Brendle \cite{Bre-3d-noncpt}  {showed that any complete noncompact  $3$D $\kappa$-ancient solution has to be rotationally symmetric, and that such  a solution is  necessarily either {the shrinking cylinder, or the quotient of shrinking cylinder}, or the steady Ricci soliton, hence the Bryant soliton as the unique simply-connected, noncollapsed 3D steady gradient Ricci soliton}. Beyond three dimensions, a work by Brendle, Daskalopoulos, Naff, and Sesum \cite{BDNS-compact} provided a {classification of compact $\kappa$-noncollapsed  solutions that have uniformly positive isotropic curvature (PIC) and are weakly PIC2,  showing they have to be either a family of shrinking spheres or the higher-dimensional Perelman ancient ovals.} Higher dimensional  asymptotics  and  classification of $\kappa$-solutions were studied in \cite{BN-noncompact}, \cite{zhao20234}, \cite{ZZ-notype2}, \cite{ma2023unique}, \cite{hebbar2026asymptotic} and \cite{hebbar2026kappa}. 

Despite the above successes  in 3D and the  pinched PIC setting, the general higher-dimensional landscape, where curvature conditions are relaxed, remains incredibly rich and much harder to fully understand. This complexity is sharply captured by Haslhofer's classification conjecture for 4D Ricci flow \cite{haslhofer2024kappa}, which claims that any simply-connected,  4-dimensional $\kappa$-solution must belong to a highly restricted, yet geometrically diverse, list of models. Specifically, the conjecture dictates that shrinking gradient solitons must be the round sphere ${S}^4$, $\mathbb{CP}^2$, ${S}^2\times {S}^2$ or the cylinders ${S}^3 \times \mathbb{R}$ and ${S}^2 \times \mathbb{R}^2$, and this has been verified in \cite{munteanu2017positively, DengZhu2020, naber4dshrinker}. Steady gradient solitons are conjectured to be the 4D Bryant soliton, the 3D Bryant soliton crossed with $\mathbb{R}$, or the 4D ``flying wing'' solutions (the non-rotationally symmetric steady solitons, which are $O(3)$-symmetric and  whose existence was established by Yi Lai \cite{Lai20}). Finally, for ancient solutions that are not solitons, the models are conjectured to be the 4D $O(4)$-symmetric Perelman oval, the 3D Perelman oval crossed with $\mathbb{R}$, the $O(2)\times O(3)$ symmetric ancient oval constructed by Buttsworth \cite{buttsworth2022so},  or the 1-parameter family of $O(3) \times \mathbb{Z}^2_2$-symmetric bubble-sheet ovals  constructed by Haslhofer in \cite{haslhofer2024kappa}  .


In the higher dimensional compact setting, the construction of $4$D ancient ovals of Ricci flow with $SO(2)\times SO(3)$ symmetry of Buttsworth in   \cite{buttsworth2022so} and Haslhofer \cite{haslhofer2024kappa} can be easily extended to the higher dimensional construction of an explicit family of ancient ovals possessing  $SO(k) \times SO(n-k+1)$ symmetry and  the families of $\mathbb{Z}_2^k\times O(n-k+1)$-symmetric ancient ovals.  A critical prerequisite to resolving broad classification conjectures is establishing a rigorous, quantitative understanding of all  higher-dimensional ancient models. In particular, in this paper we aim to study asymptotics  of any compact $SO(k)\times SO(n-k+1)$-symmetric ancient oval of Ricci flow.   In doing so, our results establish the sharp, unique asymptotic behavior of these solutions.

The trajectory of these investigations in Ricci flow closely mirrors breakthroughs in Mean Curvature Flow (MCF), where the classification program has rapidly accelerated. In MCF, the  work of Bamler and Lai \cite{bamler2025classification, bamler2025pde} completely resolved the classification of ancient asymptotically cylindrical solutions across arbitrary dimensions. Their framework provides a classification of  all ancient flows whose tangent flow at $-\infty$ is a generalized cylinder $\mathbb{R}^k \times S^{n-k}$. A subcase of Bamler-Lai's results contains the classification of the bubble-sheet solutions in $\mathbb{R}^4$, which are the $\mathbb{Z}^2_2\times O(2)$-symmetric ancient ovals whose tangent flow at $-\infty$ is ${{S}}^1\times\mathbb{R}^2$. These solutions were constructed in \cite{DH-oval}, and in the Bamler-Lai framework they correspond to the case when $k=2$ and $n=3$. Before their recent works,  progress in understanding the specific rigidity of this subcase was established in our prior work \cite{choi2025classification} 
classifying ancient bubble-sheet solutions in $\mathbb{R}^4$, and in the work by \cite{choi2024classification}. 
Building upon these foundations, generalizations regarding the rigidity and classification of higher-dimensional $k$-ovals in MCF, prior to the work of Bamler-Lai, have been achieved, that is, in  \cite{CDZ25} and \cite{CDZ26} the authors established the strict rigidity of $k$-ovals, successfully classifying them up to space-time rigid motions and parabolic dilations.

\subsection{Main results}
 {Let us now return to the Ricci flow and state the main results of this paper. First, we describe the geometric properties and blowdown behavior of ancient ovals and 
 of compact, simply-connected, factorwisely non-self-similar $\kappa$-solutions of the Ricci flow. These geometric properties are of independent interest, and some of them hold even without the symmetry assumption, while at the same time they played a crucial role in obtaining precise asymptotics of compact symmetric ancient ovals.

\begin{theorem}[geometry of ancient ovals and compact  factorwisely non-self-similar $\kappa$-solutions]\label{remark tangent flow}
Ancient ovals and compact,  factorwisely non-self-similar $\kappa$-solutions of the Ricci flow satisfy the following properties:
  \begin{itemize}
  \item Every ancient oval is diffeomorphic to the standard sphere and  has positive curvature operator.
   \item  Four-dimensional   compact,  simply-connected, { factorwisely non-self-similar} $\kappa$-solutions  must be ancient ovals, whose unique tangent flow at $-\infty$ can only be one of the 
    self-shrinking cylinders $\mathbb{R}\times S^{3}$ 
      or $\mathbb{R}^2\times S^{2}$.
    \item  Let $(M^n,g(t))$ be a
$SO(k)\times SO(n-k+1)$-symmetric, compact, simply-connected, factorwisely
non-self-similar $\kappa$-solution. Then exactly one of the
following alternatives occurs.

\begin{enumerate}
\item
The flow is an ancient oval. Its unique tangent flow at
$-\infty$ is either
\[
\mathbb R^k\times
S^{n-k}\bigl(\sqrt{2(n-k-1)|t|}\bigr),
\]
where $2\le n-k\le n-1$, or
\[
\mathbb R^{n-k+1}\times
S^{k-1}\bigl(\sqrt{2(k-2)|t|}\bigr),
\]
where $2\le k-1\le n-1$.

\item
the flow is a Riemannian product of a
Perelman's ancient oval and a shrinking round sphere. Its tangent
flow at $-\infty$ is
\[
\mathbb R\times
S^{k-1}\bigl(\sqrt{2(k-2)|t|}\bigr)
\times
S^{n-k}\bigl(\sqrt{2(n-k-1)|t|}\bigr),
\]
where $k-1\geq 2$ and $n-k\geq 2$. More precisely, up to interchanging the factors, the flow is either a product of a $k$-dimensional Perelman ancient oval and a shrinking round  $S^{n-k}$ or a product of an $(n-k+1)$-dimensional Perelman ancient oval and a shrinking round $S^{k-1}$.
\end{enumerate}
    \item For any  sequence of times $t_i \rightarrow-\infty$ and an arbitrary sequence of points $x_i$ on a $SO(k)\times SO(n-k+1)$-symmetric $n$-dimensional ancient oval $(M^{n}, g(t_i))$ with $\mathbb R^{k} \times {{S}}^{n-k}$ type cylindrical tangent flow at $-\infty$,
    the rescaled flow around the points $\left(x_i, t_i\right)$ by the scalar curvature $R\left(x_i, t_i\right)$   converges  (after passing through a subsequence) 
    to a shrinking cylinder or the $\mathbb{R}^{k-1}$ times the Bryant soliton.
  \end{itemize}
\end{theorem}

Next, we obtain the following result concerning the unique sharp asymptotics of $SO(k) \times SO(n-k+1)$-symmetric, $n$-dimensional ancient ovals. 
In particular, we will assume $n\geq 4$ and $1\leq k\leq n-2$, since in the case when $n=3$ and $k=1$, the classification has been addressed in \cite{ABDS, BDS2021uniqueness}, and when $n\geq 4$ and $k=1$, the classification has been addressed by the results in \cite{BDNS-compact, hebbar2026kappa}. }


\begin{theorem}[unique sharp asymptotics of $SO(k) \times SO(n-k+1)$-symmetric ancient ovals]\label{unique asymptotics theorem}  
Assume that $(M^n, g(t))$ is an $n$-dimensional, $SO(k) \times SO(n-k+1)$-symmetric ancient oval of the Ricci flow, where $n\geq 4$ and $2\leq k\leq n-2$. 
Suppose that the cylinder $\mathbb{R}^k \times S^{n-k}(\sqrt{2(n-k-1)|t|})$ is the tangent flow at $-\infty$ of $(M^n, g(t))$. Then, for the metric $g(t)$ written in the warped product form
\begin{equation}\label{eqn-metric} 
    g(t) = dz \otimes dz + F^2(z,t) g_{S^{k-1}} + G^2(z,t) g_{S^{n-k}},
\end{equation}
where $z\in [0, z_{\mathrm{max}}(t)]$,
we have the following unique sharp asymptotics for $G(z, t)$ as $t \to -\infty$:

\begin{itemize}
    \item {\bf Parabolic region.} For any constant $L > 0$, the profile function $G$ satisfies
    \begin{equation}
    G(z,t)^2 = 2(n-k-1)(-t)   - \frac{(n-k-1)(z^2+2kt)}{2\log(-t)}  + o\left(\frac{-t}{\log(-t)}\right) 
    \end{equation}
    in $C^{\infty}_{loc}$-sense    for all $0\leq z \leq L\sqrt{-t}$.
    \item {\bf Intermediate region.} For a sufficiently small $\theta > 0$, we have that 
    \begin{equation}
    G(z, t)^2 = -2(n-k-1) t - (n-k-1)\frac{z^2}{2 \log (-t)} + o(-t)
    \end{equation}
     in $C^{\infty}_{loc}$-sense  for $0\leq z \leq 2\sqrt{1-\theta^2} \sqrt{(-t) \log (-t)}$. 
    \item {\bf Tip region.} If $g_{i}(t) = R(p_i, t_i)g(t_i + R(p_i, t_i)^{-1}t)$ is a sequence of metrics obtained by rescaling the ancient solution around the tip points $(p_i, t_i) \in \{G(z, t_i)=0\}$ with $t_i \to -\infty$, then the pointed rescaled solutions $(M^n, g_i(t), p_i)$ converge smoothly to the unique $(n-k+1)$-dimensional Bryant soliton crossed with $\mathbb{R}^{k-1}$.
\end{itemize}
\end{theorem}

Then, the following result allows one to algebraically reconstruct the profile $F$ from the profile $G$. In particular, the uniqueness and unique asymptotics of $F$ follow directly from the uniqueness and unique asymptotics of $G$.

\begin{theorem}[uniqueness of the $G$-profile implies uniqueness of the $F$-profile]\label{GuniqueimpliesFunique} 
Under the same assumptions as in Theorem \ref{unique asymptotics theorem}, for $z\in (0, z_{\mathrm{max}}(t)]$ the profile function $F(z, t)$ in \eqref{eqn-metric} can be explicitly represented in terms of $G(z, t)$ by
\begin{equation}\label{FQG}
F(z, t) = z \exp \left( \int^{z}_{0} \left[ \sqrt{Q^G(\tilde{z}, t)} - \frac{1}{\tilde{z}} \right] d\tilde{z} \right), 
\end{equation}
where
\begin{equation}
Q^G(z,t) = -\frac{1}{k-1} \left[ \partial_z \left( \frac{G_t - G_{zz} + (n-k-1)\frac{1-G_z^2}{G}}{G_z} \right) + (n-k)\frac{G_{zz}}{G} \right],
\end{equation}
and the representation formula \eqref{FQG} extends smoothly to $z=0$.

In particular, for two $SO(k) \times SO(n-k+1)$-symmetric $n$-dimensional ancient ovals $(M^n, g_i(t))$,  $(i=1, 2)$,  of the Ricci flow, after suitably fixing the time translation, parabolic
scaling, origin and orientation of the $z$-coordinate, and the
labeling of the two spherical factors, if they are represented by the profile functions $(F_i(z, t), G_i(z, t))$ as in \eqref{eqn-metric}, then 
\begin{equation}
G_1(z, t) = G_2(z, t)
\end{equation}
implies
\begin{equation}
F_1(z, t) = F_2(z, t).
\end{equation}
\end{theorem}

Let us emphasize a major geometric distinction that sets the present work apart from much of the prior classification literature. To date, in virtually all known rigorous classification results for ancient solutions of the Ricci flow (such as the classical Perelman ancient oval), the proofs rely heavily on the solution having a standard round cylinder $\mathbb{R} \times S^{n-1}$ as its tangent flow at $-\infty$, while the geometry at the tip points is modeled by the $n$-dimensional Bryant steady soliton. The present paper tackles ancient Ricci flows whose tangent flows at $-\infty$ are given by {$\mathbb{R}^k \times S^{n-k}$}, while the geometry at the high-curvature tips features a non-compact, higher-dimensional Euclidean factor. Specifically, under parabolic rescaling at the tip points, these solutions converge to the Bryant soliton times $\mathbb{R}^{k-1}$. 

The presence of this higher-dimensional Euclidean factor in the asymptotics of both the parabolic and the tip regions makes the study of these solutions considerably more difficult—both geometrically and analytically—than in earlier Ricci flow literature. In higher dimensions, more lines can split off in the blowdown analysis. Because the tangent flow at $-\infty$ only provides line-splitting information near the parabolic center, obtaining this information for arbitrary pointed blowdown sequences presents a subtle geometric obstacle. Consequently, a major challenge of this work is managing the transition between the parabolic center of the oval and the non-compact $\text{Bryant} \times \mathbb{R}^{k-1}$ tip region. 

For earlier works where $k=1$, the blowdown analysis relied on known classifications of complete, non-compact $\kappa$-solutions. Because such classifications remain unavailable for $k \ge 2$, we perform a delicate blowdown analysis of $SO(k) \times SO(n-k+1)$-symmetric ancient ovals along arbitrary pointed blowdown sequences. A key ingredient of our approach is demonstrating the splitting of at least $k-1$ Euclidean factors in the blowdown limits. Establishing this structural splitting allows us to show that every pointed blowdown limit is either a standard shrinking cylinder or $\mathbb{R}^{k-1}$ times the Bryant soliton.



Furthermore, analyzing these $SO(k) \times SO(n-k+1)$-symmetric ancient ovals introduces a substantial PDE difficulty. Without loss of generality, we can assume that the tangent flow at $-\infty$ of $(M, g(t))$ is the cylinder $\mathbb{R}^{k} \times S^{n-k}(\sqrt{2(n-k-1)|t|})$. Recall that the evolving metric can be written as a double warped product, as in \eqref{eqn-metric}, and thus is governed by a fully coupled system of nonlinear parabolic PDEs for the profile functions $(F(z, t), G(z, t))$. Here, the conformal factor $G$ controls the shrinking $S^{n-k}$ fiber, while the profile function $F$ controls the $S^{k-1}$ fiber. Because previous geometric classification results have overwhelmingly focused on scalar profile equations, this coupled interaction requires new techniques to prevent the evolution of one fiber from destabilizing the other. 

More precisely, in our symmetric setting, the Ricci flow equation for the metric reduces to a fully coupled system for $F$ and $G$, where both equations contain non-local terms depending on both functions. This mutual dependence makes the analysis and the derivation of the estimates needed to establish the desired asymptotics more challenging than in previous rotationally symmetric settings. Recall that in the 3D case \cite{Bre-3d-noncpt, ABDS, BDS2021uniqueness} and certain higher-dimensional cases \cite{BN-noncompact, BDNS-compact}, the authors first showed that the ancient solutions must be rotationally symmetric, and then established the asymptotics for those rotationally symmetric solutions. In those settings, the analysis reduced to a scalar PDE for a single function containing a non-local term. 

Let us also emphasize that in the present paper, the linearization of the Ricci flow equation around the cylinder $\mathbb{R}^k \times S^{n-k}$ yields two linear operators—denoted below as $\mathcal{L}_1$ and $\mathcal{L}_2$—in \eqref{L2intro}  and \eqref{L3intro} that are distinct in structure. Consequently, we must work with two different function spaces equipped with different weighted norms, which introduces  technical obstacles when estimating the coupled error terms in both linearized equations.

Lastly, we discuss the uniqueness implications of Theorem \ref{unique asymptotics theorem}. The fact that the unique asymptotics of $G(z, t)$ are independent of $F(z, t)$, combined with the rigidity relation between the two profile functions, suggests a  geometric ``locking" mechanism in the general phenomenon of ancient, $\kappa$-noncollapsed Ricci flows: namely, that the geometry of the $S^{n-k}$ fiber strictly dictates the behavior of the $S^{k-1}$ fiber.

Using the precise asymptotics obtained in Theorem \ref{unique asymptotics theorem}, the  uniqueness of the profile $G(z,t)$ itself—and thus the final completion of the classification—will be addressed in a subsequent work. 
Note that the constructions in \cite{buttsworth2022so} and \cite{haslhofer2024kappa} can be easily generalized to higher dimensions, yielding $SO(k) \times SO(n-k+1)$-symmetric ancient ovals with  tangent flows at $-\infty$ being either $\mathbb{R}^{k} \times S^{n-k}(\sqrt{2(n-k-1)|t|})$ with $2 \leq n-k \leq n-1$, or $\mathbb{R}^{n-k+1} \times S^{k-1}(\sqrt{2(k-2)|t|})$ with $2 \leq k-1 \leq n-1$. Once the uniqueness result is established in our subsequent paper, it will follow that the $SO(k) \times SO(n-k+1)$-symmetric ancient ovals found in \cite{buttsworth2022so} and \cite{haslhofer2024kappa} are identical.

\subsection{Outline of the paper} In this section, we provide a brief overview of the strategy used to establish our main theorems. The proof is structured around three core components: the classification of arbitrary pointed blowdown limits, the derivation of sharp asymptotic estimates for the profile functions, 
and the proof of the geometric locking mechanism that recovers the full solution from the primary fiber. 

\smallskip 
{
In Section \ref{Preliminaries}, we establish the geometric framework and carry out the blowdown analysis for the ancient ovals considered in this paper, which leads to the proof of Theorem \ref{remark tangent flow}. We first show that compact, simply-connected, factorwisely non-self-similar $\kappa$-solutions with \(SO(k)\times SO(n-k+1)\)-symmetry are either ancient ovals or Perelman's ancient oval times shrinking sphere, and in dimension four without assuming such symmetry they are ancient ovals. In particular, ancient ovals are spherical and have positive  curvature operator and  cylindrical tangent flows at \(-\infty\). We then perform a detailed blowdown analysis of $SO(k) \times SO(n-k+1)$-symmetric ancient ovals along arbitrary pointed blowdown sequences. A key ingredient of this analysis is characterizing how the Euclidean factors split off in the blowdown limits. Although the tangent flow at $-\infty$ guarantees line splitting only near the parabolic center, we prove that any sequence of points whose $SO(n-k+1)$-orbit radius remains on the cylindrical scale yields a blowdown limit containing at least $k-1$ straight lines when rescaled by the scalar curvature. This line-splitting information allows us to cleanly separate the cylindrical region from the tip region, to show that points of maximal curvature are localized near the tips, and to prove that the rescaled  flow  by the scalar curvature around these tip points converges smoothly to $\mathbb{R}^{k-1}$ times the Bryant steady soliton. Consequently, we conclude that every pointed, curvature-scale blowdown limit is either the standard shrinking cylinder or $\mathbb{R}^{k-1}$ times the Bryant soliton.

}

In Section \ref{Evolution equations and barrier}, we set up some preliminaries regarding  the evolution equations 
of the profile  functions $(F(z, t), G(z, t))$. We then analyze the evolution equations of the renormalized profiles and their corresponding linearized operators acting on suitable weighted function spaces. More precisely, we consider the renormalized profile functions  $(f(\xi, \tau), g(\xi, \tau))$ defined by
\begin{equation}
    f(\xi, \tau) = e^{\frac{\tau}{2}}F(e^{-\frac{\tau}{2}}\xi, -e^{-\tau})-\xi,\,
    g(\xi, \tau) = e^{\frac{\tau}{2}}G(e^{-\frac{\tau}{2}}\xi, -e^{-\tau})-\sqrt{2(n-k-1)}.
\end{equation}
We know that $h(\xi, \tau)=f_{\xi}(\xi, \tau)$ and $f$ with $f(0, \tau)=0$ {uniquely determine each other.  } The rescaled functions 
$g$ and $h$ satisfy the evolution equations 
\begin{align}
   g_\tau =\mathcal{L}_1 g+E_1,
\end{align}
where
\begin{equation}\label{L2intro}
   \mathcal{L}_1g= g_{\xi\xi}+\frac{k-1}{\xi}g_{\xi} - \frac{1}{2}\xi g_\xi + g,
\end{equation}
and
\begin{align}
   h_\tau =\mathcal{L}_2 h+E_2,
\end{align}
where 
\begin{align}\label{L3intro}
\mathcal{L}_2 h = h_{\xi\xi} - \frac{1}{2}\xi h_\xi +\frac{k-3}{\xi}h_\xi - \frac{2k-4}{\xi^2}h
\end{align} 
{The rescaled profile $g$ will be considered on  the  Gaussian weighted space}  $\mathcal{H}=L^{2}([0, +\infty), \xi^{k-1}e^{-\frac{\xi^2}{4}})$  (for radial functions) with inner product 
\begin{equation}
    \langle g_1, g_2\rangle_{\mathcal{H}}=\int_{0}^{+\infty} g_1 g_2 \, \xi^{k-1}e^{-\frac{\xi^2}{4}}d\xi.
\end{equation}
The operator   $\mathcal{L}_1$ is  self-adjoint  on $\mathcal{H}$. Note that 
 $1$ is the only eigenfunction with respect to the only positive eigenvalue $1$,  and $\xi^{2}-2k$ spans the neutral mode eigenspace. {The rescaled profile $h$ will be considered on   the function space } 
\begin{equation}\label{mathscrHdef}
    \mathscr{H}=\{h\in L^{2}([0, +\infty), \xi^{k-3}e^{-\frac{\xi^2}{4}}): h(0, \tau)=h(0, \tau)=0\}
\end{equation}
with inner product 
\begin{equation}
    \langle h_1, h_2\rangle_{\mathscr{H}}=\int_{0}^{+\infty} h_1h_2\,  \xi^{k-3}e^{-\frac{\xi^2}{4}}d\xi.
\end{equation}
The operator
\begin{align}
\mathcal{L}_2 h = h_{\xi\xi} - \frac{1}{2}\xi h_\xi +\frac{k-3}{\xi}h_\xi - \frac{2k-4}{\xi^2}h
\end{align} 
is a self-adjoint operator on ${\mathscr{H}}$ with only negative
spectrum \begin{equation}
    \sigma(\mathcal{L}_2)= \sigma(\tilde{\mathcal{L}})=\{-(m+1)\}_{m=0}^{\infty}. 
\end{equation} In the same section, we also recall Brendle's barrier in \cite{Bre-3d-noncpt} and its properties.

In Section \ref{Spectral alternatives},  we examine the spectral mode alternatives for the suitably truncated, renormalized profile functions 
$(g,h)$. By carefully accounting for the effect of $f$-related terms, we first derive refined estimates for the highly coupled error terms 
$E_1$ and $E_2$ as summarized below: 
\begin{align}
    &\quad\|E_1\cdot \chi_{\{0 \leq \xi\leq \delta(\tau)^{-\frac{1}{100}}\}}\|^2_{\mathcal{H}}\\\notag
&\leq C\delta(\tau)^{\frac{1}{100}} \|h\cdot \chi_{\{0 \leq \xi\leq \delta^{-\frac{1}{100}}\}}\|^2_{\mathscr{H}} +C\delta(\tau)^{\frac{1}{100}} \|g\cdot \chi_{\{0 \leq \xi\leq \delta(\tau)^{-\frac{1}{100}}\}}\|^2_{\mathcal{H}} + C \exp \left(-\frac{1}{8} \delta(\tau)^{-\frac{1}{50}}\right)\notag
\end{align}

\begin{align}
  &\quad\|E_2\cdot \chi_{\{0 \leq \xi\leq \delta(\tau)^{-\frac{1}{100}}\}}\|^2_{\mathscr{H}}\\\notag
&\leq \epsilon\|h\cdot \chi_{\{0 \leq \xi\leq \delta(\tau)^{-\frac{1}{100}}\}}\|^2_{\mathscr{H}} +C\delta(\tau)^{\frac{1}{100}} \|g\cdot \chi_{\{0 \leq \xi\leq \delta(\tau)^{-\frac{1}{100}}\}}\|^2_{\mathcal{H}} + C \exp \left(-\frac{1}{8} \delta(\tau)^{-\frac{1}{50}}\right)\notag
\end{align}
where $\epsilon>0$ is sufficiently small, $\chi_{\{0 \leq \xi\leq \delta(\tau)^{-\frac{1}{100}}\}}$ is a cutoff function and
\begin{align}
    \delta({\tau}):=\sup _{\tilde{\tau} \leq {\tau}}\left(|g(0, \tilde{\tau})|+|g_{\xi\xi}(0, \tilde{\tau})|\right)\to 0\quad \textrm{as}\quad\tau\to -\infty. 
\end{align}

{Then, for the mode analysis of  the linear operator $\mathcal{L}_1$ appearing  in the equation for $g$, we   employ Brendle's barrier
technique in  \cite{Bre-3d-noncpt}.  This is justified by the observation  that the $F(z, t)$-related terms enter the equation for $U(r, t)$ (the inverse profile function  of $G^2(z, t)$)  with a  favorable sign, which allows  us to control $g$ in the coupled system of $(f, g)$. Together with  the error estimates above,  and the fact that the linear operator $\mathcal{L}_2$   in the equation for  $f$ has only strictly negative modes, we can establish Merle-Zaag type differential inequalities for $h$ and for positive, neutral and negative mode projections of $g$  following \cite{ABDS, zaag2001liouville}.
This ultimately shows that the full norm of $h$  is negligible compared to either the norm of a  dominant neutral mode projection or a dominant unstable mode projection of $g$.

In Section \ref{rule out+}, we show that the neutral mode is dominant. By a careful analysis of the regularity estimates for $h=f_{\xi}$ in Appendix \ref{regularity estimates of tildeh}, and by precisely accounting for the effect of $F$ in Perelman's pointwise curvature derivative estimates, we can apply a diameter estimates argument for these compact ancient ovals following \cite{ABDS}. This successfully rules out the dominance of unstable modes for $g$ that arise from the non-compact tip geometries. 

In Section \ref{Unique sharp asymptotics of  G}, we prove the results stated in Theorem \ref{unique asymptotics theorem} and Theorem \ref{GuniqueimpliesFunique}. For Theorem \ref{unique asymptotics theorem}, the mode analysis of $g$ and $h$ mentioned above  yields the sharp asymptotics in the parabolic region. For the intermediate region asymptotics, the equations for $G(z, t)$ and its inverse profile $U(r, t)$ (the inverse of $G^2(z, t)$) contain favorable signs from the $F(z, t)$-related terms. This allows us to apply a barrier argument using Brendle's barriers alongside the method of characteristics as in \cite{ABDS}. For the tip region asymptotics, we utilize the $SO(k)\times SO(n-k+1)$ symmetry, the new line-splitting ingredient for blowdown limits shown in Section \ref{Preliminaries}, and the arguments from \cite{ABDS}. 

Finally, we prove Theorem \ref{GuniqueimpliesFunique}, which asserts that the uniqueness of $G(z, t)$ implies the uniqueness of $F(z, t)$; this follows directly from differentiating the evolution equation for $G(z, t)$ with respect to the $z$-variable.
}

\bigskip

\textbf{Acknowledgments:}  The first author
has been supported by the  NSF grants DMS 1900702 and DMS 2454018. The second author has been supported by
Massachusetts Institute of Technology.  The third author has been  supported by the NSF
grants DMS 2105508 and DMS 2505574. The fourth author has been supported by Institute for Theoretical Sciences at Westlake University and Zhejiang Normal University.   The authors thank Professor Robert Haslhofer for
suggesting the problem and Zilu Ma for helpful communications.

\section{Geometry of ancient ovals and their blowdown analysis}\label{Preliminaries}
In this section, we prove Theorem \ref{remark tangent flow}. In particular, we examine the basic geometry of ancient ovals and their relationship to compact,  simply-connected, factorwisely non-self-similar $\kappa$-solutions. We then analyze the blowdown limits of $SO(k)\times SO(n-k+1)$-symmetric ancient ovals along an arbitrary sequence of times $t_i \to -\infty$.

\subsection{Basic geometric properties of ancient ovals}\label{geometry of ancient oval}

In this subsection, we discuss the basic 
properties of compact,  simply-connected, factorwisely non-self-similar $\kappa$-solutions. 
In particular, we establish that compact,  simply-connected, factorwisely non-self-similar $\kappa$-solutions are identical to ancient ovals in the four-dimensional case, as well as in the case of $SO(k)\times SO(n-k+1)$-symmetric solutions.

\begin{lem}\label{lemma-sphere}
    Every ancient oval is diffeomorphic to the standard sphere and has a positive curvature operator.
\end{lem}
\begin{proof}
     Every ancient oval $(M^n, g(t))$ must have a positive curvature operator. Indeed, by the classification results for compact, simply connected Riemannian manifolds with a nonnegative curvature operator in \cite[Theorem 7.34]{CLN} and \cite{BS-weakly-PIC2}, $M$ must be isometric to a product of the form 
     $N_1 \times \ldots \times N_j$, where $N_1, \ldots, N_j$ are compact, simply connected, and irreducible, and each factor $N_1, \ldots, N_j$ is either diffeomorphic to a sphere, a K\"ahler manifold biholomorphic to complex projective space, or isometric to a compact symmetric space. By our assumption, $(M, g(t))$ is compact and has the cylinder $\mathbb R^{k} \times {{S}}^{n-k}(\sqrt{2(n-k-1)|t|})$ as its tangent flow at $-\infty$; thus, there are either one or two compact non-Euclidean splitting factors in $(M, g(t))$. We claim that ancient oval $(M, g(t))$ has exactly one compact factor. Arguing by contradiction, suppose that $M$ is isometric to product of compact factors $N_1\times N_2$, which implies that the tangent flow of $M$ at $-\infty$ is the product of the tangent flows at $-\infty$ of $N_1$ and $N_2$. Thus, by the above assumption of a cylindrical tangent flow at $-\infty$, one of the tangent flows at $-\infty$ of $N_1$ or $N_2$ must be a Euclidean space. Then, {by Perelman's monotonicity formula}, it follows that either $N_1$ or $N_2$ must be the Euclidean space itself. This contradicts the compactness of $M$. By the results in \cite{li2024ancient, DengZhu2020}, our definition of an ancient oval in Definition \ref{ancient oval defn}, the definition of a $\kappa$-solution in Definition \ref{kappa solution}, and the self-shrinking property of symmetric spaces with nonnegative Ricci curvature, we can exclude the aforementioned compact symmetric space and K\"ahler manifold possibilities. Hence, $(M, g(t))$ has exactly one non-Euclidean compact factor, and $(M, g(t))$ must be diffeomorphic to a sphere and has positive curvature operator by Hamilton’s strong maximum principle.
\end{proof}

By definition, we know that ancient ovals must be compact, non-self-similar $\kappa$-solutions. Conversely, we have the following lemmas. First, in the four-dimensional case, we can remove the cylindrical tangent flow assumption from the definition of ancient ovals.

\begin{lem}\label{4d tangent flow}
   Four-dimensional, compact,  simply-connected, factorwisely non-self-similar $\kappa$-solutions must be ancient ovals. 
\end{lem}
\begin{proof}
     For four-dimensional, compact,  simply-connected, factorwisely non-self-similar $\kappa$-solution $(M^4, g(t))$, by \cite{munteanu2017positively}, the tangent flow at $-\infty$ must be one of the gradient shrinking solitons
      ${S}^4$, $\mathbb{CP}^2$, $\mathbb{R}\times {S}^3$, $\mathbb{R}^2\times {S}^2$,  ${S}^2\times {S}^2$. The case of having a shrinking ${S}^4$ as the tangent flow at $-\infty$ of $(M^4, g(t))$ can be easily ruled out by the monotonicity formula. In the case of having  $\mathbb{CP}^2$ as the tangent flow at $-\infty$ of $(M^4, g(t))$, for sufficiently negative time, $(M^4, g(t))$ must be diffeomorphic to  $\mathbb{CP}^2$. By Berger’s holonomy classification, $(M^4, g(t))$ is compact $\kappa$-solutions to the K\"ahler Ricci flow, by the classification result in \cite[Theorem 1.2]{DengZhu2020}, $(M^4, g(t))$ must be isometric to $\mathbb{CP}^2$, which is self-similar solution.
      In the case where ${S}^2\times {S}^2$ is a tangent flow at $-\infty$ of $(M^4, g(t))$, for sufficiently negative time, $(M^4, g(t))$ must be diffeomorphic to ${S}^2\times {S}^2$. Similar to the arguments in the proof of Lemma \ref{lemma-sphere}, by \cite[Theorem 7.34]{CLN}, we know the metric must be the product of two 2-dimensional metrics. By pinching estimates in \cite{BS-weakly-PIC2} applied to each metric factor, we know the forward singularity of $(M^4, g(t))$ must be ${S}^2\times {S}^2$. Then, again by Perelman's monotonicity formula, $(M^4, g(t))$ has to be isometric to a shrinking ${S}^2\times {S}^2$, which contradicts the definition of a factorwisely non-self-similar $\kappa$-solution. Hence, the tangent flow at $-\infty$ of four-dimensional, compact, simply-connected, factorwisely non-self-similar $\kappa$-solutions must be cylindrical, i.e., either $\mathbb{R}\times{S}^3$ or $\mathbb{R}^2\times {S}^2$, and thus $(M,g(t))$ is an ancient oval.
\end{proof}

In addition, we have that an $SO(k)\times SO(n-k+1)$-symmetric, compact, simply-connected, factorwisely non-self-similar  $\kappa$-solution must have a unique cylindrical tangent flow at $-\infty$ (up to diffeomorphism).  

\begin{lem}[blowdown of $SO(k)\times SO(n-k+1)$-symmetric,  compact, simply-
connected, factorwisely non-self-similar $\kappa$-solutions]\label{blowdown is cylinder}
\label{asymptotic-soliton}
Let $(M^n,g(t))$ be a
$SO(k)\times SO(n-k+1)$-symmetric, compact, simply-connected, factorwisely
non-self-similar $\kappa$-solution. Then exactly one of the
following alternatives occurs.
\begin{enumerate}
\item
The flow is an ancient oval. Its unique tangent flow at
$-\infty$ is either
\[
\mathbb R^k\times
S^{n-k}\bigl(\sqrt{2(n-k-1)|t|}\bigr),
\]
where $2\le n-k\le n-1$, or
\[
\mathbb R^{n-k+1}\times
S^{k-1}\bigl(\sqrt{2(k-2)|t|}\bigr),
\]
where $2\le k-1\le n-1$.

\item
After passing to the universal cover and interchanging the two
factors if necessary, the flow is a Riemannian product of a
Perelman's ancient oval and a shrinking round sphere. Its tangent
flow at $-\infty$ is
\[
\mathbb R\times
S^{k-1}\bigl(\sqrt{2(k-2)|t|}\bigr)
\times
S^{n-k}\bigl(\sqrt{2(n-k-1)|t|}\bigr),
\]
where $k-1\geq 2$ and $n-k\geq 2$. More precisely, up to interchanging the factors, the flow is either a product of a $k$-dimensional Perelman ancient oval and a shrinking round  $S^{n-k}$ or a product of an $(n-k+1)$-dimensional Perelman ancient oval and a shrinking round $S^{k-1}$.
\end{enumerate}
\end{lem}
\begin{remark}\label{tangent flow definitions}
Here, the tangent flow at $-\infty$ of $(S^n, g(t))$ can be formulated  in Perelman's sense in \cite[Section 11.2]{perelman2002entropy}. For example, in the situation of ancient oval case $(1)$, for a fixed point $q \in S^n$, we consider a sequence of times $t_i \to -\infty$ and a sequence of points $x_i \in S^n$ such that $\sup_i \ell(x_i,t_i) < \infty$, where $\ell$ denotes the reduced distance from $(q,0)$ (see \cite{perelman2002entropy} for the definition). If we parabolically dilate the manifold $(S^n,g(t))$ around the point $x_i$ by the factor $(-t_i)^{-\frac{1}{2}}$, then the rescaled manifolds converge either to a cylinder $\mathbb R^{k} \times {{S}}^{n-k}(\sqrt{2(n-k-1)|t|})$ or to a cylinder ${{S}}^{k-1}(\sqrt{2(k-2)|t|}) \times \mathbb{R}^{n-k+1}$. Since the convergence of any $\kappa$-solution is smooth, by \cite{CMZ-asymptotic-solitons}, the definition of the tangent flow at $-\infty$ in the weak sense (cf. \cite{bamler2020structure}) is equivalent to Perelman's definition. In addition, by \cite[Theorem 7.9]{fang2025strong}, the blowdown limit asymptotic cylinder of $\kappa$-solutions is strongly unique and independent of the choice of scaling sequences.
\end{remark}

\begin{proof}[Proof of Lemma \ref{blowdown is cylinder}]
Since our $SO(k)\times SO(n-k+1)$-symmetric, compact, simply-connected, factorwisely
non-self-similar $\kappa$-solution $(M^n, g(t))$ has nonnegative curvature operator,
by the classification results for compact, simply connected Riemannian manifolds with a nonnegative curvature operator in \cite[Theorem 7.34]{CLN} and \cite{BS-weakly-PIC2}, $M^n$ must be isometric to a product of the form 
$N_1 \times \ldots \times N_j$, where $N_1, \ldots, N_j$ are compact, simply connected, and irreducible, and each factor $N_1, \ldots, N_j$ is either diffeomorphic to a sphere, a K\"ahler manifold biholomorphic to complex projective space, or isometric to a compact symmetric space. By the $SO(k)\times SO(n-k+1)$ symmetry assumption, we have the following two possibilities: (i) $M^n$ either splits as $N_1\times N_2$, where $N_1$ is a $\ell$-dimensional $\kappa$-solution with $SO(k)$-symmetry, and $N_2$ is a $n-\ell$-dimensional $\kappa$-solution with $SO(n-k+1)$-symmetry, (ii) or $M^n$ has only one compact factor being a $SO(k)\times SO(n-k+1)$-symmetric, compact, simply-connected, 
non-self-similar $\kappa$-solution with positive curvature operator. 

For the first case in which $M^n$ splits as $N_1\times N_2$, since $\ell\geq k-1$ and $n-\ell \geq n-k$, it follows that $k-1\leq \ell\leq k$. If $\ell = k$, $N_2$ is an $(n-k)$-dimensional $\kappa$-solution with $SO(n-k+1)$-symmetry, it must be isometric to $(n-k)$-dimensional shrinking sphere. $N_1$ is a $SO(k)$-symmetric compact $k$-dimensional $\kappa$-solution which cannot be sphere by factorwisely non-self-similar property, and it  is uniformly PIC by this $SO(k)$-symmetry. Then by the classification of rotationally symmetric $\kappa$-solutions \cite[Theorem 1.3]{BDNS-compact}
, it must be isometric to $k$-dimensional Perelman's solution. In this case, its tangent flow at $-\infty$ is $\mathbb R \times {{S}}^{k-1}(\sqrt{2(k-2)|t|}) \times S^{n-k}(\sqrt{2(n-k-1)|t|})$. On the other hand, if $\ell = k-1$, $N_1$ is a $k-1$-dimensional $\kappa$-solution with $SO(k)$-symmetry, and it must be isometric to $k-1$-dimensional shrinking sphere. $N_2$ is a $SO(n-k+1)$-symmetric compact $(n-k+1)$-dimensional $\kappa$-solution which cannot be sphere by factorwisely non-self-similar property, and it is uniformly PIC by this $SO(n-k+1)$-symmetry. Then by the classification of rotationally symmetric $\kappa$-solutions \cite[Theorem 1.3]{BDNS-compact}, and it must be isometric to $(n-k+1)$-dimensional Perelman's solution. In this case, its tangent flow at $-\infty$ is $\mathbb R \times {{S}}^{k-1}(\sqrt{2(k-2)|t|}) \times S^{n-k}(\sqrt{2(n-k-1)|t|})$ as well.

Then we consider the second case where $M^n$ is a $SO(k)\times SO(n-k+1)$-symmetric, compact, non-self-similar $\kappa$-solution with positive curvature operator, so $M^n$ is diffemorphic to $S^n$. By the work of Perelman \cite{perelman2002entropy}, the rescaled manifolds converge in the Cheeger--Gromov sense to a smooth limit, which is a shrinking gradient Ricci soliton.

We claim that the limiting soliton is either a cylinder $\mathbb R^{k} \times S^{n-k}(\sqrt{2(n-k-1)|t|})$ or a cylinder ${{S}}^{k-1}(\sqrt{2(k-2)|t|}) \times \mathbb{R}^{n-k+1}$, so $M^n$ is a $SO(k)\times SO(n-k+1)$-symmetric ancient oval diffemorphic to $S^n$. To this end, we distinguish two cases:

\textit{Case 1:} If the limiting soliton is compact, then it must have constant sectional curvature $\frac{1}{2(n-1)}$. In particular, the sectional curvatures of $(S^n,g(t_i))$ lie in the interval $[\frac{1-\varepsilon_i}{(-2(n-1)t_i)},\frac{1+\varepsilon_i}{(-2(n-1)t_i)}]$, where $\varepsilon_i \to 0$ as $i \to \infty$. The result in \cite{BS} implies that $(S^n,g(t))$ has constant sectional curvature for each $t$. Thus, $(S^n,g(t))$ is a family of shrinking round spheres, contrary to our assumption of non-self-similarity.

\textit{Case 2:} If the limiting soliton is noncompact, then by the classification result for nonnegatively curved shrinking Ricci solitons \cite{munteanu2017positively} and the $SO(k)\times SO(n-k+1)$ symmetry assumption, it must split $k$ lines or $n-k+1$ lines. In the first case, the shrinking soliton is $\mathbb{R}^k\times N^{n-k}$, where $N^{n-k}$ is a shrinking Ricci soliton with a positive curvature operator. Using the $SO(k)\times SO(n-k+1)$ symmetry assumption again, it follows that $N^{n-k}$ must be an $(n-k)$-dimensional sphere with radius $\sqrt{2(n-k-1)}$, so the limiting soliton must be a cylinder $\mathbb R^{k} \times S^{n-k}(\sqrt{2(n-k-1)|t|})$ with $2\leq n-k\leq n-1$. Similarly, in the second case, the shrinking soliton is $\mathbb{R}^{n-k+1}\times{{S}}^{k-1}(\sqrt{2(k-2)|t|})$ with $2\leq k-1\leq n-1$.

Because the reduced-volume (entropy) values of these two cylinders are different (we note that when $n-k = k-1$, the two tangent flows are the same cylinder), the monotonicity formula implies that the aforementioned cylindrical type of the tangent flow at $-\infty$ of $(M^n, g(t))$ is unique up to diffeomorphism and independent of the choice of subsequence. By \cite[Theorem 7.9]{fang2025strong}, the (strong) uniqueness (subsequence limit independence) of the cylindrical tangent flow at $-\infty$ (in a fixed gauge) of $\kappa$-solutions holds; thus, the blowdown limit asymptotic cylinder is unique and independent of the choice of scaling sequences.
\end{proof}

Without loss of generality, in the following discussion, we assume that the tangent flow at $-\infty$ of the $SO(k)\times SO(n-k+1)$-symmetric, $n$-dimensional, non-self-similar compact $\kappa$-solution $(M, g(t))$ is $\mathbb R^{k} \times {{S}}^{n-k}(\sqrt{2(n-k-1)|t|})$. 

\begin{lem}
\label{diameter}
Given any sequence of times $t_i \to -\infty$, we have 
\[R_{\text{\rm max}}(t_i) \, \text{\rm diam}_{g(t_i)}(S^n,g(t_i))^2 \to \infty,\] 
where $R$ denotes the scalar curvature and $R_{\text{\rm max}}$ denotes the maximum of the scalar curvature over all points in space.
\end{lem}

\begin{proof}
This lemma follows directly from Lemma \ref{asymptotic-soliton}.
\end{proof}

By the neck stability theorem of Kleiner and Lott \cite{KL-singular} (see also \cite{ABDS}), there exists a point $q$ such that 
\begin{align}\label{eq-q-R}
    \sup_{t\to -\infty} (-t)R(q, t)\leq 100.
\end{align}
By Perelman's compactness theorem for asymptotic solitons and our assumption, we have the following proposition:
\begin{prop}[cf. {\cite[Proposition 3.2]{ABDS}}]\label{tangent-flow}
    If we parabolically dilate the flow around the point $(q, t)$ by the factor $(-t)^{-\frac{1}{2}}$, then the rescaled flow converges to a unique cylinder $\mathbb R^{k} \times S^{n-k}(\sqrt{2(n-k-1)|t|})$.
\end{prop}

Recall that $(M, g(t))$ is an ancient oval of the Ricci flow in the sense of Definition \ref{ancient oval defn} that is $SO(k)\times SO(n-k+1)$-symmetric, and $\mathbb{R}^k\times S^{n-k}(\sqrt{2(n-k-1)|t|})$ is its tangent flow at $-\infty$. We denote the fixed point of the $SO(k)$-action by $\mathcal O$. For each $t$, we denote by $F(z,t)$ the radius of the $S^{k-1}$ fiber at distance $z$ from $\mathcal O$, and by $G(z,t)$ the radius of the $S^{n-k}$ fiber at distance $z$ from $\mathcal O$. Then the metric can be written as 
\begin{align}\label{g-metric}
    g(t)&= dz\otimes dz + F^2(z,t) g_{S^{k-1}}  + G^2(z,t)g_{S^{n-k}},
\end{align}
where $z\in[0, z_{\max}]$, $g_{S^{k-1}}$ is the standard spherical metric on $S^{k-1}$ and $g_{S^{n-k}}$ is the standard spherical metric on $S^{n-k}$. In order for the Riemannian metric $g$ to close up smoothly at the singular orbits, $F$ and $G$ must satisfy the following conditions at $0$ and $z_{\max}$:\\
(i) $F(z,t)$ is a smooth function after an odd reflection at $z=0$, and $G(z,t)$ is a smooth function after an even reflection at $z=0$. They satisfy $F(0, t)=0$, $F_z(0, t)=1$, $G_z(0, t)=0$, and $G(0, t)>0$;\\
(ii) $F(z,t)$ is a smooth function after an even reflection at $z_{\max}$, and $G(z,t)$ is a smooth function after an odd reflection at $z_{\max}$. They satisfy $G_z(z_{\max}, t)=-1$, $F_{z}(z_{\max}, t)=0$, and $F(z_{\max}, t)>0$.\\
{\em By the above description, we only need to consider the uniqueness problem for $(F, G)$ defined on $[0, z_{\max}]$.}

Consider the orthonormal frame $e_0=\partial_z$, $\{e_i\}_{i=1}^{k-1}$ on $S^{k-1}$, and $\{e_a\}_{a=1}^{n-k}$ on $S^{n-k}$. In the ordered basis
\begin{equation}\label{curvature operator formula}
    \{\,e_0\wedge e_i\bigr\}_{i=1}^{k-1}
\;,\,\bigl\{\,e_0\wedge e_a\bigr\}_{a=1}^{n-k}
\;,\,\bigl\{\,e_i\wedge e_j\bigr\}_{1\le i<j\le k-1}
\;,\,\bigl\{\,e_i\wedge e_a\bigr\}_{i,a}
\;,\,\bigl\{\,e_a\wedge e_b\bigr\}_{1\le a<b\le n-k}
\end{equation}
the curvature operator $\mathcal{R}$ is represented by
\begin{equation}
\mathcal R
=
\begin{pmatrix}
-\dfrac{F_{zz}}{F}\,I_{\,k-1}
&0&0&0&0\\[6pt]
0&-\dfrac{G_{zz}}{G}\,I_{\,n-k}
&0&0&0\\[6pt]
0&0&\dfrac{1-F_z^2}{F^2}\,I_{\binom{k-1}{2}}
&0&0\\[6pt]
0&0&0&-\dfrac{F_z\,G_z}{F\,G}\,I_{(k-1)(n-k)}
&0\\[6pt]
0&0&0&0&\dfrac{1-G_z^2}{G^2}\,I_{\binom{n-k}{2}}
\end{pmatrix},
\end{equation}
where $I_{m}$ represents the $m$-dimensional identity matrix. Since the metric has a nonnegative curvature operator, we have 
\begin{align}\label{derivative sign}
    F_{zz}\leq 0,\quad G_{zz}\leq 0, \quad \text{and}\quad F_zG_z\leq 0.
\end{align}
By the smoothness conditions and \eqref{derivative sign}, we have
\begin{align}\label{gradient-sign}
    0\leq F_z\leq 1\quad \text{and} \quad -1\leq G_z\leq 0, \qquad \mbox{on }\,\, z \in [0, z_{\max}].
\end{align}
 Moreover, by symmetry, we have  
\begin{align}
    &\frac{G_{zz}}{G} = \frac{G_zF_z}{GF} \quad \text{at} \quad z=0,\\\
    &\frac{F_{zz}}{F} = \frac{G_zF_z}{GF} \quad \text{at} \quad z=z_{\max}.
\end{align}

\subsection{Blowdown analysis}\label{blowdown analysis section} 
The goal of this subsection is to classify all possible pointed blowdown limits of  $SO(k)\times SO(n-k+1)$-symmetric ancient ovals of Ricci flow as $t\to-\infty$. We first show that, for any sequence of points whose $SO(n-k+1)$-orbit radius $G$ remains comparable to the cylindrical scale $\sqrt{-t}$, the corresponding curvature-scale blowdown limits split off at least $k-1$ Euclidean directions.

\begin{lem}\label{k-1lines split} 
Let $(M, g(t))$ be the ancient oval of the Ricci flow described above. Consider a sequence of points $(x_i,t_i)$ such that $G(x_i, t_i)\geq \delta \sqrt{-2(n-k-1)t_i}$ for some $\delta>0$. For the rescaled flows 
\begin{align}
    (M,g_i(t);(x_i, t_i)) = (M, R(x_i,t_i)g(t_i+R(x_i,t_i)^{-1}t); (x_i, t_i)),
\end{align}
the limit $(M_\infty,g_\infty(t);(x_\infty, 0))$ of the rescaled flows contains at least $k-1$ straight lines.
\end{lem}

\begin{proof}
Let $q$ be the point such that \eqref{eq-q-R} holds, and let $t_i\to -\infty$ be an arbitrary sequence. By Proposition \ref{tangent-flow}, $(M,g_i(t);(q, t_i))$ converges to the cylinder $\mathbb R^{k} \times {{S}}^{n-k}(\sqrt{2(n-k-1)|t|})$, and hence
\begin{align}
    (-t_i)^{-\frac{1}{2}}F((-t_i)^{\frac{1}{2}}\xi, (-t_i)t)\to \xi
\end{align}
in the $C^\infty_{\text{loc}}$ sense with $(-t_i)^{\frac{1}{2}}\xi = z$. This implies that $F_z((-t_i)^{\frac{1}{2}}\xi, (-t_i)t)\to 1$ around $(q, t_i)$. Denoting the $z$-coordinate of $q$ by $z_q$, by the convexity in \eqref{derivative sign}, we have 
\begin{align}
    F_z(z_q, t_i)\leq F_z(z,t_i)\leq 1
\end{align}
for all $z\in [0,z_q]$. We find that $F_z((-t_i)^{\frac{1}{2}}\xi, (-t_i)t)\to 1$ around all points $(z, t_i)$ for $z\in[0,z_q]$.
On the other hand, we also have 
\begin{align}
    (-t_i)^{-\frac{1}{2}}G((-t_i)^{\frac{1}{2}}\xi, (-t_i)t)\to \sqrt{2(n-k-1)}
\end{align}
in the $C^\infty_{\text{loc}}$ sense. This implies that $G_z((-t_i)^{\frac{1}{2}}\xi, (-t_i)t)\to 0$ around $(q, t_i)$. Thus, by the convexity in \eqref{derivative sign}, we have 
\begin{align}
    G_z(z_q, t_i)\leq G_z(z,t_i)\leq 0
\end{align}
for all $z\in [0,z_q]$. We obtain that $G_z((-t_i)^{\frac{1}{2}}\xi, (-t_i)t)\to 0$ around all points $(z, t_i)$ for $z\in[0,z_q]$. Therefore, if we parabolically dilate the flow around the point $(z, t_i)$ by the factor $(-t_i)^{-\frac{1}{2}}$, the rescaled flows converge to a unique cylinder $\mathbb R^{k} \times S^{n-k}(\sqrt{2(n-k-1)|t|})$ for all $z\in[0,z_q]$. Since $t_i\to -\infty$ is an arbitrary sequence, the above arguments show that, due to convergence, we have 
\begin{align}
    c\leq (-t)R(z,t)\leq C
\end{align}
for all $z\in[0,z_q]$. This shows that the lemma holds for all $z\in[0,z_q]$. 

Next, we claim that there exists a uniform constant $A>0$ such that 
\begin{align}\label{curvature-lower-bound}
    R(z,t)\geq \frac{A}{-t}
\end{align}
for all $(z,t)$ such that $G(z,t)\geq \delta\sqrt{-2(n-k-1)t}$ and for sufficiently negative time $t$. We argue by contradiction. Suppose there exists a sequence of points $(z_j,t_j)$ with $t_j\to-\infty$ such that 
\begin{align}\label{tR-0}
   (-t_j) R(z_j,t_j)\to 0.
\end{align}
By assumption, we always have 
\begin{align}\label{G-cylinder-bounded}
c\sqrt{-t_j}\leq G(z_j,t_j) \leq C\sqrt{-t_j}.
\end{align}
Since the sectional curvature of the $SO(n-k+1)$-orbit $K_G$ is $\frac{1-G_z^2}{G^2}$, by convexity and our assumption, there exists a constant $c_0$ depending only on $\delta$ such that $G_z^2\leq 1-c_0$. Thus, by \eqref{G-cylinder-bounded} and the nonnegativity of curvature, we obtain
\begin{align}
    (-t_j) R(z_j,t_j)\geq (-t_j)K_G(z_j, t_j) \geq \frac{c_0}{C^2}>0,
\end{align}
which contradicts \eqref{tR-0}, thereby proving the claim. 

Lastly, we show that the lemma holds for all points satisfying the condition in the lemma with $z\geq z_q$. For all $z\leq L\sqrt{-t}$, if we parabolically dilate the flow around the point $(z, t)$ by the factor $(-t)^{\frac{1}{2}}$, any geodesic ball with diameter $D$ around $z$ is contained in the geodesic ball with diameter $D+L$ around $z = 0$. Thus, the limit of the rescaled flow is also the cylinder $\mathbb R^{k} \times S^{n-k}(\sqrt{2(n-k-1)|t|})$. For those $z$ satisfying $\frac{z}{\sqrt{-t}}\to \infty$, we first prove that 
\begin{align}\label{RF-infty}
    \liminf_{-t\to\infty} R(z,t)^{\frac{1}{2}}F(z,t) = \infty.
\end{align}
Suppose towards a contradiction that we can find a sequence $(z_j,t_j)$ with $-t_j\to\infty$ and $\frac{z_j}{\sqrt{-t_j}}\to \infty$ such that 
\begin{align}
    R(z_j,t_j)^{\frac{1}{2}}F(z_j,t_j) \leq C.
\end{align}
By \eqref{curvature-lower-bound}, we have
\begin{align}\label{rescaled-F-bound}
    (-t_j)^{-\frac{1}{2}}F(z_j,t_j)\leq \frac{C}{\sqrt{A}}.
\end{align}
Then for each $t_j$, we can find another $z_j'\leq z_j$ such that $z_j' = \frac{10C}{\sqrt{A}}\sqrt{-t_j}$. By \eqref{gradient-sign}, it always holds that $F(z_j',t_j)\leq F(z_j,t_j)$. It follows that 
\begin{align}
    (-t_j)^{-\frac{1}{2}}F(z_j',t_j)\leq \frac{C}{\sqrt{A}}.
\end{align}
Since $z_j' = \frac{10C}{A}\sqrt{-t_j}$, we recall that $(-t_i)^{\frac{1}{2}}\xi = z$, we have 
\begin{align}\label{F-convergence}
    (-t_j)^{-\frac{1}{2}}F((-t_j)^{\frac{1}{2}}\xi, (-t_j)t)\to F_\infty(\xi,\tau) = \xi
\end{align}
Substituting $z_j'$ into \eqref{F-convergence}, and take $t = -1$, we get 
\begin{align}
    (-t_j)^{-\frac{1}{2}}F(z_j', t_j)\to \xi_\infty = \frac{10C}{\sqrt{A}},
\end{align}
which contradicts \eqref{rescaled-F-bound}. This proves \eqref{RF-infty}. 

Now, fixing $t=t_j$, along an $SO(k)$-orbit, choose an orthonormal frame $e_1,\dots,e_{k-1}$ tangent to the ${{S}}^{k-1}$-factor. In the double warped product metric
\[
g=dz^2+F^2 g_{{{S}}^{k-1}}+G^2 g_{{{S}}^{n-k}},
\]
the second fundamental form of the $SO(k)$-orbit satisfies
\begin{align}\label{second fundamental=F_z/F}
\textrm{I\!I}(e_\alpha,e_\beta)=-\frac{F_z}{F}\delta_{\alpha\beta}\,\partial_z.
\end{align}
Hence,
\begin{align}
|\textrm{I\!I}|\leq C\frac{|F_z|}{F}.
\end{align}
For ancient ovals with a nonnegative curvature operator, one has the standard estimate $|F_z|\leq 1$. Therefore,
\[
|\textrm{I\!I}|
\leq
\frac{C}{F}.
\]
For any sequence of $(x_j, t_j)$ satisfying  $\frac{x_j}{\sqrt{-t_j}}\to \infty$, the rescaled flows $(M, g_j(t); (x_j,t_j))$, the rescaled second fundamental form satisfies
\begin{align}
    |\textrm{I\!I}|_{g_j(0)}=R(z_j,t_j)^{-1/2}|\textrm{I\!I}|_{g(t_j)}\leq\frac{C}{R(x_j,t_j)^{\frac{1}{2}}F(x_j,t_j)}.
\end{align}
Thus, by \eqref{RF-infty}, we have 
\[
|\textrm{I\!I}|_{g_i(0)}\to 0.
\]
Hence, the $SO(k)$-orbits become asymptotically totally geodesic.
Next, the intrinsic sectional curvature of the $SO(k)$-orbit is
\[
K_{\mathrm{orbit}}
=
\frac{1}{F^2}.
\]
After rescaling by $R(z_j,t_j)$, it becomes
\[
K_{\mathrm{orbit},j}
=
\frac{1}{R(x_j,t_j)F(x_j,t_j)^2}.
\]
Therefore, by \eqref{RF-infty}, we have 
\[
K_{\mathrm{orbit},j}\to 0.
\]
Hence, the rescaled $SO(k)$-orbits become asymptotically flat.
Moreover, their injectivity radii satisfy
\[
\operatorname{inj}_{g_j(0)}(x_j)
=
\pi R(x_j,t_j)^{\frac{1}{2}}F(x_j,t_j)
\to
\infty.
\]
Therefore, for every fixed $L<\infty$, the intrinsic ball of radius $L$ inside the orbit converges smoothly to a Euclidean ball in $\mathbb R^{k-1}$. Consequently, the orbit directions produce $k-1$ complete flat directions in the pointed limit, completing the proof.
\end{proof}

We then show that the points of maximal scalar curvature must lie in the tip region, while points away from the tips are asymptotically cylindrical.

\begin{lem}\label{Rmax-tip}
For any given $\delta>0$, let $x_t$ be points satisfying $R(x_t, t) = R_{\max}(t)$. Then we have $x_t\in \{x \mid G(x,t)<\delta\sqrt{-2(n-k-1)t}\}$ for all sufficiently negative time $t$.
\end{lem}

\begin{proof}
We argue by contradiction. Suppose there exists a sequence $-t_i\to\infty$ and points $x_{t_i}$ satisfying $R(x_{t_i}, t_i) = R_{\max}(t_i)$ such that 
\begin{equation}
    x_{t_i}\in \{x \mid G(x,t)\geq\delta\sqrt{-2(n-k-1)t}\}.
\end{equation}
We then have 
\begin{align}\label{G-bound-cyl}
   G(x_{t_i},t_i)\geq \delta\sqrt{-2(n-k-1)t_i}.
\end{align}
Since the ancient oval is a backward Type II solution, it follows that 
\begin{align}\label{tRtoinfinity}
    \lim_{i\to\infty}(-t_i)R(x_{t_i}, t_i)\to \infty.
\end{align}
Hence, by \eqref{G-bound-cyl} and \eqref{tRtoinfinity}, we obtain
\begin{align}\label{rescaled-G-infty}
    \lim_{i\to\infty}R(x_{t_i}, t_i)^{\frac{1}{2}} G(x_{t_i},t_i) \geq \lim_{i\to\infty} \delta\sqrt{-2(n-k-1)t_i}R(x_{t_i}, t_i)^{\frac{1}{2}}\to \infty.
\end{align}
We consider the rescaled Ricci flows $(M,g_i(t);(x_{t_i}, t_i))$ around $(x_{t_i}, t_i)$. By the compactness of $\kappa$-solutions, this sequence of Ricci flows converges to a limiting Ricci flow $(M_\infty,g_\infty(t);(x_\infty, 0))$ satisfying 
\begin{align}\label{R-infty-1}
    R(x_\infty, 0) = 1.
\end{align}
However, by the argument from \eqref{second fundamental=F_z/F} to the end of the proof of Lemma \ref{k-1lines split}, the limit $(M_\infty,g_\infty(t);(x_\infty, 0))$ contains at least $k-1$ straight lines. Moreover, since the second fundamental form of the $SO(n-k+1)$-orbit satisfies
\begin{align}
\textrm{I\!I}(e_\alpha,e_\beta)=-\frac{G_z}{G}\delta_{\alpha\beta}\,\partial_z,
\end{align}
and the intrinsic sectional curvature of the $SO(n-k+1)$-orbit is $K_{\mathrm{orbit}} = \frac{1}{G^2}$, and by \eqref{rescaled-G-infty}, we know that the rescaled $SO(n-k+1)$-orbits become flat after taking limit. Together with the above $k-1$ lines splitting, this implies that the limit $(M_\infty,g_\infty(t);(x_\infty, 0))$ must be flat, which contradicts \eqref{R-infty-1}. This completes the proof.
\end{proof}

\begin{defn}
   Let $\left(x_0, t_0\right)$ be a point in spacetime with $R\left(x_0, t_0\right)=r^{-2}$. We say that $\left(x_0, t_0\right)$ lies at the center of an evolving $\varepsilon$-cylinder if, after rescaling by the factor $r^{-2}$, the parabolic neighborhood $B_{g\left(t_0\right)}\left(x_0, \varepsilon^{-1} r\right) \times\left[t_0- \varepsilon^{-1} r^2, t_0\right]$ is $\varepsilon$-close in $C^{\left[\varepsilon^{-1}\right]}$ to a family of shrinking cylinders $\mathbb R^{k} \times {{S}}^{n-k}(\sqrt{2(n-k-1)(\frac{n-k}{2}-t)})$. 
\end{defn}

\begin{prop}\label{epsilon-bubble-sheet}
    Given $\varepsilon>0$ and $\delta>0$, we can find a time $t_0=t_0(\varepsilon, \delta)$ such that the following holds. Suppose $(x, t)$ is a point in spacetime such that $t \leq t_0$, and the radius of the sphere of symmetry through $(x, t)$ is at least $\delta \sqrt{-2(n-k-1) t}$, i.e., $G(x,t)\geq \delta \sqrt{-2(n-k-1) t}$. Then, $(x, t)$ lies at the center of an evolving $\varepsilon$-cylinder.
\end{prop}

\begin{proof}
By Lemma \ref{k-1lines split}, any limit flow of a sequence of rescaled Ricci flows contains $k-1$ straight lines. Thus, the curvature components $\frac{1-F_z^2}{F^2}$, $-\frac{F_{zz}}{F}$, and $-\frac{F_zG_z}{FG}$ vanish at the limit. We argue by contradiction. Suppose that there exists a sequence of points $(x_i,t_i)$ in spacetime with the following properties: 
\begin{itemize}
\item $t_i \to -\infty$.
\item The sphere of $SO(n-k+1)$-symmetry through $(x_i,t_i)$ has radius at least $\delta \sqrt{-2(n-k-1)t_i}$. 
\item The point $(x_i,t_i)$ does not lie at the center of an evolving $\varepsilon$-cylinder.
\end{itemize}

By assumption, the sphere of $SO(n-k+1)$-symmetry through $(x_i,t_i)$ has radius 
\begin{equation}
    r_i \geq \delta \sqrt{-2(n-k-1)t_i}.
\end{equation}
 At the point $(x_i,t_i)$, the second fundamental form of the $SO(n-k+1)$-orbit satisfies
\begin{align}
\textrm{I\!I}(e_\alpha,e_\beta)=-\frac{G_z}{G}\delta_{\alpha\beta}\,\partial_z,
\end{align}
and the intrinsic sectional curvature of the $SO(n-k+1)$-orbit is $K_{\mathrm{orbit}} = \frac{1}{G^2}$. Thus, by the Gauss formula, the sectional curvature of the plane tangent to the sphere of $SO(n-k+1)$-symmetry is bounded from above by $r_i^{-2}$. We denote the curvature operator component $-\frac{G_{zz}}{G}$ by $K_{\text{z}}$. 
We first prove that there exists a uniform constant $C>0$, such that $R(x,t)\leq \frac{C}{G(x,t)^2}$ for all $G(x,t)\geq \delta \sqrt{-2(n-k-1) t}$ and $t\leq t_0$. We argue by contradiction, suppose the claim fails, we can find a sequence $(x_j,t_j)\in \{(x,t)|G(x,t)\geq \delta \sqrt{-2(n-k-1) t}\}$ and $t_j\to -\infty$ such that 
\begin{align}\label{RG-infty}
    R(x_j,t_j)G(x_j,t_j)^2\to\infty.
\end{align}
By compactness, the sequence of rescaled Ricci flows $(M, g_j(t); (x_j, t_j))$ converges to a limit flow $(M_\infty, g_\infty; x_\infty)$. By Lemma \ref{k-1lines split}, the limit flow splits as 
$(\mathbb R^{k-1}\times N, g_{\mathbb R^{k-1}}+h; x_\infty)$, where $(N,h;x_\infty)$ is a rotationally symmetric noncompact $\kappa$-solution with nonnegative curvature operator. Then by the classification results in \cite{BN-noncompact}, $(N, h(s); x_\infty)$ must be either  the $(n-k+1)$-dimensional Bryant soliton or the  shrinking cylinder $\mathbb{R} \times S^{n-k}$. In both case, we have $G_\infty(x_\infty)\leq C$, which contradicts \eqref{RG-infty}. Consequently, $K_{\text{z}}$ at $(x_i,t_i)$ satisfies $K_{\text{z}}(x_i,t_i) \leq r_i^{-2} \leq \frac{1}{\delta^2 (-2(n-k-1)t_i)}$.

Since the point $(x_i,t_i)$ does not lie at the center of an evolving $\varepsilon$-cylinder, we must have $\liminf_{i \to \infty} R(x_i,t_i)^{-1} \, K_{\text{z}}(x_i,t_i) > 0$. Since $(-t_i) \, K_{\text{z}}(x_i,t_i) \leq \frac{1}{2(n-k-1)\delta^2}$, it follows that $\limsup_{i \to \infty} (-t_i) \, R(x_i,t_i) < \infty$. On the other hand, let $x_{t_i}$ be the point such that $R(x_{t_i}, t_i) = R_{\max}(t_i)$. By Lemma \ref{Rmax-tip}, $x_{t_i}\notin \{x \mid G(x,t)\geq \delta \sqrt{-2(n-k-1) t}\}$. Moreover, we have $(-t_i) \, R(x_{t_i}, t_i) \to \infty$. This gives $R(x_i,t_i)^{-1} \, R(x_{t_i}, t_i) \to \infty$. By Perelman's long-range curvature estimate, the distance of $(x_i,t_i)$ from $(x_{t_i},t_i)$ is bounded from below by $A_i \, R(x_i,t_i)^{-\frac{1}{2}}$, where $A_i \to \infty$. Hence, if we dilate the flow around the point $(x_i,t_i)$ by the factor $R(x_i,t_i)$ and pass to the limit, the limit contains a straight line in the $z$-direction. By the Cheeger--Gromov splitting theorem, the limit splits as a product. Therefore, the limit flow contains $k$ lines. This contradicts the fact that 
\begin{align}
\liminf_{i \to \infty} R(x_i,t_i)^{-1} \, K_{\text{z}}(x_i,t_i) > 0,
\end{align}
which yields a contradiction and completes the proof.
\end{proof}

We further show that the blowdown sequence at the tip points where  $\{G=0\}$ converges to the product of $\mathbb R^{k-1}$ with the Bryant soliton.

\begin{prop}\label{tip-asymptotics} 
If we dilate the flow around the points $(x_i, t_i)$ which satisfy $R(x_i,t_i)^{\frac{1}{2}}G(x_i,t_i)\to 0$ by the factor $R(x_i, t_i)$, then (after passing through a subsequence) the rescaled flows converge to $\mathbb R^{k-1}$ times the Bryant soliton.
\end{prop}

\begin{proof}
We first show that the lemma holds for the maximal scalar curvature points $(x_t, t)$. Let $\Phi(t):=|t|R_{\max}(t)$. By the Type II assumption,
\begin{align}
\Phi(t)\to\infty
\qquad\text{as }t\to-\infty.
\end{align}
Choose $A_j\to\infty$ and choose $t_j\in[-A_j,0)$ such that
\begin{align}
\Phi(t_j)=\max_{t\in[-A_j,0]}\Phi(t).
\end{align}
Then $t_j\to-\infty$ and
\begin{align}\label{type2-condition}
Q_j|t_j|=R_{\max}(t_j)|t_j|=\Phi(t_j)\to\infty.
\end{align}
Choose $x_{t_j}\in M$ such that
\begin{align}
R(x_{t_j},t_j)=R_{\max}(t_j)=Q_j.
\end{align}
Define $g_j(s):=Q_j\,g\!\left(t_j+\frac{s}{Q_j}\right)$.
Since the original flow exists for $t\in(-\infty,0)$, the rescaled flow exists for $s\in(-\infty,-Q_jt_j)$. Hence, by \eqref{type2-condition}, the rescaled time intervals exhaust $\mathbb R$.
For the rescaled scalar curvature,
\begin{align}
R_j(y,s):=R_{g_j}(y,s)
=\frac{1}{Q_j}R\!\left(y,t_j+\frac{s}{Q_j}\right).
\end{align}
At the base point,
\begin{align}
R_j(x_{t_j},0)=1.
\end{align}
Hamilton's trace Harnack inequality for ancient solutions with a nonnegative curvature operator gives
\begin{align}
\partial_t R\ge 0.
\end{align}
Therefore, $R_{\max}(t)$ is nondecreasing in $t$. Hence, for $s\leq 0$,
\begin{align}
R_j(y,s)
\leq \frac{R_{\max}\!\left(t_j+\frac{s}{Q_j}\right)}{Q_j}
\leq \frac{R_{\max}(t_j)}{Q_j}
=1.
\end{align}
For $s\geq 0$ fixed and $j$ sufficiently large,
\begin{align}
t_j+\frac{s}{Q_j}\in[t_j,0]\subset[-A_j,0].
\end{align}
Therefore, by the choice of $t_j$,
\begin{align}
\Phi\!\left(t_j+\frac{s}{Q_j}\right)
\le \Phi(t_j).
\end{align}
Thus,
\begin{align}
R_j(y,s)
\le
\frac{R_{\max}\!\left(t_j+\frac{s}{Q_j}\right)}{Q_j}
=
\frac{\Phi\!\left(t_j+\frac{s}{Q_j}\right)}
{\Phi(t_j)}
\cdot
\frac{|t_j|}{\left|t_j+\frac{s}{Q_j}\right|}
\le
\frac{|t_j|}{\left|t_j+\frac{s}{Q_j}\right|}.
\end{align}
Since $t_j<0$,
\begin{align}
\frac{|t_j|}
{\left|t_j+\frac{s}{Q_j}\right|}
=
\frac{|t_j|}{|t_j|-\frac{s}{Q_j}}
=
\frac{1}{1-\frac{s}{Q_j|t_j|}}.
\end{align}
Hence, for every fixed $S>0$, we have
\begin{align}
\sup_{M\times[-S,S]}R_j
\le
\frac{1}{1-\frac{S}{Q_j|t_j|}}
\to 1.
\end{align}
Because the curvature operator is nonnegative, $|\operatorname{Rm}_{g_j}|\le C(n)R_j$, so the curvatures of $g_j(s)$ are uniformly bounded on every compact time interval. Hamilton--Cheeger--Gromov compactness then provides a subsequence such that
\begin{align}
(M,g_j(s),(x_{t_j}, t_j))
\longrightarrow
(M_\infty,g_\infty(s),x_\infty)
\end{align}
smoothly on compact subsets for all $s\in\mathbb R$.
The limit is complete, eternal, $\kappa$-noncollapsed, has a bounded nonnegative curvature operator, and satisfies the condition $R_{g_\infty}(x_\infty,0)=1$.
Moreover, the previous curvature estimate gives
\begin{align}
R_{g_\infty}(y,s)\le 1
\qquad
\text{for all }(y,s)\in M_\infty\times\mathbb R.
\end{align}
Therefore, we obtain
\begin{align}
R_{g_\infty}(x_\infty,0)
=
\sup_{M_\infty\times\mathbb R}R_{g_\infty}.
\end{align}
Hamilton's trace Harnack inequality for the eternal limit implies that, for every vector field $V$,
\begin{align}
\partial_s R
+
2\langle\nabla R,V\rangle
+
2\operatorname{Ric}(V,V)
\ge 0.
\end{align}
Since $R$ attains a spacetime maximum at $(x_\infty,0)$, Hamilton's equality case for the Harnack inequality implies that the limit flow is a steady gradient Ricci soliton. Moreover, by \eqref{type2-condition}, $(x_t,t)$ satisfies the lower bound estimates on scalar curvature \eqref{curvature-lower-bound}. By the argument from \eqref{second fundamental=F_z/F} to the end of the proof of Lemma \ref{k-1lines split}, the rescaled $SO(k)$-orbits become asymptotically flat, which implies that $(M_\infty,g_\infty(s),x_\infty)$ contains at least $k-1$ straight lines. Thus, we have 
\begin{align}
    (M_\infty,g_\infty(s),x_\infty) = (N\times \mathbb R^{k-1}, h(s)+g_{\mathbb{R}^{k-1}}; x_\infty),
\end{align}
where $(N, h(s); x_\infty)$ is a rotationally symmetric steady gradient Ricci soliton, which is isometric to the Bryant soliton up to rescaling. Therefore, the lemma holds for the maximal scalar curvature points $(x_t, t)$.

Now, for the sequence of points $(x_i, t_i)$ satisfies
\begin{align}\label{scalar-G-0}
    R(x_i,t_i)^{\frac{1}{2}}G(x_i,t_i)\to 0.
\end{align}
We only need to prove that there exists a uniform constant $C>0$ such that  
\begin{align}\label{tip-maximal-close}
d^z_{R(x_{t_i}, t_i)g(t_i)}(x_i, x_{t_i})\leq C
\end{align}
for all sufficiently large $i$, where $d^z$ denotes the distance in the $z$-direction (by $SO(k)$-symmetry, we can always assume $p$ and $x_t$ have the same coordinate components on the $SO(k)$-orbit). Then, by compactness, the rescaled geodesic ball around $(x_i, t_i)$ must be contained in an enlarged rescaled geodesic ball around $x_{t_i}$, implying that the limiting flow of the sequence of rescaled Ricci flows around $(x_i, t_i)$ is isometric to the limiting flow of $(M, g_{x_{t_i}}(t), (x_{t_i},t_i))$ up to rescaling. Moreover, since the limiting flow of $(M, g_{x_{t_i}}(t), (x_{t_i},t_i))$ is $\mathbb R^{k-1}\times Bry_{n-k+1}$, then the scalar curvature is bounded from below in any bounded domain, which implies that the curvatures $R(x_{t_i}, t_i)$ and $R(x_{i}, t_i)$ are comparable. Thus, the lemma follows directly. We argue by contradiction: suppose that for $C_j\to\infty$, there exists a sequence of times $t_j\to-\infty$ such that 
\begin{align}
d^z_{R(x_{t_j}, t_j)g(t_j)}(x_j, x_{t_j}) > C_j.
\end{align}
We divide into two cases. In the first case, suppose that there exists a subsequence (we still denote by $(x_j, t_j)$), such that 
\begin{align}
    G(x_j, t_j)\leq G(x_{t_j}, t_j).
\end{align}
By convexity of $G$, we know the $z$-coordinate of $x_j$ is larger than $x_{t_j}$.
Then, by Lemma \ref{Rmax-tip}, 
\begin{align}
d^z_{R(x_{t_j}, t_j)g(t_j)}(\mathcal{O}, x_{t_j})\to\infty.
\end{align}
Thus, passing to the limit, the segment $\mathcal{O}x_{t_j}x_j$ will converge to a geodesic line in the $z$-direction in the limit space. Moreover, since $(x_{t_j}, t_j)$ is a point of maximal scalar curvature, we have 
\begin{align}
(-t_j)R(x_{t_j}, t_j)\to \infty.
\end{align}
By the argument from \eqref{second fundamental=F_z/F} to the end of the proof of Lemma \ref{k-1lines split}, the rescaled $SO(k)$-orbits become asymptotically flat. Thus, the limit of the rescaled Ricci flow $(M,g_j(t),(x_{t_j}, t_j))$ contains at least $k$ straight lines, which contradicts the fact that the limit is isometric to $\mathbb R^{k-1}$ times the Bryant soliton. In the second case, we may assume for all large $j$, we have 
\begin{align}
    G(x_j, t_j)\geq G(x_{t_j}, t_j).
\end{align}
In this case, we may find another sequence of $C_i'\to\infty$, such that 
\begin{align}
d^z_{R(x_{j}, t_j)g(t_j)}(x_j, x_{t_j}) > C_j'.
\end{align}
Since $(x_j, t_j)$ satisfies $R(x_j,t_j)^{\frac{1}{2}}G(x_j,t_j)\to 0$, we also have 
\begin{align}\label{xj-O-infty}
d^z_{R(x_{j}, t_j)g(t_j)}(\mathcal{O}, x_{t_j})\to\infty.
\end{align}
Passing to the limit, the segment $\mathcal{O}x_jx_{t_j}$ will converge to a geodesic line in the $z$-direction in the limit space. By Cheeger-Gromoll splitting theorem, the limiting space splits as  
\begin{align}
    \left(M_\infty, g_\infty(s); x_\infty\right) = \left(\mathbb R\times N', d\xi^2+h'(s);x_\infty\right)
\end{align}
with
\begin{align}\label{scalar-intermediate-1}
    R_{g_\infty}(x_\infty,0) = 1.
\end{align}
We recall that the metric is written as \eqref{g-metric}, function $F$ and $G$ depend only on $z$ in space. We perform the rescaling $z = z_j+ R(x_j,t_j)^{-\frac{1}{2}}\xi$, where $z_j$ is the $z$-coordinate of $x_j$, and $s = R(x_j,t_j)(t-t_j)$. Then we define the rescaled function 
\begin{align}
    F_j(\xi,s) = R(x_j,t_j)^{\frac{1}{2}}F(z_j+R(x_j,t_j)^{-\frac{1}{2}}\xi, t_j+ R(x_j,t_j)^{-1}s)
\end{align}
and 
\begin{align}
    G_j(\xi,s) = R(x_j,t_j)^{\frac{1}{2}}G(z_j+R(x_j,t_j)^{-\frac{1}{2}}\xi, t_j+ R(x_j,t_j)^{-1}s)
\end{align}
Thus, by convergence, we have $F_j(\xi,s)\to F_\infty(\xi,s)$ and $G_j(\xi,s)\to G_\infty(\xi,s)$. By the flatness of curvature in $\xi$ direction, we have $-\frac{G_{\infty,\xi\xi}}{G_{\infty}} = 0$ and $-\frac{F_{\infty,\xi\xi}}{F_{\infty}} = 0$, it follows that $G_\infty$ and $F_\infty$ are linear functions. Moreover, by \eqref{scalar-G-0}, $G_\infty = 0$ at $x_\infty$. Hence, by the completeness of the limit metric, $G_\infty = \xi$. 
For the metric on $S^{k-1}$ fiber in the limit metric $g_\infty$, it either blows up as a flat metric in $\mathbb R^{k-1}$, or converges to some  $F_{\infty}^2(\xi,t) g_{S^{k-1}}$. we can exclude the first case by \eqref{scalar-intermediate-1}. Thus the limit metric can be written as 
\begin{align}
    \left(M_\infty, g_\infty(t) = d\xi^2 + F_{\infty}^2(\xi,t) g_{S^{k-1}} + \xi^2 g_{S^{n-k}}; x_\infty\right).
\end{align}
By the completeness of the metric and $F_\infty$ is a linear function, $F_\infty$ must be some constant $c$ in space, that is, 
\begin{align}\label{F-cylinder}
    R(x_j,t_j)^{\frac{1}{2}}F(x_j,t_j)\to c. 
\end{align}
Now suppose that there exists a uniform constant $A>0$ such that \eqref{curvature-lower-bound} holds, by the argument from \eqref{second fundamental=F_z/F} to the end of the proof of Lemma \ref{k-1lines split}, the rescaled $SO(k)$-orbits become asymptotically flat, which contradicts \eqref{F-cylinder}. Thus, we may assume that 
\begin{align}
    (-t_j)R(x_j,t_j)\to 0.
\end{align}
By \eqref{xj-O-infty}, We can find a large $D>0$, such that for any $x\in B_{R(\mathcal{O},t_j)g(t_j)}(\mathcal{O},D)$ and any $y\in B_{R(x_j,t_j)g(t_j)}(x_j,D)$, the $z$-coordinate of $y$ is larger than $x$. By the convexity of $G$, we have $G(x, t_j)\geq G(y,t_j)$, while $\mathcal{O}$ lies in the cylindrical region, by Proposition \ref{epsilon-bubble-sheet}, $R(\mathcal{O}, t_j)\geq \frac{c}{-t_j}\geq R(x_j,t_j)$. Thus, we have 
\begin{align}
    R(\mathcal{O},t_j)^{\frac{1}{2}}G(x,t_j)\geq R(x_j,t_j)^{\frac{1}{2}}G(y,t_j).
\end{align}
Since $G_\infty = \xi$ is an unbounded function, $R(\mathcal{O},t_j)^{\frac{1}{2}}G(x,t_j)$ cannot converge to a constant uniformly on $B_{R(\mathcal{O},t_j)g(t_j)}(\mathcal{O},D)$, which contradicts the limit flow at $\mathcal{O}$ being a cylinder. This completes the proof of  this proposition.

\end{proof}

\begin{prop}\label{all-sequences-asymptotics}
Consider a sequence of times $t_i \rightarrow-\infty$ and an arbitrary sequence of points $x_i$ on an $SO(k)\times SO(n-k+1)$-symmetric ancient oval $(M^{n}, g(t))$. If we dilate the flow around the point $\left(x_i, t_i\right)$ by the scalar curvature $R\left(x_i, t_i\right)$, then (after passing through a subsequence) the rescaled flows converge either to a family of shrinking cylinders $\mathbb R^{k} \times S^{n-k}$ or to $\mathbb{R}^{k-1}$ times the Bryant soliton.
\end{prop}

\begin{proof}
By Perelman's compactness theorem for $\kappa$-solutions, we know that the sequence of rescaled flows around $(x_i, t_i)$ subsequentially converges to a limit $\kappa$-solution $(M_\infty, g_\infty(t))$. We divide this into two cases. First, assume there exists a constant $\delta>0$ such that the radii of the spheres of $SO(n-k+1)$-symmetry through $(x_i, t_i)$ are all at least $\delta \sqrt{-2(n-k-1) t}$, i.e., $G(x_i, t_i)\geq \delta \sqrt{-2(n-k-1) t}$. In this case, by Proposition \ref{epsilon-bubble-sheet}, the limit $\kappa$-solution $(M_\infty, g_\infty(t))$ must be the shrinking cylinder $\mathbb R^{k} \times S^{n-k}(\sqrt{2(n-k-1)|t|})$. Otherwise, we assume that 
\begin{align}\label{type2-sequence}
\frac{G(x_i, t_i)}{\sqrt{-t_i}}\to 0.
\end{align}
We divide the argument into two cases. In  the first case if this sequence of points $(x_i, t_i)$ satisfies $R(x_i,t_i)^{\frac{1}{2}}G(x_i,t_i)\to 0$, by Proposition \ref{tip-asymptotics}, the convergent limit is $\mathbb{R}^{k-1}$ times the $(n-k+1)$-dimensional Bryant soliton. In the second case, there exists a constant $c>0$, such that
\begin{align}
    R(x_i,t_i)^{\frac{1}{2}}G(x_i,t_i)\geq c,
\end{align}
combining this with \eqref{type2-sequence}, it holds
\begin{align}
    (-t_i)R(x_i,t_i)\to \infty,
\end{align}
which means $(x_i,t_i)$ satisfies the lower bound estimates on scalar curvature \eqref{curvature-lower-bound}. By the argument from \eqref{second fundamental=F_z/F} to the end of the proof of Lemma \ref{k-1lines split}, the rescaled $SO(k)$-orbits become asymptotically flat, which implies that the limiting flow $(M_\infty,g_\infty(s),x_\infty)$ contains at least $k-1$ straight lines. Thus, we have 
\begin{align}
    (M_\infty,g_\infty(s),x_\infty) = (N\times \mathbb R^{k-1}, h(s)+g_{\mathbb{R}^{k-1}}; x_\infty),
\end{align}
where $(N, h(s); x_\infty)$ is a rotationally symmetric noncompact $\kappa$-solution with nonnegative curvature operator and by symmetry, it is uniformly PIC. Then by the classification results in \cite{BN-noncompact}, $(N, h(s); x_\infty)$ must be either  the $(n-k+1)$-dimensional Bryant soliton or the  shrinking cylinder $\mathbb{R} \times S^{n-k}$. This completes the proof of  this proposition.
\end{proof}
Now, we conclude the proof of Theorem \ref{remark tangent flow} (Geometry of ancient ovals and compact non-self-similar $\kappa$-solutions) below.
\begin{proof}[Proof of Theorem \ref{remark tangent flow}]
    Now, the proof of Theorem \ref{remark tangent flow} follows from Lemma \ref{lemma-sphere}, Lemma \ref{4d tangent flow}, Lemma \ref{blowdown is cylinder}, and Proposition \ref{all-sequences-asymptotics}. 
\end{proof}

\begin{cor}[cf. {\cite[Corollary 2.5]{ABDS}}]
\label{scalar.curvature.at.tips.1}
Let $R_{\text{\rm tip}}(t)$ denote the scalar curvature at the tip region. Then, $\frac{d}{dt} R_{\text{\rm tip}}(t) \leq o(1) \, R_{\text{\rm tip}}(t)^2$.
\end{cor}

\begin{proof}
This is a direct consequence of Proposition \ref{tip-asymptotics}.
\end{proof}

\begin{cor}[cf. {\cite[Corollary 2.6]{ABDS}}]
\label{scalar.curvature.at.tips.2} 
We have $(-t) \, R_{\text{\rm tip}}(t) \to \infty$ as $t \to -\infty$. 
\end{cor}

\begin{proof}
This follows by integrating the differential inequality in Corollary \ref{scalar.curvature.at.tips.1}.
\end{proof}

\section{Evolution equations and barrier}\label{Evolution equations and barrier}
{In this section, we derive the evolution equations of profile functions, their renormalized profile functions and inverse profile function, and we also recall the barrier constructed in \cite{Bre-3d-noncpt}.}
\subsection{Evolution equation of profile functions $F(z, t)$ and $G(z, t)$}
Recall that the metric can be written as 
\begin{align}
    g(t)&= dz\otimes dz + F^2(z,t) g_{S^{k-1}}  + G^2(z,t)g_{S^{n-k}}.
\end{align}
Since metric $g(t)$ evolves under Ricci flow, the functions  $F(z, t)$ and $G(z, t)$ satisfy the following equations. 
\begin{prop} Profile functions $F$ and $G$ satisfy the following evolution equations.
\begin{align}\label{F-equation}
\begin{aligned}
    F_t &= F_{zz} - (k-2)\frac{1-F_z^2}{F} + (n-k)\frac{F_zG_z}{G} \\
    &- F_z \int_0^{z(t)} (k-1)\frac{F_{zz}}{F}(x,t)+(n-k)\frac{G_{zz}}{G}(x,t) dx
\end{aligned}
\end{align}
and
\begin{align}\label{G-equation}
\begin{aligned}
    G_t &= G_{zz} - (n-k-1)\frac{1-G_z^2}{G} + (k-1)\frac{F_zG_z}{F}\\
    &- G_z \int_0^{z(t)} (k-1)\frac{F_{zz}}{F}(x,t)+(n-k)\frac{G_{zz}}{G}(x,t) dx. 
\end{aligned}
\end{align}
\end{prop}
\begin{proof}
    
By \eqref{curvature operator formula}, we have
\begin{align}\label{Ricci}
    &{\rm Ric}_{zz} = -(k-1)\frac{F_{zz}}{F}-(n-k)\frac{G_{zz}}{G},\notag\\
    &{\rm Ric}_{S^{k-1}} =\biggl(\frac{(k-2)}{F^2}-\frac{F_{zz}}{F}-(k-2)\frac{F_z^2}{F^2}-(n-k)\frac{F_z G_z}{FG}\biggr)\;g_{S^{k-1}},\\
    &{\rm Ric}_{S^{n-k}} 
=\biggl(\frac{(n-k-1)}{G^2}-\frac{G_{zz}}{G}-(n-k-1)\frac{G_z^2}{G^2}-(k-1)\,\frac{F_zG_z}{FG}\biggr)\;g_{S^{n-k}}.\notag
\end{align}
On the other hand, we also have 
\begin{align*}
    \frac{\partial g}{\partial t} = 2\frac{dz'}{dz} dz\otimes dz + 2 \frac{F'}{F} F^2g_{S^{k-1}}+ 2 \frac{G'}{G} G^2g_{S^{n-k}}.
\end{align*}
By Ricci flow equation we have 
\begin{align*}
    &\frac{dz'}{dz} = (k-1)\frac{F_{zz}}{F}+(n-k)\frac{G_{zz}}{G},\\
    &\frac{F'}{F} = -\frac{(k-2)}{F^2}+\frac{F_{zz}}{F}+(k-2)\frac{F_z^2}{F^2}+(n-k)\frac{F_z G_z}{FG},\\
    &\frac{G'}{G} = -\frac{(n-k-1)}{G^2}+\frac{G_{zz}}{G}+(n-k-1)\frac{G_z^2}{G^2}+(k-1)\frac{F_z G_z}{FG}.
\end{align*}
Therefore
\begin{align*}
   & z' = \int_0^{z(t)} (k-1)\frac{F_{zz}}{F}(x, t)+(n-k)\frac{G_{zz}}{G}(x, t) dx.\\
    &F' = F_{zz}+(n-k)\frac{F_zG_z}{G}-\frac{(k-2)}{F}+(k-2)\frac{F_z^2}{F^2},\\
    &{G'}= G_{zz}-\frac{(n-k-1)(1-G_z^2)}{G}+(k-1)\frac{F_z G_z}{F}.
\end{align*}
Furthermore, by the chain rule, we get $F' = F_zz'+F_t$ and $G' = G_zz'+G_t$. Combining  the above equations, we get the desired evolution equations for $F$ and $G$. 
\end{proof}

We also have the following pointwise derivative estimates for $G$
\begin{prop}\label{Gmorder estimates}
For every $m \geq 0$, the   following  pointwise estimate for derivatives of $G$ holds  
$$
G(z, t)^m|\partial_z^{m+1} G(z, t)| \leq C(m)\left(1+G(z, t)\left|G_{z z}(z, t)\right|\right)^m. 
$$
\end{prop}

\begin{proof}
We argue by induction on $m$. Since we have $-1\leq G_z\leq 0$
at each point in space-time, the assertion holds for $m=0$.
Moreover, the assertion clearly holds for $m=1$.

Suppose now that $m \geq 2$, and the assertion holds for all integers less than $m$. Using the standard formula for the scalar curvature of a double warped product, we obtain 

\begin{align}
\begin{aligned}
    R=&-2(k-1)\frac{F_{zz}}{F} + (k-1)(k-2)\frac{1-F_z^2}{F^2} - 2(k-1)(n-k)\frac{F_zG_z}{FG}\\
    &-2(n-k)\frac{G_{zz}}{G} + (n-k-1)(n-k)\frac{1-G_z^2}{G^2}.
\end{aligned}
\end{align}
By Proposition \ref{all-sequences-asymptotics}, any sequence of rescaled Ricci flows converges either to the shrinking cylinder, or the Euclidean space times Bryant soliton. In both cases, the curvature operator component related to $F$ {tends to zero when   $t\to -\infty$ }. Thus we have 
\begin{align}
\begin{aligned}
&\quad-2(k-1)\frac{F_{zz}}{F} + (k-1)(k-2)\frac{1-F_z^2}{F^2} - 2(k-1)(n-k)\frac{F_zG_z}{FG}\\
&\leq \frac{1}{100}\left(-2(n-k)\frac{G_{zz}}{G} + (n-k-1)(n-k)\frac{1-G_z^2}{G^2}\right)
\end{aligned}
\end{align}
for all sufficiently negative time. This implies 
\begin{align}
R\leq C\left(-\frac{G_{zz}}{G} + \frac{1-G_z^2}{G^2}\right).
\end{align}
Differentiating ${\rm Rm}_1 = -\frac{G_{zz}}{G}$ with respect to $z$ gives
\begin{align}\begin{aligned}
& \partial_z^{m-1} {\rm Rm}_1 +  G^{-1} \partial_z^{m+1} G \\
& =\sum_{k=0}^{m+1} \sum_{\substack{i_1 \geq 0, \ldots, i_k \geq 0 \\
i_1+\ldots+i_k \leq m-1}} c_{i_1 \ldots i_k} G^{i_1+\ldots+i_k-m-1} \partial_z^{i_1+1} G \cdots \partial_z^{i_k+1} G .
\end{aligned}
\end{align}
Using the induction hypothesis, we obtain
\begin{align}
\left|\partial_z^{m-1} {\rm Rm}_1 +  G^{-1} \partial_z^{m+1} G\right| \leq C(m) G^{-m-1}\left(1+G\left|G_{z z}\right|\right)^{m-1} .
\end{align}
On the other hand, Perelman's pointwise curvature derivative estimate for curvature operator (cf. \cite{CMZ23}) implies
\begin{align}
\left|\partial_z^{m-1} {\rm Rm_1}\right| \leq C(m) R^{\frac{m+1}{2}} \leq C(m) G^{-m-1}\left(1+G\left|G_{z z}\right|\right)^{\frac{m+1}{2}} .
\end{align}
Putting these facts together, we conclude that
\begin{align}
\left|G^{-1} \partial_z^{m+1} G\right| \leq C(m) G^{-m-1}\left(1+G\left|G_{z z}\right|\right)^m .
\end{align}
This completes the proof of Proposition \ref{Gmorder estimates}.
\end{proof}

\subsection{Evolution equation of renormalized profile functions $h=f_{\xi}$ and $g$}
By Proposition \ref{epsilon-bubble-sheet}, we know the Type-I blowdown around $SO(k)$-action fixed point $\mathcal{O}$ is the shrinking cylinders $\mathbb R^{k} \times S^{n-k}(\sqrt{2(n-k-1)|t|})$, we now consider the evolution of the renormalized profile functions
\begin{equation}\label{f-g}
    f(\xi, \tau) = e^{\frac{\tau}{2}}F(e^{-\frac{\tau}{2}}\xi, -e^{-\tau})-\xi,\, \quad
    g(\xi, \tau) = e^{\frac{\tau}{2}}G(e^{-\frac{\tau}{2}}\xi, -e^{-\tau})-\sqrt{2(n-k-1)},
\end{equation}
where $z = e^{-\frac{\tau}{2}}\xi$ and $t = -e^{-\tau}$.  
Proposition \ref{epsilon-bubble-sheet} implies that, as $\tau\to-\infty$, the functions $f(\xi, \tau)$ and $g(\xi, \tau)$ converge to $0$ in $C_{loc}^\infty$;  in particular $f(\xi, \tau)/\xi$ converges to $0$ in $C_{loc}^\infty$. The properties of $F, G$ at $z=0$ described in the setup section imply that 
\begin{equation}\label{fgx=0}
    f(0, \tau)=0, \qquad f_{\xi}(0, \tau)=g_{\xi}(0, \tau)=0.
\end{equation}
By the boundedness of the curvature operator
\begin{equation}
    |f_{\xi\xi}(\xi, \tau)|\leq C\, \big (f(\xi, \tau)+\xi \big )
\end{equation}
in $\xi\leq L$,
we have
\begin{equation}\label{fxx=0}
    f_{\xi\xi}(0, \tau)=0. 
\end{equation}
Furthermore, we  have    $f_{\xi}+1=F_z$ and  $f_{\xi\xi}=e^{-\tau/2}F_{zz}$. Hence,     since $F_z$ is decreasing we have     $f_{\xi}$ is decreasing,  and since $f_\xi(0,\tau)=0$, we have  $f_{\xi}\leq 0$. Also since $f(0,\tau)=0$, we also get $f\leq 0$. {We summarize these properties below for future reference} 
\begin{equation}f_{\xi\xi }(\xi,\tau) \leq 0, \qquad f_\xi(\xi,\tau) \leq 0, \qquad f (\xi, \tau) \leq 0.\end{equation}
\begin{prop}
The rescaled profiles $f(\xi,\tau)$ and $g(\xi,\tau)$ satisfy the following equations
\begin{align}\label{f-equation}
    f_\tau &=f_{\xi\xi} -\frac{1}{2}\xi f_\xi +\frac{1}{2}f \\
    &-(k-1)(f_\xi(\xi,\tau) + 1)\int_0^\xi \frac{f_{\eta\eta}(\eta,\tau)}{f(\eta,\tau)+\eta} d\eta+(k-2)\,\frac{2f_{\xi} + f^2_{\xi}}{f + \xi}\notag\\ 
    &-(n-k)(f_\xi(\xi,\tau) + 1)\int_0^\xi \frac{g_{\eta}^2(\eta,\tau)}{(g(\eta,\tau) + \sqrt{2(n-k-1)})^2} d\eta\notag
\end{align}
and
\begin{align}\label{g-equation}
    g_\tau&=  g_{\xi\xi}- \frac{1}{2}\xi g_\xi \\
    & -   \frac{(n-k-1)}{g(\xi,\tau) + \sqrt{2(n-k-1)})} +\frac{1}{2}(g(\xi,\tau) + \sqrt{2(n-k-1)}) \notag\\
    & - \frac{g_\xi^2(\xi,\tau)}{g(\xi,\tau) + \sqrt{2(n-k-1)})} -(n-k)g_{\xi}(\xi,\tau)\int_0^\xi \frac{g_{\eta}^2(\eta,\tau)}{(g(\eta,\tau) + \sqrt{2(n-k-1)})^{2}}  \notag\\
    & + (k-1)g_\xi(\xi,\tau)\left( \frac{1}{\xi} - \frac{f(\xi,\tau)}{\xi(f(\xi,\tau)+\xi)} -\int_{0}^{\xi}
     \frac{f_{\eta}(\eta,\tau)(f_{\eta}(\eta,\tau) +1) }{(f(\eta,\tau)  + \eta)^2} d\eta\right).\notag
\end{align}   

\end{prop}

\begin{proof}
Equations \eqref{F-equation} and \eqref{G-equation} 
{expressed in terms of the rescaled profiles become}
\begin{align} 
f_{\tau} 
&=f_{\xi\xi}- \frac{\xi}{2}\,f_{\xi}+\frac12\,f 
  + (k-2)\,\frac{2f_{\xi} + f^2_{\xi}}{f + \xi}- (k-1)(f_{\xi} + 1)\int_{0}^{\xi}
\frac{f_{\eta\eta}}{f + \eta}
    d\eta
   \notag\\
&
 + (n-k)(f_{\xi} + 1)\Big[\frac{g_{\xi}}{g + \sqrt{2(n-k-1)}} - \int_{0}^{\xi}
   \frac{g_{\eta\eta}}{g + \sqrt{2(n-k-1)}}
 d\eta\Big], \notag
\end{align}
and
\begin{align}
g_{\tau}
&= g_{\xi\xi}- \frac{\xi}{2}\,g_{\xi}+\frac12 (g+\sqrt{2(n-k-1)})   
  - \frac{(n-k-1)(1 - g_{\xi}^2)}{g + \sqrt{2(n-k-1)}}
 \notag\\
& - (n-k)g_{\xi}\int_{0}^{\xi}
   \frac{g_{\eta\eta}}{g + \sqrt{2(n-k-1)}}  
    d\eta
  +(k-1)g_{\xi} \Bigr[\frac{f_{\xi} + 1}{f + \xi}- \int_{0}^{\xi}
     \frac{f_{\eta\eta}}{f + \eta} \Bigr]d\eta.\notag\\
\end{align} 
Using $g_{\xi}(0, \tau)=0$ $f_{\xi}(0)=f_{\xi\xi}(0)=0$ and integrating by parts as follows 
\begin{equation}
    \frac{g_{\xi}}{g + \sqrt{2(n-k-1)}} - \int_{0}^{\xi}
   \frac{g_{\eta\eta}}{g + \sqrt{2(n-k-1)}}
 d\eta=\int_0^\xi \frac{g_{\eta}^2(\eta,\tau)}{(g(\eta,\tau) + \sqrt{2(n-k-1)})^2} d\eta\notag
\end{equation}

\begin{align}
\begin{aligned}
    &\quad(k-1)g_\xi\left[\frac{1}{\xi} + \left(\frac{1}{f+\xi}-\frac{1}{\xi}\right) + \frac{f_\xi}{f+\xi}-\int_{0}^{\xi}
     \frac{f_{\eta\eta}}{f + \eta} d\eta\right]\\
&= (k-1)g_\xi\left[\frac{1}{\xi} - \frac{f}{\xi(f+\xi)} -\int_{0}^{\xi}
     \frac{f_{\eta}(f_{\eta}+1)}{(f + \eta)^2} d\eta\right],
\end{aligned}
\end{align}
leads to  the desired  evolution equations  for $f$ and $g$. 
\end{proof}

Next, we decompose the evolution equations into the linear part and the error part.

\begin{prop}[evolution expansion for $(f, g)$]\label{evolution of (f, g)}
For the evolution equations of $f$ and $g$, we have
\begin{align}\label{f-g-expansion}
   f_\tau =\mathcal{L}_0 f+E_0\quad
    g_\tau =\mathcal{L}_1g+E_1
\end{align}   
where $\mathcal{L}_0$ is defined by
\begin{equation}
    \mathcal{L}_0 f=f_{\xi\xi} - \frac{1}{2}\xi f_\xi +\frac{1}{2}f+\frac{k-3}{\xi}f_{\xi}-(k-1)\int_{0}^{\xi}\frac{f_{\eta}}{\eta^2}d\eta,
\end{equation}
and $\mathcal{L}_1$ is the radial Ornstein-Uhlenbeck operator
\begin{equation}\label{L2definition}
   \mathcal{L}_1g= g_{\xi\xi}+\frac{k-1}{\xi}g_{\xi} - \frac{1}{2}\xi g_\xi + g,
\end{equation}
with errors 
\begin{align}\label{f-error-expansion}
    E_0=   
    &-(k-1)\int_0^\xi {f_{\eta\eta}(\eta,\tau)}\Big[\frac{1}{f(\eta,\tau)+\eta}-\frac{1}{\eta}\Big] d\eta+(2k-4){f_{\xi}(\xi,\tau)}\Big[\frac{1}{f(\xi,\tau)+\xi}-\frac{1}{\xi}\Big]  \\
    &-\frac{f^2_{\xi}}{f + \xi}-(k-1)f_{\xi}\int_0^\xi \frac{f_{\eta}(\eta,\tau)(f_{\eta}(\eta,\tau)+1)}{(f(\eta,\tau)+\eta)^2 }d\eta \notag\\ 
    &-(n-k)(f_\xi(\xi,\tau) + 1)\int_0^\xi \frac{g_{\eta}^2(\eta,\tau)}{(g(\eta,\tau) + \sqrt{2(n-k-1)})^2} d\eta, \notag
\end{align}
\begin{align}\label{g-error-expansion}
   E_1&=-g(\xi, \tau) -   \frac{(n-k-1)}{g(\xi,\tau) + \sqrt{2(n-k-1)})} +\frac{1}{2}(g(\xi,\tau) + \sqrt{2(n-k-1)}) \\ &- \frac{g_\xi^2(\xi,\tau)}{g(\xi,\tau) + \sqrt{2(n-k-1)})}-(n-k)g_{\xi}(\xi,\tau)\int_0^\xi \frac{g_{\eta}^2(\eta,\tau)}{(g(\eta,\tau) + \sqrt{2(n-k-1)})^{2}}  \notag\\
    & -(k-1)g_\xi(\xi,\tau)\left( \frac{f(\xi,\tau)}{\xi(f(\xi,\tau)+\xi)} +\int_{0}^{\xi}
     \frac{f_{\eta}(\eta,\tau)(f_{\eta}(\eta,\tau) +1) }{(f(\eta,\tau)  + \eta)^2} d\eta\right).\notag
\end{align}
\end{prop}
\begin{proof}
{Let us first compute the evolution equation for $f$. We express equation \eqref{f-equation} as follows, after rearranging terms }
    \begin{align}
    f_\tau &=f_{\xi\xi} -\frac{1}{2}\xi f_\xi +\frac{1}{2}f-(k-1)\int_0^\xi f_{\eta\eta}\Big[\frac{1}{f(\eta,\tau)+\eta}-\frac{1}{\eta}+\frac{1}{\eta}\Big] d\eta\notag\\
    &+(2k-4)f_{\xi} \Big[\frac{1}{f(\xi,\tau)+\xi}-\frac{1}{\xi}+\frac{1}{\xi}\Big]\\
    &-(k-1)f_\xi \int_0^\xi \frac{f_{\eta\eta}(\eta,\tau)}{f(\eta,\tau)+\eta} d\eta+(k-2)\frac{f^2_{\xi}}{f + \xi}\notag\\ 
    &-(n-k)(f_\xi + 1)\int_0^\xi \frac{g_{\eta}^2(\eta,\tau)}{(g(\eta,\tau) + \sqrt{2(n-k-1)})^2} d\eta. \notag
\end{align}
Separating the terms linearly depending on $f$ and its derivatives and integrating by parts  we get
\begin{equation}
  -(k-1)\int_{0}^{\xi} \frac{f_{\eta\eta}(\eta,\tau)}{\eta}d\eta={-(k-1)}\,\frac{f_{\xi}(\xi, \tau)}{\xi}-{(k-1)}\int_{0}^{\xi}\frac{f_{\eta}(\eta, \tau)}{\eta^2}d\eta
\end{equation}
and
\begin{equation}
\begin{aligned}
    &\quad-(k-1)f_\xi (\xi, \tau)  \int_0^\xi \frac{f_{\eta\eta}(\eta,\tau)}{f(\eta,\tau)+\eta} d\eta\\
    &=-(k-1)\frac{f_{\xi}^2(\xi,\tau)}{f(\xi,\tau)+\xi}-(k-1)f_{\xi}\int_0^\xi \frac{f_{\eta}(\eta,\tau)(f_{\eta}(\eta,\tau)+1)}{(f(\eta,\tau)+\eta)^2 }d\eta
\end{aligned}
\end{equation}
By the above and $\lim_{\xi\to 0}\frac{f_{\xi}}{f+\xi}=\frac{f_{\xi\xi}(0, \tau)}{1+f_{\xi}(0, \tau)}=0$, {where we have used \eqref{fxx=0})}, 
we obtain the desired evolution equation for $f$ with linear operator $\mathcal{L}_0$ and error term $E_0$. 

The evolution expansion for $g$ follows from adding and subtracting $g$ in \eqref{g-equation} and separating the $(k- 1)g_\xi(\xi,\tau) \frac{1}{\xi}$ term from last line of \eqref{g-equation} and \eqref{gequation term expansion} follows from direct Taylor expansion.

%
%
\end{proof}

An immediate consequence of Proposition \ref{evolution of (f, g)} is the following corollary.
\begin{cor}
We have 
\begin{align}\label{f-error-derivative-expansion}
    E_{0, \xi}&=\partial_\xi E_0\notag\\
    &=(k-3) {f_{\xi \xi }}\Big[\frac{1}{f(\xi ,\tau)+\xi }-\frac{1}{\xi}\Big] \notag -(2k-4){f_{\xi}}\Big[\frac{f_{\xi}+1}{(f(\xi,\tau)+\xi)^2}-\frac{1}{\xi^2}\Big] \notag\\
    &-\frac{2f_{\xi}f_{\xi\xi}}{f+\xi} {-}(k-2)\frac{f^2_{\xi}(f_{\xi}+1)}{(f+\xi)^2}-(k-1)f_{\xi\xi}(\xi,\tau)\int_0^\xi \frac{f_{\eta}(\eta,\tau)(f_{\eta}(\eta,\tau)+1)}{(f(\eta,\tau)+\eta)^2} d\eta\\
    &-(n-k)f_{\xi\xi}(\xi,\tau)\int_0^\xi \frac{g_{\eta}^2(\eta,\tau)}{(g(\eta,\tau) + \sqrt{2(n-k-1)})^2} d\eta  \notag \\
    & -(n-k)(f_\xi(\xi,\tau) + 1) \frac{g_{\xi}^2(\xi,\tau)}{(g(\xi,\tau) + \sqrt{2(n-k-1)})^2}  \notag
\end{align}
and
\begin{align}\label{gequation term expansion}
\begin{aligned}
    &\quad-g(\xi, \tau) - \frac{(n-k-1)}{g(\xi,\tau) + \sqrt{2(n-k-1)})} +\frac{1}{2} \big (g(\xi,\tau) + \sqrt{2(n-k-1)} \big ) \\
    &=-\frac{g^2}{2\sqrt{2(n-k-1)}}+O(g^3).  
\end{aligned}
\end{align}
\end{cor} 
\begin{proof}
Equation \eqref{f-error-derivative-expansion} immediately follows by differentiating \eqref{f-error-expansion}. On the other hand, \eqref{gequation term expansion} follows  directly by using Taylor's expansion.

\end{proof}

Let $h = f_\xi$ with boundary value $f(0)=0$, then $f_\xi$ and $h$ are mutually uniquely determined. Then,
\begin{prop}[evolution expansion for $h=f_\xi$]\label{evolution of h=fxi}
We have
\begin{align}\label{h-expansion}
   h_\tau =\mathcal{L}_2 h+E_2,
\end{align}
where 
\begin{align}
\mathcal{L}_2 h = h_{\xi\xi} - \frac{1}{2}\xi h_\xi +\frac{k-3}{\xi}h_\xi - \frac{2k-4}{\xi^2}h
\end{align} 
and

\begin{align}\label{e3hg}
E_2(h, g)
={}&(k-3)\, h_{\xi}(\xi ,\tau)\Bigg[\frac{1}{\xi+\int_0^\xi h(\eta,\tau)\,d\eta }-\frac{1}{\xi}\Bigg] \\\nonumber
&-(2k-4)\,h(\xi,\tau)\Bigg[\frac{h(\xi,\tau)+1}{\left(\xi+\int_0^\xi h(\eta,\tau)\,d\eta\right)^2}-\frac{1}{\xi^2}\Bigg] \\\nonumber
&-\frac{2h(\xi,\tau)h_{\xi}(\xi,\tau)}{\xi+\int_0^\xi h(\eta,\tau)\,d\eta}
+(2-k)\frac{h(\xi,\tau)^2(h(\xi,\tau)+1)}
{\left(\xi+\int_0^\xi h(\eta,\tau)\,d\eta\right)^2} \\\nonumber
&-(k-1)h_{\xi}(\xi,\tau)\int_0^\xi
\frac{h(\eta,\tau)\bigl(h(\eta,\tau)+1\bigr)}
{\left(\eta+\int_0^\eta h(s,\tau)\,ds\right)^2}\,d\eta \\\nonumber
&-(n-k)h_{\xi}(\xi,\tau)\int_0^\xi
\frac{g_{\eta}^2(\eta,\tau)}
{\left(g(\eta,\tau)+\sqrt{2(n-k-1)}\right)^2}\,d\eta \\\nonumber
&-(n-k)\bigl(h(\xi,\tau)+1\bigr)\frac{g_{\xi}^2(\xi,\tau)}
{\left(g(\xi,\tau)+\sqrt{2(n-k-1)}\right)^2}.
\end{align}
\end{prop}
\begin{proof}
    By differentiating the equation of $f$ in \eqref{f-g-expansion} and applying \eqref{f-error-derivative-expansion}, we obtain the desired result.
\end{proof}

\subsection{Evolution  of the inverse profile function of $G^2(z, t)$} 
We reparametrize  the metric $g(t) $  in  \eqref{eqn-metric} by setting $dz= U(r,t)^{-\frac{1}{2}}\, dr$ and {$G(z, t) = r$}, that is we consider 
\begin{equation}
 \bar g(t)  = U^{-1}(r,t) dr^2 + F^2(r,t) g_{{{S}}^{k-1}} + r^2 g_{{{S}}^{n-k}}
\end{equation}
Then,  we have 
\begin{equation}
    U=G_z^2,\quad \frac{dr}{dz}=G_z\leq0
\end{equation}
and here we use the following convention about choosing the square root $U^{\frac{1}{2}}:=G_{z}\leq 0$, so
\begin{align}
    \frac{\partial}{\partial z} = U^{\frac{1}{2}}\frac{\partial}{\partial r}=G_z\frac{\partial}{\partial r},
\end{align}
which leads to 
\begin{align}
    \frac{F_{zz}}{F} = U\frac{F_{rr}}{F} + \frac{1}{2}U_r \frac{F_r}{F},
\end{align}
\begin{align}
    \frac{1-F_z^2}{F^2} = \frac{1-UF_r^2}{F^2}
\end{align}
\begin{align}
    \frac{G_{zz}}{G} = \frac{1}{2r}U_r,
\end{align}
\begin{align}
    \frac{F_zG_z}{FG} = \frac{UF_r}{rF}
\end{align}
and
\begin{align}
    \frac{1-G_z^2}{G^2} = \frac{1-U}{r^2}. 
\end{align}
Then \eqref{Ricci} gives 
\begin{align}
    {\rm Ric}_{rr} &= -(k-1)\frac{F_{rr}}{F}- \frac{(k-1)U_r}{2U}\frac{F_r}{F} - \frac{(n-k)}{2}\frac{U_r}{Ur}, \notag\\
    {\rm Ric}_{{{S}}^{k-1}} &= -U\frac{F_{rr}}{F}- \frac{1}{2}U_r\frac{F_r}{F} + (k-2)\frac{1-UF_r^2}{F^2} - (n-k)\frac{UF_r}{rF}, \notag\\
    {\rm Ric}_{{{S}}^{n-k}} &= -\frac{U_r}{2r} + (n-k-1)\frac{1-U}{r^2} - (k-1)\frac{UF_r}{rF}.
\end{align}
Since the original metric $g(t)$ evolves  by  Ricci flow, the reparametrized metric $\bar g(t)$ satisfies  an evolution equation of the form
\begin{align}
    \frac{\partial \bar g}{\partial t} = -2{\rm Ric}_{\bar g} + \mathcal{L}_{V} (\bar g)
\end{align}
where $V = v(r,t)\frac{\partial}{\partial r}$,  depending  also on time.
Clearly, 
\begin{align}
    \frac{\partial \bar g}{\partial t} = -U^{-2} U_t \, dr^2 + 2FF_t \, d\theta^2.
\end{align}
Moreover, $\mathcal{L}_V(r) = v$, $\mathcal{L}_V(dr) = v_rdr$, implying  
\begin{align}
    \mathcal{L}_{V} (\bar g) = (-U^{-2}U_r v + 2U^{-1}v_r)dr^2 + 2FF_r v d\theta^2 + 2rv g_{{{S}}^{n-k}}.
\end{align}
Hence 
\begin{align}
    -2{\rm Ric}_{\bar g} + \mathcal{L}_{V} (\bar g) &= (2(k-1)\frac{F_{rr}}{F} + (k-1)\frac{U_r}{U}\frac{F_r}{F} + (n-k)\frac{U_r}{Ur} - U^{-2}U_r v + 2U^{-1}v_r) dr^2 \notag\\
    &+ (2U\frac{F_{rr}}{F} + U_r\frac{F_r}{F} - 2(k-2)\frac{1-UF_r^2}{F^2} + 2(n-k)\frac{UF_r}{rF} + 2\frac{F_r}{F}v ) F^2d\theta^2 \notag\\
    & + (\frac{1}{r}U_r - 2(n-k-1)\frac{1-U}{r^2} + 2(k-1)\frac{UF_r}{rF} + 2\frac{1}{r}v) r^2g_{{{S}}^{n-k}}.
\end{align}
Putting these facts together, we conclude that
\begin{align}
    v = -\frac{1}{2}U_r + (n-k-1)\frac{1-U}{r} - (k-1)\frac{UF_r}{F}
\end{align}
and 
\begin{align}
    v_r =& -\frac{1}{2}U_{rr} - (n-k-1)\frac{1-U}{r^2}-(n-k-1)\frac{1}{r}U_r\notag\\
    &- (k-1)U_r\frac{F_r}{F} - (k-1)U\frac{F_{rr}}{F} +(k-1) U\frac{F_r^2}{F^2}.
\end{align}
Thus we have
\begin{align}\label{u-equation}
    U_t =& -2(k-1)U^2\frac{F_{rr}}{F} - (k-1)UU_r\frac{F_r}{F} - (n-k)\frac{UU_r}{r} + U_r v - 2Uv_r\notag\\ 
    =& U\left(U_{rr} + (n-k-2)\frac{U_r}{r}\right) - \frac{1}{2}U_r^2 + (n-k-1)r^{-2}(1-U)(rU_r+2U)\notag\\
    &- 2(k-1)U^2\frac{F_r^2}{F^2}\notag\\
    =& UU_{rr} - \frac{1}{2}U_r^2 + \frac{n-k-1-U}{r}U_r + \frac{2(n-k-1)}{r^2}U(1-U) \notag\\
    &- 2(k-1)U^2\frac{F_r^2}{F^2}
\end{align}
and 
\begin{align}\label{evolution-F}
    F_t =& 2UF_{rr} + \frac{1}{2}U_rF_r - (k-2)\frac{1-UF_r^2}{F} + (n-k)\frac{UF_r}{r} + F_rv \notag\\
    =&U(F_{rr}+\frac{1}{r}F_r)+\frac{n-k-1}{r}F_r -\frac{k-2+UF_r^2}{F}.
\end{align}

\subsection{Truncated profile functions and Brendle's barrier}

Note that $G$ is monotone decreasing on $[0, z_{\max}]$, {and therefore it makes sense to define the maximal radius as follows.}
\begin{defn}[maximal radius]\label{maximal-radius} Define the maximal radius 
  \begin{equation}\label{maximal-radiusrmax}
    r_{max}(t)=G(0, t).
\end{equation}  
\end{defn}

\smallskip
\begin{lem}\label{r-low}
    We have $-\frac{1}{2}\frac{d}{dt}(r_{max}(t)^2)\geq (n-k-1)$. In particular, $r_{max}(t)\geq \sqrt{-2(n-k-1)t}$,  for each $t$.  
\end{lem}

\begin{proof}
     Consider the point where the radius is maximal. At that point, $G = r_{max}(t)$, $G_z = 0$, $G_{zz} \leq 0$. Using the evolution equation for $G$, we conclude that $-\frac{1}{2} \frac{d}{dt}(G^2)\geq (n-k-1)$. Integrating over $t$, we obtain $r_{max}(t)\geq \sqrt{-2(n-k-1)t}$ for each $t$.
\end{proof}

\begin{lem}\label{r-up}
    We have $-\frac{1}{2}\frac{d}{dt}(r_{max}(t)^2)\leq (1+o(1))(n-k-1)$ as $t\to -\infty$. In particular, $r_{max}(t)\leq (1+o(1))\sqrt{-2(n-k-1)t}$ as $t\to -\infty$.
\end{lem}

\begin{proof}
    Let $\epsilon>0$ be given. Let us consider the point where the radius is maximal. At that point, $G=r_{max }(t) \geq \sqrt{-2(n-k-1) t}$ and $G_z=0$. Proposition \ref{epsilon-bubble-sheet} implies that the point where the radius is maximal lies on an evolving $\epsilon$-cylinders if $-t$ is sufficiently large. Hence, if $-t$ is sufficiently large, then we have $G_{z z} \geq-\epsilon G^{-1}$ at the point where the radius is maximal. Using the evolution equation for $G$, we obtain $-\frac{1}{2} \frac{\partial}{\partial t}(G^2) \leq {(n-k-1)}+\epsilon$ at the point where the radius is maximal. Thus, we conclude that $-\frac{1}{2} \frac{d}{d t}(r_{max }(t)^2) \leq {(n-k-1)}+\epsilon$ if $-t$ is sufficiently large. Since $\epsilon>0$ is arbitrary, it follows that $-\frac{1}{2} \frac{d}{d t}(r_{\max }(t)^2) \leq {(n-k-1)\, (1 + o(1))}$ as $t \to -\infty$. This finally implies $r_{max }(t) \leq(1+o(1)) \sqrt{-2(n-k-1) t}$ as $t \to -\infty$.
\end{proof}

\begin{lem}
    We have $g(0, \tau) \geq 0$ for each $\tau$. Moreover, $g(0, \tau) \rightarrow 0$ as $\tau \rightarrow-\infty$.
\end{lem}

\begin{proof}
     This is an immediate consequence of Lemma \ref{r-low} and Lemma \ref{r-up}.
\end{proof}

For each time $\bar{\tau}$, we define 
\begin{align}\label{delta-def}
    \delta(\bar{\tau}):=\sup _{\tau \leq \bar{\tau}}\left(|g(0, \tau)|+|g_{\xi\xi}(0, \tau)|\right) 
\end{align}
By definition, $\delta(\bar{\tau})$ is an increasing function of $\bar{\tau}$. Moreover, $\delta(\bar{\tau}) \rightarrow 0$ as $\bar{\tau} \rightarrow-\infty$.

\begin{lem}\label{derivative-delta}
    For $-\bar{\tau}$ sufficiently large, the function $\bar{\tau} \mapsto \delta(\bar{\tau})$ is Lipschitz continuous with Lipschitz constant 1. In particular, the function $\bar{\tau} \mapsto \delta(\bar{\tau})$ is differentiable almost everywhere, and $0 \leq \delta^{\prime}(\bar{\tau}) \leq 1$ for $-\bar{\tau}$ sufficiently large.
\end{lem}
\begin{proof}
    The proof is similar to  \cite[Lemma 3.11]{ABDS}.
\end{proof}

Then for a smooth cut-off function  $\chi(s)$ defined by
\begin{equation}\label{cutoffchidef}
\begin{aligned}
    &\chi(s)=
        1\quad s\in [-1/2, 1/2]\\
    &\chi(s)\in [0, 1]\quad  \quad s\in [-1, -1/2]\cup [1/2, 1]\\
       &\chi(s)= 0\quad s\in \mathbb{R}\setminus [-1, 1]
    \end{aligned}
\end{equation}
and satisfying 
$s\chi'(s)\leq 0$ for all $s\in \mathbb{R}$, we define 
\begin{equation}\label{hat-fg}
   \hat{f}(\xi, \tau)=f(\xi, \tau)\chi(\delta(\tau)^{\frac{1}{100}}\xi)\quad \hat{g}(\xi, \tau)=g(\xi, \tau)\chi(\delta(\tau)^{\frac{1}{100}}\xi).
\end{equation}

Then, we recall the following barrier construction of Brendle, which is useful for gradient estimates of $G$.
\begin{prop}[Proposition 3.7 \cite{ABDS}]\label{g-barrier}
    There exists a positive real number $r_*$ with the following property. Given any large number a, we can find a continuously differentiable function $\psi_a:\left[r_* a^{-1}, 1+\frac{1}{100} a^{-2}\right] \rightarrow \mathbb{R}$ such that
    \begin{align}
        \psi_a(s)\psi_a(s)^{''} - \frac{1}{2}\psi'_a(s)^2 + \frac{n-k-1-\psi_a(s)}{s}\psi'_a(s) + \frac{2(n-k-1)}{\mathbb{S}^2}\psi_a(s)(1-\psi_a(s))<0
    \end{align}
    for all $s \in\left[r_* a^{-1}, 1+\frac{1}{100} a^{-2}\right]$. The function $\psi_a$ satisfies $\psi_a(s) \leq C a^{-2}$ for all $s \in\left[\frac{1}{10}, 1+\frac{1}{100} a^{-2}\right]$. Moreover, $\psi_a(s) \geq \frac{1}{32} a^{-4}$ for all $s \in\left[r_* a^{-1}, 1+\frac{1}{100} a^{-2}\right]$, and $\psi_a(s) \geq a^{-2}\left(s^{-2}-1\right)+\frac{1}{16} a^{-4}$ for all $s \in\left[1-\theta, 1+\frac{1}{100} a^{-2}\right]$, where $\theta$ is a small positive number. Finally, $\psi_a\left(r_* a^{-1}\right) \geq \frac{3}{2}$.
\end{prop}

\section{Spectral alternatives}\label{Spectral alternatives}
\subsection{A priori estimates for renormalized profile functions $f$ and $g$}
In this section, we aim to derive  derivative estimates for $f(\xi,\tau)$ and $g(\xi,\tau)$ as defined
in \eqref{f-g}. Since we  have assumed $z >0$, we also have  $\xi >0$.   Also, recall the boundary conditions
\begin{equation}\label{BC fg}
    f(0, \tau)=0, \qquad f_{\xi}(0, \tau)=f_{\xi\xi}(0, \tau)=g_{\xi}(0, \tau)=0
\end{equation}
and also the conditions  
\begin{equation}
    f\leq 0,\qquad  -1\leq f_{\xi}\leq 0, \qquad  g_{\xi}\leq 0.
\end{equation}
Thus, by the monotonicity of $f$, we have 
\begin{align}\label{f-slope}
   \frac{|f|}{\xi} \leq |f_{\xi}|, \qquad \mbox{for all} \,\, (\xi, \tau), \,\,  \xi  >0. 
\end{align}
We also recall definition of $ \delta(\tau)$ in \eqref{delta-def}
\begin{align}\label{delta-def2}
    \delta(\tau):=\sup _{\tau' \leq \tau}\left(|g(0, \tau')|+|g_{\xi\xi}(0, \tau')| \right) 
\end{align}
\begin{prop}\label{G-gradient}
Let $r^*>0$ be as in Proposition \ref{g-barrier}   and  suppose  that $a$ is sufficiently large. If  $\bar{t}$ is a time with the property that $\frac{r_{\max }(t)}{\sqrt{-2(n-k-1)t}} \leq 1+\frac{1}{100} a^{-2}$ for all $t \leq \bar{t}$, then  $G_z(z, t)^2<\psi_a\left(\frac{G(z, t)}{\sqrt{-2(n-k-1)t}}\right)$ whenever $t \leq \bar{t}$ and $G(z, t) \geq r_* a^{-1} \sqrt{-2(n-k-1)t}$.
\end{prop}

\begin{proof}
By \eqref{u-equation}, we have 
\begin{align}
    U_t\leq UU_{rr} - \frac{1}{2}U_r^2 + \frac{n-k-1-U}{r}U_r + \frac{2(n-k-1)}{r^2}U(1-U),
\end{align}
which allows us to apply the maximum principle  as in \cite[Proposition 3.8]{ABDS}. Then, the proof is identical to \cite[Proposition 3.8]{ABDS}.
\end{proof}

\begin{prop}\label{g_xi-barrier}
Fix  $\bar{\tau}$ so that $-\bar \tau$ is sufficiently large, and let $a:=\frac{1}{10} \delta(\bar{\tau})^{-\frac{1}{2}}$. Then,  \begin{equation}
    g_{\xi}(\xi, \tau)^2 \leq \psi_a\,  \big(1+\frac{g(\xi, \tau)}{\sqrt{2(n-k-1)}}\big)
\end{equation}
holds,  whenever $\tau \leq \bar{\tau}$ and 
\begin{equation}
    g(\xi, \tau) \geq\left(r_* a^{-1}-1\right) \sqrt{2(n-k-1)},
\end{equation}
 where $r_*$ is defined as in \cite{Bre-3d-noncpt}.
\end{prop}
\begin{proof}
The proof is the same as the proof in \cite[Proposition 3.12]{ABDS}.
\end{proof}

We will now prove a number of estimates  holding  for $\tau \leq \bar \tau$,  where $-\bar \tau $ will be  a sufficiently large number so that Proposition \ref{g_xi-barrier} and can be made even larger so that all our estimates below hold for $\tau \leq \bar \tau$.

\begin{lem}\label{g-C01}
    We have $|g(\xi, \tau)| + |g_\xi(\xi, \tau)|\leq C \delta^{\frac{1}{4}}(\tau)$,  for $0 \leq \xi \leq 2\delta(\tau)^{-\frac{1}{100}}$, $\tau \leq \bar \tau$.
\end{lem}

\begin{proof}
We  know that $0\leq g(0,\tau)\leq \delta(\tau)$ and also by Proposition  \ref{g_xi-barrier} with $\bar \tau = \tau$, we have 
 $g_\xi(\xi,\tau)^2\leq C\delta(\tau)$,  whenever $g(\xi, \tau) \geq\left(r_* a^{-1}-1\right) \sqrt{2(n-k-1)}$. These two bounds  imply  that $|g(\xi, \tau)|   \leq \delta(\tau) + \xi \,  \sqrt{C\delta(\tau)} \leq \delta^{\frac 14}  (\tau) $, for $\xi\leq 2\delta(\tau)^{-\frac{1}{100}}$.
\end{proof}

\begin{lem}\label{g-C-m}
 We have $|\frac{\partial^m}{\partial\xi^m}g(\xi, \tau)| \leq C(m)$,  for $0 < \xi \leq 2\delta(\tau)^{-\frac{1}{100}}$, $\tau \leq \bar \tau$.
\end{lem}
\begin{proof}
    This follows from Proposition \ref{Gmorder estimates} and definition of $g$ in \eqref{f-g}.
\end{proof}

\begin{lem}\label{g-C23}
    We have $|g_{\xi\xi}(\xi, \tau)| + |g_{\xi\xi\xi}(\xi, \tau)|\leq C \delta^{\frac{1}{16}}(\tau)$,  for $0 < \xi \leq 2\delta(\tau)^{-\frac{1}{100}}$, $\tau \leq \bar \tau$.
\end{lem}

\begin{proof}
We only prove the estimate for $|g_{\xi\xi}(\xi, \tau)|$, since the same argument applies to  $|g_{\xi\xi\xi}(\xi, \tau)|$. By Lemma \ref{g-C-m} with $m=2,3$, we obtain $\left|g_{\xi \xi \xi}(\xi, \tau)\right| + \left|g_{\xi \xi}(\xi, \tau)\right| \leq C$, for $\tau \leq-k$ and $ 0 < \xi   \leq 2\delta(\tau)^{-\frac{1}{100}}$  Moreover, applying interpolation inequality and Lemma \ref{g-C01}, yield
\begin{align}
\inf _{\xi^{\prime} \in\left[\xi-\delta(\tau)^{\frac{1}{8}}, \xi+\delta(\tau)^{\frac{1}{8}}\right]}\left|g_{\xi \xi}\left(\xi^{\prime}, \tau\right)\right| \leq C \delta^{\frac{1}{8}}
\end{align}
for $\tau \leq \bar\tau$ and $\xi \leq \delta^{-\frac{1}{100}}$. Putting these facts together, we conclude that $\left|g_{\xi \xi}(\xi, \tau)\right| \leq C \delta^{\frac{1}{8}}$ for $\tau \leq\bar\tau$ and $\xi \leq \delta^{-\frac{1}{100}}$. This completes the proof of the lemma.
\end{proof}

\begin{lem}\label{apriori-f} 
For any $\epsilon>0$, there exists $\bar \tau$, such that  
\begin{equation}\label{eqn-sf}
f\geq -\tfrac{1}{2}\xi, \qquad   -\tfrac{1}{2} \leq f_\xi \leq 0, \qquad 0\leq -\frac{f_{\xi\xi}}{f+\xi}\leq \epsilon
\end{equation}
holds for all $0 < \xi\leq \delta(\tau)^{-\frac{1}{100}}$, $\tau \leq \bar \tau$. 
\end{lem}

\begin{proof}
By  smallness of $g$ on $0 < \xi\leq \delta^{-\frac{1}{100}}$ (see Lemma  \ref{g-C01}),  the definition of $\delta(\tau)$ in \eqref{delta-def} and Proposition \ref{epsilon-bubble-sheet}, we have that the  rescaled flow  converges  to  the shrinking cylinder $\mathbb R^k\times S^{n-k}(\sqrt{2(n-k-1)|t|})$. Hence,  Lemma \ref{apriori-f}  readily  follows from these facts.

\end{proof}

\subsection{Estimates for error term $E_1$}
We will now estimate the error term $E_1$  in the evolution equation of $g$ in \eqref{g-error-expansion} of Proposition \ref{evolution of (f, g)}.
All estimates below will hold for $\tau \leq \bar \tau$, where $-\bar \tau$ is a sufficiently large constant.   

\begin{lem}\label{gxi4}
The following inequality holds for $\tau \leq \bar \tau$, 
\begin{align}\label{g-k-1-integral}
&\intde  \xi^{k-1} e^{-\frac{\xi^2}{4}}\, g_{\xi}(\xi, \tau)^4  d \xi \notag\\  &\leq C \, \delta(\tau)^{\frac{1}{100}} \intde  \xi^{k-1} e^{-\frac{\xi^2}{4}} \, g(\xi, \tau)^2 d \xi 
+C \exp \big(-\tfrac{1}{8} \delta(\tau)^{-\frac{1}{50}}\big).
\end{align}

\end{lem}

\begin{proof}
We begin by observing  that for $\tau \leq \bar \tau$,  we have 
\begin{align}
\begin{aligned}
& \intde \xi^{k-1} e^{-\frac{\xi^2}{4}} g_{\xi}(\xi, \tau)^4 d \xi  \\&+3 \intde \xi^{k-1} e^{-\frac{\xi^2}{4}}  g_{\xi}(\xi, \tau)^2 g(\xi, \tau) g_{\xi \xi}(\xi, \tau) d \xi \\& -\frac{1}{2} \intde \xi^{k} e^{-\frac{\xi^2}{4}} g_{\xi}(\xi, \tau)^3 g(\xi, \tau) d \xi \\
&+ (k-1)\intde \xi^{k-2} e^{-\frac{\xi^2}{4}} g_{\xi}(\xi, \tau)^3 g(\xi, \tau) d \xi \\
& =\intde \frac{\partial}{\partial \xi}\Big (\xi^{k-1} e^{-\frac{\xi^2}{4}}  g_{\xi}(\xi, \tau)^3 g(\xi, \tau)\Big ) d \xi \\
 &\leq C \exp \left(-\tfrac{1}{8} \delta^{-\frac{1}{50}}\right)
\end{aligned}
\end{align}
where in the last inequality we use the Lemma \ref{g-C01}. Using Lemma \ref{g-C23}, we obtain 
\begin{align}
\begin{aligned}
& -3 \intde \xi^{k-1} e^{-\frac{\xi^2}{4}} g_{\xi}(\xi, \tau)^2 g(\xi, \tau) g_{\xi \xi}(\xi, \tau) d \xi \\
& \leq C \delta^{\frac{1}{100}}(\tau)\intde \xi^{k-1} e^{-\frac{\xi^2}{4}}\,  g_{\xi}(\xi, \tau)^2 \, |g(\xi, \tau)| d \xi
\end{aligned}
\end{align}
while  using Lemma \ref{g-C01}, we obtain 
\begin{align}
\begin{aligned}
& \frac{1}{2} \intde \xi^{k} e^{-\frac{\xi^2}{4}} g_{\xi}(\xi, \tau)^3 g(\xi, \tau) d \xi \\
& \leq C \delta^{\frac{1}{100}}(\tau)\intde \xi^{k-1} e^{-\frac{\xi^2}{4}} \, g_{\xi}(\xi, \tau)^2 \, |g(\xi, \tau)| d \xi. 
\end{aligned}
\end{align}
Moreover, by Lemma \ref{g-C01} and Lemma \ref{g-C23} we have
\begin{align}
\begin{aligned}
&-(k-1)\intde \xi^{k-2} e^{-\frac{\xi^2}{4}} g_{\xi}^3\,  g(\xi, \tau) d \xi  \\
&= -(k-1)\intde \xi^{k-1} e^{-\frac{\xi^2}{4}} \,\, \frac{g_{\xi}}{\xi}g_{\xi}^2 \, g \, d \xi\\
&\leq C \max_{\{ 0 < \xi < 2 \delta(\tau)^{-\frac 1{100}} \}} \, |g_{\xi\xi}| \, \intde   \xi^{k-1} e^{-\frac{\xi^2}{4}} g_{\xi}^2\,  |g| \,  d \xi \\
&\leq C \, \delta^{\frac{1}{100}}(\tau) \, \intde \xi^{k-1} e^{-\frac{\xi^2}{4}} \, g_{\xi}^2\, |g| d \xi
\end{aligned}
\end{align}
for $\tau\leq \bar\tau$. Adding these inequalities  together gives 
\begin{align}
\begin{aligned}
& \intde \xi^{k-1} e^{-\frac{\xi^2}{4}} g_{\xi}(\xi, \tau)^4 d \xi\\
&\leq C \delta^{\frac{1}{100}}(\tau) \, \intde \xi^{k-1} e^{-\frac{\xi^2}{4}} g_{\xi}(\xi, \tau)^2\, |g(\xi, \tau)| d \xi 
+ C \exp \left(-\tfrac{1}{8} \delta^{-\frac{1}{50}}\right)\\
 & \leq C \delta^{\frac{1}{100}}(\tau)\intde \xi^{k-1} e^{-\frac{\xi^2}{4}} \, \left( g_{\xi}(\xi, \tau)^4 + |g(\xi, \tau)|^2\right) d \xi
+ C \exp \left(-\tfrac{1}{8} \delta^{-\frac{1}{50}}\right).
\end{aligned}
\end{align}
\end{proof}

\begin{lem}\label{gequation errorf_1}
The following inequality holds for  $\tau \leq \bar \tau$, 
\begin{align}
\begin{aligned}
&\intde \xi^{k-1}e^{-\frac{\xi^2}{4}}\left|-(k-1)g_\xi(\xi,\tau)\left( \frac{f(\xi,\tau)}{\xi(f(\xi,\tau)+\xi)}\right)\right|^2d\xi\\
&\leq C \delta^{\frac{1}{100}}(\tau)\intde \xi^{k-3}e^{-\frac{\xi^2}{4}} f_\xi^2 d\xi +C \exp \big(-\frac{1}{8} \delta(\tau)^{-\frac{1}{50}}\big).
\end{aligned}
\end{align}

\end{lem}

\begin{proof}
This estimate  directly follows  by the slope estimates for $f$ in \eqref{f-slope} and by Lemmas \ref{g-C01} and  \ref{apriori-f}.
\end{proof}

\begin{lem}\label{gequation errorf_2}
For $\tau \leq \bar \tau$, we have
\begin{align}
\begin{aligned}
&\intde \xi^{k-1}e^{-\frac{\xi^2}{4}}\left(\int_0^\xi \frac{{f_\eta (\eta,\tau) (f_{\eta}(\eta,\tau)+1)}}{(f+\eta)^2}d\eta\right)^2 g_{\xi}^2\, d\xi\\
&\leq C \delta^{\frac{1}{100}}(\tau)\intde \xi^{k-3}e^{-\frac{\xi^2}{4}} f_\xi^2 d\xi +C \exp \big(-\frac{1}{8} \delta(\tau)^{-\frac{1}{50}}\big).
\end{aligned}
\end{align}

\end{lem}

\begin{proof}
We first claim we have the following estimate
\begin{align}\label{decay times k-3}
\begin{aligned}
    &\intde\xi^{k-1} \left(\int_0^\xi \frac{{f_\eta (\eta,\tau)}}{\eta^2}d\eta\right)^2 e^{-\frac{\xi^2}{4}}d\xi\\
    &\leq C \intde\xi^{k-3}\, e^{-\frac{\xi^2}{4}} f_\xi(\xi,\tau)^2 \, d\xi
    + C \exp \big(-\tfrac{1}{8} \delta(\tau)^{-\frac{1}{50}}\big). 
\end{aligned}
\end{align}
To prove  \eqref{decay times k-3}, we set for simplicity $A:= \delta^{-\frac{1}{100}}(\tau)$,  and consider the  weight
    \begin{equation}\label{rho weight new}
  \rho(\xi)=  \frac{
  \xi^{-(k+1)} e^{\tfrac{\xi^2}{4}} 
}{A^{-(k+1)}e^{\frac{A^2}{4}}+2\int^{A}_{\xi}
    \sigma^{-(k+1)}
    e^{\frac{\sigma^2}{4}} \,
  \mathrm{d}\sigma
}, \quad \xi\in (0, A]
    \end{equation}
that satisfies 
\begin{equation}\label{rho behavior}
 \rho >   0, \quad \rho(A)=1,\quad  \lim_{\xi\to 0+}\xi \, \rho(\xi)=\frac{k}{2}>0, \quad 
 \xi \, \rho(\xi)\geq c>0,  \quad \mbox{on}\,\,  \xi\in (0, A].
\end{equation}
Also, to simplify the notation, for a fixed $\tau \leq \bar \tau$, set
\begin{equation}\label{eqn-tilf} 
\tilde f(\xi) := \int_0^\xi \frac{f_{\eta}(\eta, \tau)}{\eta^2}\, d\eta.
\end{equation}One can easily see  that
$  \tilde{f}(\xi)$ defined  by \eqref{eqn-tilf} satisfies $
    \tilde{f}(0)=0$ and  $\tilde f'(\xi) =\xi^{-2} f_\xi(\xi,\tau).$
We will now apply the  weighted Hardy inequality \eqref{decay times k-3}  shown in Lemma \ref{weights varied Gaussian Hardy},
 applied to  the function $\tilde f(\xi)$  on $[0, \delta(\tau)^{-\frac{1}{100}}]$,  with designed weight $\rho(\xi)^2$,  
 and $p=2$,  $a=k+1>1$. Also using the bound  $\xi^{-2} \leq c^{-1}  \tilde f(\xi)^2$ shown in  \eqref{rho behavior} we get  
\begin{align}
& \quad  \intde \xi^{k-1} \tilde f(\xi)^2  e^{-\frac{\xi^2}{4}}d\xi =\intde \xi^{k+1}\frac{1}{\xi^2} \tilde f(\xi)^2 e^{-\frac{\xi^2}{4}}d\xi\notag\\
        &\leq C\intde \tilde f(\xi)^2 \xi^{k+1}\rho(\xi)^2  e^{-\frac{\xi^2}{4}}d\xi\notag \leq C    \intde  \tilde{f}' (\xi)^2 \xi^{k+1} e^{-\frac{\xi^2}{4}} d\xi+ \tilde{f} 
        \big (\delta(\tau)^{-\frac 1{100}}\big )^2\\
        &\leq C\intde \xi^{k-3} f_{\xi}(\xi, \tau)^2 \, e^{-\frac{\xi^2}{4}}d\xi+C \exp \left(-\tfrac{1}{8} \delta(\tau)^{-\frac{1}{50}}\right)\notag
\end{align}
where in the last inequality we used that by  definition,  $ \tilde f'(\xi) =\xi^{-2}  \, f_{\xi}$. This readily implies  \eqref{decay times k-3}. 

Now  the lemma directly follows from  \eqref{decay times k-3}, $|g_\xi(\xi, \tau)|  \leq C \delta^{\frac{1}{4}}(\tau)$ for $\xi\leq 2\delta(\tau)^{-\frac{1}{100}}$,  and the bound 
$$0 \leq  \frac{{(-f_\eta) \, (f_{\eta}+1)}}{(f+\eta)^2} \leq \frac{ 4 (-f_\eta)}{\eta^2}$$ which follows  from  the bounds $f(\eta,\tau)  \geq - \frac \eta2$
and $-1 \leq f_\eta \leq 0$ shown in  \eqref{eqn-sf}.
\end{proof}

Together with these lemmas, we have following estimate for $E_1$,

\begin{prop}\label{Error-E2-estimate}
For all $\tau \leq \bar \tau$,  we have 
\begin{align}
\begin{aligned}
\intde \xi^{k-1} &e^{-\frac{\xi^2}{4}}|E_{1}|^2 d\xi \leq C \delta^{\frac{1}{100}}(\tau)\intde \xi^{k-3}e^{-\frac{\xi^2}{4}} f_{\xi}^2 d\xi \\
&+ C \delta^{\frac{1}{100}}(\tau)\intde \xi^{k-1}e^{-\frac{\xi^2}{4}} \, g^2 d\xi
+ C \exp \left(-\tfrac{1}{8} \delta(\tau)^{-\frac{1}{50}}\right). 
\end{aligned}
\end{align}

\end{prop}

\begin{proof}
We recall the error term $E_1$ is 
\begin{align}
   E_1&=-g(\xi, \tau) -   \frac{(n-k-1)}{g(\xi,\tau) + \sqrt{2(n-k-1)})} +\frac{1}{2}(g(\xi,\tau) + \sqrt{2(n-k-1)}) \\ &- \frac{g_\xi^2(\xi,\tau)}{g(\xi,\tau) + \sqrt{2(n-k-1)})}-(n-k)g_{\xi}(\xi,\tau)\int_0^\xi \frac{g_{\eta}^2(\eta,\tau)}{(g(\eta,\tau) + \sqrt{2(n-k-1)})^{2}}  \notag\\
    & -(k-1)g_\xi(\xi,\tau)\left( \frac{f(\xi,\tau)}{\xi(f(\xi,\tau)+\xi)} +\int_{0}^{\xi}
     \frac{f_{\eta}(\eta,\tau)(f_{\eta}(\eta,\tau) +1) }{(f(\eta,\tau)  + \eta)^2} d\eta\right)\notag.
\end{align}
By a similar argument as in \cite[Prop 3.6, Cor 3.3]{Bre-3d-noncpt} and noticing that $g_{\xi}(0, \tau)=0$, we have
\begin{align}\label{g error geta^2 integral estimate}
    &\int_0^\xi \frac{g_{\eta}^2(\eta,\tau)}{(g(\eta,\tau) + \sqrt{2(n-k-1)})^{2}}\\  &\leq C\left|\frac{1}{\sqrt{2(n-k-1)}+g(\xi, \tau)}-\frac{1}{\sqrt{2(n-k-1)}}\right||g_{\xi}(\xi, \tau)|\\
    &\leq C|g(\xi, \tau)||g_{\xi}(\xi, \tau)|.\notag
\end{align} 
Then using \eqref{gequation term expansion} and Cauchy inequality, we have
\begin{align}
   |E_1|&\leq Cg^2+Cg^2_{\xi}\notag\\
   &+C|g_\xi(\xi,\tau)|\left|\left(\frac{f(\xi,\tau)}{\xi(f(\xi,\tau)+\xi)} +\int_{0}^{\xi}
     \frac{{f_{\eta}(\eta,\tau)(f_{\eta}(\eta,\tau)+1) }}{(f(\eta,\tau)  + \eta)^2} d\eta\right)\right|. 
\end{align}
Then the proposition follows by Lemma \ref{gxi4}, \ref{gequation errorf_1} and \ref{gequation errorf_2}
\end{proof}

\subsection{Estimates for error terms $E_2$}
Let $E_2$ be the error in the evolution equation of $h=f_{\xi}$ in \eqref{h-error-expansion}, in Proposition \ref{evolution of h=fxi}. We estimate it term by term.
\begin{lem}\label{fequation-error_1term}
 For any $\epsilon > 0$ there exists a $\bar{\tau}$,  so that for $\tau \le \bar{\tau}$ we have 
\begin{align}
\begin{aligned}
&\intde \xi^{k-3}e^{-\frac{\xi^2}{4}} {f_{\xi \xi }^2} \, \Big[\frac{1}{f(\xi ,\tau)+\xi }-\frac{1}{\xi}\Big]^2d\xi\leq \epsilon^2  \intde \xi^{k-3} \, e^{-\frac{\xi^2}{4}}\, f_\xi^2 \, d \xi.
\end{aligned}
\end{align}
\end{lem}

\begin{proof}
By Lemma \ref{apriori-f} and the slope estimate for $f$ \eqref{f-slope}, we can estimate 
\begin{align}\label{3.3E31}
    \left| {f_{\xi \xi }}\, \Big[\frac{1}{f(\xi ,\tau)+\xi }-\frac{1}{\xi}\Big]\right| = \left|\frac{f_{\xi\xi}}{(f+\xi)} \cdot \frac{f}{\xi}\right| \leq \epsilon \left| \frac{f}{\xi}\right|\leq \epsilon\,  |f_\xi| 
\end{align}
for all $\xi \leq \delta(\tau)^{- \frac{1}{100}}$. The lemma readily follows.

\end{proof}

\begin{lem}\label{fequation-error_2term}
\begin{align}
\begin{aligned}
&\quad\intde \xi^{k-3} \left|-(2k-4){f_{\xi}(\xi,\tau)}\Big[\frac{f_{\xi}+1}{(f(\xi,\tau)+\xi)^2}-\frac{1}{\xi^2}\Big]\right|^2 e^{-\frac{\xi^2}{4}}d\xi\\
&\leq \epsilon \intde \xi^{k-3} e^{-\frac{\xi^2}{4}}|f_\xi(\xi, \tau)|^2 d \xi. 
\end{aligned}
\end{align}
\end{lem}

\begin{proof}
By slope estimate for $f$ \eqref{f-slope}, 
\begin{align}
\begin{aligned}
    \left|\frac{f_{\xi}+1}{(f(\xi,\tau)+\xi)^2}-\frac{1}{\xi^2}\right| = \frac{| \xi^2f_\xi-f^2-2\xi f| }{\xi^2(f(\xi,\tau)+\xi)^2} \leq \frac{|f_\xi|\, (3+|f_\xi|)|}{(f(\xi,\tau)+\xi)^2} \leq 4\, \frac{|f_\xi|}{(f(\xi,\tau)+\xi)^2}. 
\end{aligned}
\end{align}
By Lemma \ref{apriori-f} and differential mean value theorem, we can find a $0\leq \theta\leq \xi$ such that
\begin{align}\label{fxi/f+xi2}
 \frac{|f_\xi|}{(f(\xi,\tau)+\xi)^2}  \leq \left|\frac{f_{\xi\xi}(\theta)}{2(f(\theta,\tau)+\theta)(f_\xi(\theta,\tau)+1)}\right|\leq \epsilon
\end{align}
for $\xi \leq \delta(\tau)^{- \frac{1}{100}}$. The lemma readily  follows. 
\end{proof}

\begin{lem}\label{fequation-error_34term} We have 
    \begin{align}
\begin{aligned}
\intde \xi^{k-3} \left|\frac{f^2_{\xi}(f_{\xi}+1)}{(f+\xi)^2}-\frac{2f_{\xi}f_{\xi\xi}}{f+\xi}\right|^2 e^{-\frac{\xi^2}{4}}d\xi\leq \epsilon\intde \xi^{k-3} e^{-\frac{\xi^2}{4}}f_\xi^2 d \xi.
\end{aligned}
\end{align}
\end{lem}
\begin{proof}
 Lemma \ref{apriori-f} and    \eqref{fxi/f+xi2} yield
\begin{align}
\begin{aligned}\label{3.13E3}
\left|\frac{f^2_{\xi}(f_{\xi}+1)}{(f+\xi)^2}-\frac{2f_{\xi}f_{\xi\xi}}{f+\xi}\right| &\leq\frac{f^2_{\xi}| f_{\xi}+1 |}{(f+\xi)^2} + \left|\frac{2f_{\xi}f_{\xi\xi}}{f+\xi}\right| \\
&\le |f_{\xi}|  \frac{|f_{\xi}|}{(f + \xi)^2} + \epsilon |f_{\xi}| \le 2 \epsilon|f_\xi|
\end{aligned}
\end{align}
for $\xi \leq \delta(\tau)^{- \frac{1}{100}}$, where we use the same symbol $\epsilon$ to denote a small constant, as long as it changes from line to line in a uniform way. The lemma readily follows. 
\end{proof}

\begin{lem}\label{fequation-error_5term}
We have 
\begin{align}
\begin{aligned}
&\quad\intde \xi^{k-3} f_{\xi\xi}^2 \left|\int_0^\xi \frac{|(f_{\eta}(\eta,\tau)+1)f_{\eta}(\eta,\tau)|}{(f(\eta,\tau)+\eta)^2} d\eta\right|^2e^{-\frac{\xi^2}{4}}d\xi\\
&\leq \epsilon \intde \xi^{k-3} e^{-\frac{\xi^2}{4}}|f_\xi(\xi, \tau)|^2 d \xi+C \exp \big(-\frac{1}{8} \delta(\tau)^{-\frac{1}{50}}\big). 
\end{aligned}
\end{align}
\end{lem}

\begin{proof}
By Lemma \ref{apriori-f}, we have 
\begin{align}
\begin{aligned}\label{eq-useful}
&\quad\left|f_{\xi\xi}(\xi,\tau)\int_0^\xi \frac{f_{\eta}(\eta,\tau)(f_{\eta}(\eta,\tau)+1)}{(f(\eta,\tau)+\eta)^2} d\eta\right|\\
& = \left|\frac{f_{\xi\xi}(\xi,\tau)}{(f(\xi,\tau)+\xi)}(f(\xi,\tau)+\xi)\int_0^\xi \frac{f_{\eta}(\eta,\tau)(f_{\eta}(\eta,\tau)+1)}{(f(\eta,\tau)+\eta)^2} d\eta\right|\\
&\leq   4 \epsilon\left|\xi\int_0^\xi \frac{|f_{\eta}(\eta,\tau)|}{\eta^2} d\eta\right|. 
\end{aligned}
\end{align}
{where we also used that $f_{\xi\xi} < 0$ and $f(0) = 0$ imply  $f(\xi) \ge \xi\, f_{\xi}$, yielding $f(\xi) + \xi \ge \xi\, (f_{\xi} + 1)$, and hence
\begin{equation}
\label{eq-use-again}
\frac{f_{\xi}+1}{f + \xi} \le \frac 1\xi.
\end{equation}}

{ Then by \eqref{eq-useful} and \eqref{decay times k-3} from the weighted Hardy inequality in Lemma \ref{weights varied Gaussian Hardy} 
\begin{align}
\begin{aligned}
&\quad  {\intde \xi^{k-3} f_{\xi\xi}^2 \left (\int_0^\xi \frac{|(f_{\eta}(\eta,\tau)+1)f_{\eta}(\eta,\tau)|}{(f(\eta,\tau)+\eta)^2} d\eta\right )^2e^{-\frac{\xi^2}{4}}d\xi } \\
&\le \epsilon^2\,\intde \xi^{k-1} \,\left (\int_0^\xi \frac{f_{\eta}(\eta,\tau)}{\eta^2} d\eta\right )^2e^{-\frac{\xi^2}{4}}d\xi\\ 
    &\leq C\, \epsilon^2 \intde\xi^{k-3} f_{\xi}(\xi,\tau)^2 e^{-\frac{\xi^2}{4}}d\xi+C\epsilon^2 \exp \big(-\frac{1}{8} \delta(\tau)^{-\frac{1}{50}}\big).\\
\end{aligned}
\end{align}}
This completes the proof of Lemma \ref{fequation-error_5term}.

\end{proof}

\begin{lem}\label{fequation-error_6term}
\begin{align}
\begin{aligned}
&\quad\intde \xi^{k-3} \left|-(k-1){f_{\xi}(\xi,\tau)}\frac{f_\xi(\xi,\tau)(f_{\xi}(\xi,\tau)+1)}{(f(\xi,\tau)+\xi)^2}\right|^2 e^{-\frac{\xi^2}{4}}d\xi\\
&\leq \epsilon \intde \xi^{k-3} e^{-\frac{\xi^2}{4}}|f_\xi(\xi, \tau)|^2 d \xi. 
\end{aligned}
\end{align}
\end{lem}

\begin{proof}
This lemma follows directly by Lemma \ref{apriori-f} and differential mean value theorem as in \eqref{fxi/f+xi2}.
\end{proof}

\begin{lem}\label{fequation-error_7term}
We have 
\begin{align}
\begin{aligned}
&\quad\intde \xi^{k-3} f_{\xi\xi}^2 \left|\int_0^\xi \frac{g_{\eta}^2(\eta,\tau)}{(g(\eta,\tau) + \sqrt{2(n-k-1)})^2} d\eta\right|^2e^{-\frac{\xi^2}{4}}d\xi\\
&\leq C \delta^{\frac{1}{100}}(\tau)\intde \xi^{k-1} e^{-\frac{\xi^2}{4}}|g(\xi, \tau)|^2 d \xi+ C \exp \big(-\frac{1}{8} \delta(\tau)^{-\frac{1}{50}}\big).
\end{aligned}
\end{align}
\end{lem}

\begin{proof}
Since $g_\xi\leq 0$ and $g_{\xi\xi} \leq 0$ and $g + \sqrt{2(n-k-1)} \geq  \sqrt{(n-k-1)}$  by smallness of $g$  from Lemma \ref{g-C01}, we have 
\begin{align}
\int_0^\xi \frac{g_{\eta}^2(\eta,\tau)}{(g + \sqrt{2(n-k-1)})^2} d\eta &\leq g_{\xi}^2 \, \int_0^\xi \frac{1}{(g + \sqrt{2(n-k-1)})^2} d\eta \leq \frac{\xi}{(n-k-1)} \, g_{\xi}^2.\notag
\end{align}
Also, since by Lemma \ref{apriori-f}, we have $\left|\frac{f_{\xi\xi}}{\xi}\right|\leq \epsilon$, so $|\xi^{k-3}f_{\xi\xi}^2|=|\xi^{k-1}(f_{\xi\xi}^2\xi^{-2})|\leq \epsilon^2 \xi^{k-1}$,
we obtain 
\begin{align}
 &\intde \xi^{k-3} f_{\xi\xi}^2 \left ( \int_0^\xi \frac{g_{\eta}^2(\eta,\tau)}{(g+ \sqrt{2(n-k-1)})^2} d\eta\right )^2e^{-\frac{\xi^2}{4}}d\xi\\
&\leq C\epsilon^2\intde\xi^{k-1}\left(\frac{\xi}{n-k-1}|g_{\xi}|^2\right)^2 e^{-\frac{\xi^2}{4}}d\xi.\notag
\end{align}
Thus, it is sufficient  to show 
\begin{align}
\begin{aligned}
\intde \xi^{k+1} e^{-\frac{\xi^2}{4}} g_\xi^4\,  d \xi&\leq C \delta^{\frac{1}{100}}(\tau) \intde \xi^{k-1} e^{-\frac{\xi^2}{4}}  \, g^2 d \xi
+ C \exp \left(-\tfrac{1}{8} \delta^{-\frac{1}{50}}\right).
\end{aligned}
\end{align}
We observe that 
\begin{align}
\begin{aligned}
&\intde\xi^{k+1} e^{-\frac{\xi^2}{4}} g_{\xi}(\xi, \tau)^4 d \xi  +3\intde\xi^{k+1} e^{-\frac{\xi^2}{4}}  g_{\xi}(\xi, \tau)^2 g(\xi, \tau) g_{\xi \xi}(\xi, \tau) d \xi \\
& -\frac{1}{2}\intde\xi^{k+2} e^{-\frac{\xi^2}{4}} g_{\xi}(\xi, \tau)^3 g(\xi, \tau) d \xi  + (k+1)\intde \xi^{k} e^{-\frac{\xi^2}{4}} g_{\xi}(\xi, \tau)^3 g(\xi, \tau) d \xi \\
& =\intde \frac{\partial}{\partial \xi}\left(\xi^{k+1} e^{-\frac{\xi^2}{4}}  g_{\xi}(\xi, \tau)^3 g(\xi, \tau)\right) d \xi  \leq C \exp \left(-\tfrac{1}{8} \delta^{-\frac{1}{50}}\right)
\end{aligned}
\end{align}
for $\tau\leq \bar\tau$, where in the last inequality we use Lemma \ref{g-C01}. Using lemma \ref{g-C23}, we obtain 
\begin{align}
\begin{aligned}
& -3\intde\xi^{k+1} e^{-\frac{\xi^2}{4}} g_{\xi}(\xi, \tau)^2 g(\xi, \tau) g_{\xi \xi}(\xi, \tau) d \xi  \\
&\leq C \delta^{\frac{1}{100}}(\tau)\intde \xi^{k} e^{-\frac{\xi^2}{4}} g_{\xi}(\xi, \tau)^2|g(\xi, \tau)| d \xi\notag
\end{aligned}
\end{align}
for $\tau\leq \bar\tau$. Moreover, using Lemma \ref{g-C01} again, we have 
\begin{align}
\begin{aligned}
& \frac{1}{2}\intde\xi^{k+2} e^{-\frac{\xi^2}{4}} g_{\xi}(\xi, \tau)^3 g(\xi, \tau) d \xi  \leq C \delta^{\frac{1}{100}}(\tau)\intde \xi^{k} e^{-\frac{\xi^2}{4}} g_{\xi}(\xi, \tau)^2|g(\xi, \tau)| d \xi
\end{aligned}
\end{align}
for $\tau\leq \bar\tau$. Then by Lemma \ref{g-C01},
\begin{align}
\begin{aligned}
&-(k+1)\intde \xi^{k} e^{-\frac{\xi^2}{4}} g_{\xi}(\xi, \tau)^3 g(\xi, \tau) d \xi\\
&\leq C \max_{\left\{ \xi \leq \delta(\tau)^{-\frac{1}{100}}\right\}}|g_{\xi}(\xi,\tau)| \, \intde \xi^{k} e^{-\frac{\xi^2}{4}} g_{\xi}(\xi, \tau)^2 |g(\xi, \tau)| d \xi\\
&\leq C \delta^{\frac{1}{100}}(\tau)\intde \xi^{k} e^{-\frac{\xi^2}{4}} g_{\xi}(\xi, \tau)^2|g(\xi, \tau)| d \xi
\end{aligned}
\end{align}
for $\tau\leq \bar\tau$. Adding these inequalities together yields 
\begin{align}
\begin{aligned}
&\intde\xi^{k+1} e^{-\frac{\xi^2}{4}} g_{\xi}(\xi, \tau)^4 d \xi \\
&\leq C \delta^{\frac{1}{100}}(\tau)\intde \xi^{k} e^{-\frac{\xi^2}{4}} g_{\xi}(\xi, \tau)^2\, |g(\xi, \tau)| d \xi + C \exp \big(-\frac{1}{8} \delta(\tau)^{-\frac{1}{50}}\big)\\&\leq C \delta^{\frac{1}{100}}(\tau)\intde \xi^{k+1} e^{-\frac{\xi^2}{4}} g_{\xi}(\xi, \tau)^4 d \xi+ C \delta^{\frac{1}{100}}(\tau)\intde \xi^{k-1} e^{-\frac{\xi^2}{4}}  g(\xi, \tau)^2 d \xi\\
&+ C \exp \big(-\frac{1}{8} \delta(\tau)^{-\frac{1}{50}}\big).
\end{aligned}
\end{align}
This implies the Lemma \ref{fequation-error_7term}.
\end{proof}

\begin{lem}\label{k-3gxig} 
For all $\tau \leq \bar \tau$, we have 
\begin{align}\label{g-k-3-integral}
\intde \xi^{k-3} e^{-\frac{\xi^2}{4}}\, g_{\xi}(\xi, \tau)^4 d \xi \leq& C\,  \delta(\tau)^{\frac{1}{200}}\intde\xi^{k-1} e^{-\frac{\xi^2}{4}} \, g(\xi, \tau)^2 d \xi \notag\\
& +C \exp \big(-\frac{1}{8} \delta(\tau)^{-\frac{1}{50}}\big).
\end{align}
\end{lem}
\begin{proof}
Let us consider the function 
\begin{equation}
   \mathfrak{h}(\xi, \tau)=g(\xi, \tau)-\frac{\xi^2 g_{\xi\xi}(0, \tau)}2, \quad 
    \mathfrak{h}_{\xi}(0, \tau)=\mathfrak{h}_{\xi\xi}(0, \tau)=0. 
\end{equation}
We observe that 
\begin{align}
\begin{aligned}
&\intde\xi^{k-3} e^{-\frac{\xi^2}{4}} \mathfrak{h}_{\xi}(\xi, \tau)^4 d \xi \\
& +3\intde\xi^{k-3} e^{-\frac{\xi^2}{4}}  \mathfrak{h}_{\xi}(\xi, \tau)^2 \mathfrak{h}(\xi, \tau) \mathfrak{h}_{\xi \xi}(\xi, \tau) d \xi \\
& -\frac{1}{2}\intde\xi^{k-2} e^{-\frac{\xi^2}{4}} \mathfrak{h}_{\xi}(\xi, \tau)^3 \mathfrak{h}(\xi, \tau) d \xi \\
& + (k-3)\intde \xi^{k-4} e^{-\frac{\xi^2}{4}} \mathfrak{h}_{\xi}(\xi, \tau)^3 \mathfrak{h}(\xi, \tau) d \xi \\
& =\intde \frac{\partial}{\partial \xi}\left(\xi^{k-3} e^{-\frac{\xi^2}{4}}  \mathfrak{h}_{\xi}(\xi, \tau)^3 \mathfrak{h}(\xi, \tau)\right) d \xi \\
& \leq C \exp \big(-\frac{1}{8} \delta(\tau)^{-\frac{1}{50}}\big)
\end{aligned}
\end{align}
for $\tau\leq \bar\tau$, where the last inequality we use  Lemma \ref{g-C01} and Lemma \ref{g-C23}. Using Lemma \ref{g-C23} again, we obtain 
\begin{align}
\begin{aligned}
& 3\, \Big|\intde \xi^{k-3} e^{-\frac{\xi^2}{4}} \mathfrak{h}_{\xi}(\xi, \tau)^2 \mathfrak{h}(\xi, \tau) \,\mathfrak{h}_{\xi \xi}(\xi, \tau) d \xi \Big|\\
& = 3 \, \Big|\intde \xi^{k-2} e^{-\frac{\xi^2}{4}} \mathfrak{h}_{\xi}(\xi, \tau)^2 \mathfrak{h}(\xi, \tau) \frac{1}{\xi}\mathfrak{h}_{\xi \xi}(\xi, \tau) d \xi\, \Big| \\
& \leq C\sup_{\{\xi\leq \delta(\tau)^{-\frac{1}{100}}\}} |\mathfrak{h}_{\xi\xi\xi}(\xi, \tau)|
\intde\xi^{k-2} e^{-\frac{\xi^2}{4}} |\mathfrak{h}_{\xi}(\xi, \tau)|^2|\mathfrak{h}(\xi, \tau)| d \xi\\
& \leq C \delta^{\frac{1}{100}}(\tau)\intde \xi^{k-2} e^{-\frac{\xi^2}{4}} |\mathfrak{h}_{\xi}(\xi, \tau)|^2|\mathfrak{h}(\xi, \tau)| d \xi
\end{aligned}
\end{align}
for $\tau\leq \bar\tau$. Moreover, using Lemma \ref{g-C01} and Lemma \ref{g-C23} again, we have 
\begin{align}
\begin{aligned}
& \frac{1}{2}\intde\xi^{k-2} e^{-\frac{\xi^2}{4}} \mathfrak{h}_{\xi}(\xi, \tau)^3 \mathfrak{h}(\xi, \tau) d \xi \\
& \leq C \delta^{\frac{1}{100}}(\tau)\intde \xi^{k-2} e^{-\frac{\xi^2}{4}} |\mathfrak{h}_{\xi}(\xi, \tau)|^2|\mathfrak{h}(\xi, \tau)| d \xi
\end{aligned}
\end{align}
for $\tau\leq \bar\tau$. Then by Lemma \ref{g-C01} and Lemma \ref{g-C23} we have
\begin{align}
\begin{aligned}
&(k-3)\, \Big|\,\intde\xi^{k-4} e^{-\frac{\xi^2}{4}} \mathfrak{h}_{\xi}(\xi, \tau)^3 \mathfrak{h}(\xi, \tau) d \xi\, \Big|\\
& = (k-3)\, \Big|\intde\xi^{k-2} e^{-\frac{\xi^2}{4}} \frac{\mathfrak{h}_{\xi}(\xi, \tau)}{\xi^2}\mathfrak{h}_{\xi}(\xi, \tau)^2 \mathfrak{h}(\xi, \tau) d \xi\, \Big|\\
&\leq C \max_{\left\{ \xi  \leq \delta(\tau)^{-\frac{1}{100}}\right\}}|\mathfrak{h}_{\xi\xi\xi}(\xi,\tau)|\, \intde \xi^{k-2} e^{-\frac{\xi^2}{4}} \mathfrak{h}_{\xi}(\xi, \tau)^2 |\mathfrak{h}(\xi, \tau)| d \xi\\
&\leq C \delta^{\frac{1}{100}}(\tau)\intde \xi^{k-2} e^{-\frac{\xi^2}{4}} \mathfrak{h}_{\xi}(\xi, \tau)^2|\mathfrak{h}(\xi, \tau)| d \xi
\end{aligned}
\end{align}
for $\tau\leq \bar\tau$. Adding these inequality together gives 
\begin{align}
\begin{aligned}
 & \intde\xi^{k-3} e^{-\frac{\xi^2}{4}} \mathfrak{h}_{\xi}(\xi, \tau)^4 d \xi \notag\\
&\leq C \delta^{\frac{1}{100}}(\tau)\intde \xi^{k-2} e^{-\frac{\xi^2}{4}} |\mathfrak{h}_{\xi}(\xi, \tau)|^2|\mathfrak{h}(\xi, \tau)| d \xi + C \exp \big(-\frac{1}{8} \delta(\tau)^{-\frac{1}{50}}\big)\\
&  \leq C \delta^{\frac{1}{100}}(\tau)\intde e^{-\frac{\xi^2}{4}} \left(\xi^{k-3}\mathfrak{h}_{\xi}(\xi, \tau)^4 + \xi^{k-1}|\mathfrak{h}(\xi, \tau)|^2\right) d \xi+ C \exp \big(-\frac{1}{8} \delta(\tau)^{-\frac{1}{50}}\big)
\end{aligned}
\end{align}

After rearranging terms we obtain
\begin{align}
\begin{aligned}
&\intde\xi^{k-3} e^{-\frac{\xi^2}{4}} \mathfrak{h}_{\xi}(\xi, \tau)^4 d \xi  \\
&\leq C \delta^{\frac{1}{100}}(\tau)\intde \xi^{k-1}e^{-\frac{\xi^2}{4}}  |\mathfrak{h}(\xi, \tau)|^2 d \xi+ C \exp \big(-\frac{1}{8} \delta(\tau)^{-\frac{1}{50}}\big)
\end{aligned}
\end{align}
Using this and  the definition of $\mathfrak{h}$,  we obtain 
\begin{align}\label{k-3g_xi4estimates}
\begin{aligned}
&\quad\intde\xi^{k-3} e^{-\frac{\xi^2}{4}}\left|g_{\xi}(\xi, \tau)\right|^4  d \xi\\
&\leq C\intde\xi^{k-3} e^{-\frac{\xi^2}{4}}\left|\mathfrak{h}_{\xi}\right|^4  d \xi+ C\intde\xi^{k-3} e^{-\frac{\xi^2}{4}}\left|\xi g_{\xi\xi}(0, \tau)\right|^4  d \xi\\
&\leq C\delta^{\frac{1}{100}}(\tau)\intde \xi^{k-1} e^{-\frac{\xi^2}{4}}\left|\mathfrak{h}\right|^2 d \xi+ C\intde\xi^{k-3} e^{-\frac{\xi^2}{4}}\left|\xi g_{\xi\xi}(0, \tau)\right|^4  d \xi \\
&+C \exp \big(-\frac{1}{8} \delta(\tau)^{-\frac{1}{50}}\big)\\
&\leq C\delta^{\frac{1}{100}}(\tau)\intde\xi^{k-1} e^{-\frac{\xi^2}{4}}\left|g\right|^2 d \xi+ C\intde\xi^{k-3} e^{-\frac{\xi^2}{4}}\left|\xi g_{\xi\xi}(0, \tau)\right|^4  d \xi\\
&+C\delta^{\frac{1}{100}}(\tau)\intde \xi^{k-1} e^{-\frac{\xi^2}{4}}\, \big | \xi^2 g_{\xi\xi}(0, \tau)\big |^2 d \xi+ C \exp \big(-\frac{1}{8} \delta(\tau)^{-\frac{1}{50}}\big)\\
&\leq C\delta^{\frac{1}{100}}(\tau)\intde\xi^{k-1} e^{-\frac{\xi^2}{4}}\left|g\right|^2 d \xi+C\delta^{\frac{1}{100}}(\tau)|g_{\xi\xi}(0, \tau)|^2 + C \exp \big(-\frac{1}{8} \delta(\tau)^{-\frac{1}{50}}\big)\\
\end{aligned}
\end{align}
where we have  bounded   $|g_{\xi\xi}(0,\tau)|^4$ by $|g_{\xi\xi}(0,\tau)|^2$. 
In order to conclude the proof of Lemma we need the following claim.\begin{claim}\label{lem error gxixi gamma} 
    \begin{equation}\label{error gxixi gamma}
    \delta(\tau)^{\frac{1}{100}}|g_{\xi\xi}(0, \tau)|^{2}\leq C\delta(\tau)^{\frac{1}{200}}\intde \xi^{k-1} e^{-\frac{\xi^2}{4}}|g(\xi, \tau)|^2 d \xi 
\end{equation}
\end{claim}
\begin{proof}
Let ${\bf{g}}({{\bf y}}, \tau)=g(|{{\bf y}}|, \tau)$. By  $k$-dimensional Gagliardo-Nirenberg interpolation inequality on $B(0, 1)\subset \mathbb{R}^k$ and similar Gaussian weight comparison on $B(0, 1)$ as in proof of \cite[Lem 3.10]{Bre-3d-noncpt} we have
\begin{equation}\label{GNinterpolation}
\begin{split}
\|D^{2}{\bf{g}}\|_{L^{\infty}(B(0, 1))}
&\leq
C_{n,m,r}
\|{\bf{g}}\|_{W^{m, r}(B(0, 1))}^{\theta}
\|{\bf{g}}\|_{L^{2}(B(0, 1))}^{1-\theta}\\&\leq 
C_{n,m,r}
\|{\bf{g}}\|_{W^{m, r}(B(0, 1))}^{\theta}
\|{\bf{g}}\|_{\mathcal{H}(B(0, \delta(\tau)^{-\frac{1}{100}})}^{1-\theta}
\end{split} 
\end{equation}  
with 
\[
\theta\;=\;
\frac{\tfrac12+\tfrac{2}{k}}
     {\tfrac{m}{k}+\tfrac12-\tfrac1r},
\qquad
\frac{2}{m}\le\theta<1, \qquad m>\frac{k}{r}+2.
\]
The above interpolation inequality,   Lemma \ref{g-C-m}, and the standard interpolation inequality,   yield  
\begin{align}\label{g-0}
    |g(0, \tau)|\leq C\left(\intde \xi^{k-1} e^{-\frac{\xi^2}{4}}|g(\xi, \tau)|^2 d \xi \right)^\frac{1}{4}
\end{align}
and by \eqref{GNinterpolation},
\begin{equation}\label{g-xixi-0}
    |g_{\xi\xi}(0, \tau)|\leq C \left(\intde \xi^{k-1} e^{-\frac{\xi^2}{4}}|g(\xi, \tau)|^2 d \xi \right)^{\frac{1-\theta}{2}}. 
\end{equation}
Then,  for $\theta>0$ sufficiently small  (i.e. $m$ large), by \eqref{delta-def}, \eqref{g-0} and \eqref{g-xixi-0}  we have
\begin{align}\label{eq-bound}
\begin{aligned}
    \delta^{1/200}&\leq |g(0, \tau)|^{1/200}+|g_{\xi\xi}(0, \tau)|^{1/200}\\
    &\leq C\left(\!\intde \xi^{k-1} e^{-\frac{\xi^2}{4}}|g(\xi, \tau)|^2 d \xi \!\right)^{\frac{1}{800}}\!\!\!+\!C\!\left(\!\intde \xi^{k-1} e^{-\frac{\xi^2}{4}}|g(\xi, \tau)|^2 d \xi \!\right)^{\frac{1-\theta}{400}}\\
    &\leq \left(\intde \xi^{k-1} e^{-\frac{\xi^2}{4}}|g(\xi, \tau)|^2 d \xi \right)^{\theta}. 
    \end{aligned}
\end{align}
Therefore,  by \eqref{g-xixi-0}, we have
\begin{align}
\delta^{\frac{1}{100}}|g_{\xi\xi}(0, \tau)|^{2}\leq C\delta^{\frac{1}{100}}\left(\intde \xi^{k-1} e^{-\frac{\xi^2}{4}}|g(\xi, \tau)|^2 d \xi \right)^{1-\theta}.
\end{align}
Moreover, by \eqref{eq-bound}, we have
\begin{align}
\begin{aligned}
\delta^{\frac{1}{100}}|g_{\xi\xi}(0, \tau)|^{2}&\leq C\, \delta^{\frac{1}{200}} \cdot \delta^{\frac{1}{200}}\left(\intde \xi^{k-1} e^{-\frac{\xi^2}{4}}|g(\xi, \tau)|^2 d \xi \right)^{1-\theta}\\
&\leq C\delta^{\frac{1}{200}}\left(\intde \xi^{k-1} e^{-\frac{\xi^2}{4}}|g(\xi, \tau)|^2 d \xi \right)^{1-\theta+\theta}\\
& = C\delta^{\frac{1}{200}}\left(\intde \xi^{k-1} e^{-\frac{\xi^2}{4}}|g(\xi, \tau)|^2 d \xi \right)
\end{aligned}
\end{align}
completing the proof of the claim. 
\end{proof} 

The desired estimate \eqref{g-k-3-integral} now readily follows by  \eqref{k-3g_xi4estimates} and Claim \ref{lem error gxixi gamma}.  This completes
the proof of the lemma. \end{proof}

\begin{lem}\label{fequation-error_8term}
We have 
\begin{align}
\begin{aligned}
&\intde \xi^{k-3}e^{-\frac{\xi^2}{4}}\left|-(n-k)(f_\xi(\xi,\tau) + 1) \frac{g_{\xi}^2(\xi,\tau)}{(g(\xi,\tau) + \sqrt{2(n-k-1)})^2}\right|^2 d\xi \\
&\leq C \delta^{\frac{1}{200}}(\tau)\intde \xi^{k-1} e^{-\frac{\xi^2}{4}}|g(\xi, \tau)|^2 d \xi + C \exp \big(-\frac{1}{8} \delta(\tau)^{-\frac{1}{50}}\big).
\end{aligned}
\end{align}
\end{lem}

\begin{proof}
This lemma follows directly by  Lemma \ref{apriori-f} and Lemma \ref{k-3gxig}.
\end{proof}

Combining all lemmas above, we obtain the following estimate for $E_2$.

\begin{prop}\label{Error-E3-estimate}
For any given $\epsilon>0$, there exists a uniform $C_\epsilon<\infty$, such that 
\begin{align}
\begin{aligned}
&\quad\intde \xi^{k-3}e^{-\frac{\xi^2}{4}}|E_{2}|^2 d\xi \\
&\leq \epsilon\intde \xi^{k-3}e^{-\frac{\xi^2}{4}}|h|^2 d\xi + C\delta^{\frac{1}{100}}(\tau)\intde \xi^{k-1}e^{-\frac{\xi^2}{4}}|g|^2 d\xi\\
& + C \exp \big(-\frac{1}{8} \delta(\tau)^{-\frac{1}{50}}\big)
\end{aligned}
\end{align}
for all $\tau\leq \bar\tau$. 
\end{prop}
\begin{proof}
Recalling that $h=f_{\xi}$,  this proposition immediately follows by sequence of lemmas \ref{fequation-error_1term}, \ref{fequation-error_2term}, \ref{fequation-error_34term}, \ref{fequation-error_5term}, \ref{fequation-error_6term}, \ref{fequation-error_7term} and \ref{fequation-error_8term}. 

\end{proof}

In addition, we also give the $L^{\infty}$ boundedness estimates of $E_2$ in compact region.
\begin{prop}
    We have for any fixed $R>0$,
    \begin{equation}
        \|E_2\|_{L^{\infty}([0, R])}\leq C(R, k)
    \end{equation}
\end{prop}
\begin{proof}
We recall the $h=f_{\xi}$ and the definition of $E_2$  in \eqref{e3hg}. 
\begin{align}\label{h-error-expansion}
    E_{2}&=(k-3) {f_{\xi \xi }(\xi ,\tau)}\Big[\frac{1}{f(\xi ,\tau)+\xi }-\frac{1}{\xi}\Big] \\
    &-(2k-4){f_{\xi}(\xi,\tau)}\Big[\frac{f_{\xi}+1}{(f(\xi,\tau)+\xi)^2}-\frac{1}{\xi^2}\Big] \notag\\
    &-\frac{2f_{\xi}f_{\xi\xi}}{f+\xi}+\frac{f^2_{\xi}(f_{\xi}+1)}{(f+\xi)^2}-(k-1)f_{\xi\xi}(\xi,\tau)\int_0^\xi \frac{f_{\eta}(\eta,\tau)(f_{\eta}(\eta,\tau)+1)}{(f(\eta,\tau)+\eta)^2} d\eta\notag\\
    &-(k-1)f_\xi(\xi,\tau)\frac{f_{\xi}(\xi,\tau)(f_{\xi}(\xi,\tau)+1)}{(f(\xi,\tau)+\xi)^2} \notag\\
    &-(n-k)f_{\xi\xi}(\xi,\tau)\int_0^\xi \frac{g_{\eta}^2(\eta,\tau)}{(g(\eta,\tau) + \sqrt{2(n-k-1)})^2} d\eta  \notag\\
    & -(n-k)(f_\xi(\xi,\tau) + 1) \frac{g_{\xi}^2(\xi,\tau)}{(g(\xi,\tau) + \sqrt{2(n-k-1)})^2} \notag.
\end{align}
The detailed estimates follow from term by term estimates as proved in Lemma \ref{x-2E3Calpha}.

\end{proof}

\subsection{Spectral ODE and spectral alternatives} 
We recall that $\mathcal{H}=L^2(\xi^{k-1}e^{-\frac{\xi^2}{4}})$ and $\mathscr{H}=L^2(\xi^{k-3}e^{-\frac{\xi^2}{4}})$. 
For the $k$-dimensional radial Ornstein-Uhlenbeck operator below  acting on $\mathcal{H}$
\begin{equation}
   \mathcal{L}_1g= g_{\xi\xi}+\frac{k-1}{\xi}g_{\xi} - \frac{1}{2}\xi g_\xi + g,
\end{equation}
the spectrum is 
\begin{equation}
\sigma(\mathcal{L}_1)=\{1-m\}_{m=0}^{+\infty},    
\end{equation} the positive and neutral mode eigenspaces
\begin{equation}
    \mathcal{H}_{+}=\textrm{span}\{1\}, \qquad   \mathcal{H}_{0}=\textrm{span}\{\xi^2-2k\}, 
\end{equation}
and for $\mathcal{L}_2$ below acting on $\mathscr{H}$
\begin{align}
\mathcal{L}_2 h &= h_{\xi\xi} - \frac{1}{2}\xi h_\xi +\frac{k-3}{\xi}h_\xi - \frac{2k-4}{\xi^2}h
\end{align} 
by \eqref{negative spectrum of L3}
the spectrum is the set of negative integers
\begin{equation}
\sigma(\mathcal{L}_2)=\{-(m+1)\}_{m=0}^{+\infty}.    
\end{equation}
Recall that we defined $\hat{g}$ as in \eqref{hat-fg}. We define similarly $\hat{h}(\xi,\tau) = h(\xi,\tau) \chi(\delta(\tau)^{\frac{1}{100}}\, \xi)$.
Let us now define
\begin{equation}
    \gamma_+(\tau)=\|\mathfrak{p}_+ \hat{g}(\cdot, \tau)\|^2_{\mathcal{H}}\quad \gamma_0(\tau)=\|\mathfrak{p}_0 \hat{g}(\cdot, \tau)\|^2_{\mathcal{H}}\quad \gamma_-(\tau)=\|\mathfrak{p}_- \hat{g}(\cdot, \tau)\|^2_{\mathcal{H}}
\end{equation}
and 
\begin{equation}
\gamma(\tau) = \|\hat{g}(\cdot,\tau)\|^2_{\mathcal H},  \qquad \lambda(\tau)=\|\hat{h}(\cdot, \tau)\|^2_{{\mathscr{H}}}
\end{equation}
Clearly, $\frac{1}{C} \gamma(\tau) \leq \gamma^{+}(\tau)+\gamma^0(\tau)+\gamma^{-}(\tau) \leq C \gamma(\tau)$. Using Lemma \ref{g-C01}, we obtain
$$
\gamma(\tau) \leq C \sup _{\left\{\xi \leq \delta(\tau)^{-\frac{1}{100}} \right\}}|g(\xi, \tau)|^2 \leq C \delta(\tau)^{\frac{1}{4}} .
$$
In particular, $\gamma(\tau) \rightarrow 0$ as $\tau \rightarrow-\infty$.
We first analyze the evolution of $\gamma^{+}(\tau), \gamma^0(\tau)$, and $\gamma^{-}(\tau)$.

\begin{lem}\label{gamma-lambda-diff}
We have
\begin{align}
\begin{aligned}
& \gamma^{+}(\tau-1) \leq c^{-1} \gamma^{+}(\tau)+C \delta(\tau)^{\frac{1}{200}} \sup _{[\tau-1, \tau]} \left(\gamma(\cdot) + \lambda(\cdot)\right) +
C \exp \left(-\tfrac{1}{64} \delta(\tau)^{-\frac{1}{50}}\right), \\
& \left|\gamma^0(\tau-1)-\gamma^0(\tau)\right| \leq C \delta(\tau)^{\frac{1}{200}} \sup _{[\tau-1, \tau]} \left(\gamma(\cdot) + \lambda(\cdot)\right)+C 
\exp \left(-\tfrac{1}{64} \delta(\tau)^{-\frac{1}{50}}\right), \\
& \gamma^{-}(\tau-1) \geq c \gamma^{-}(\tau)-C \delta(\tau)^{\frac{1}{200}} \sup _{[\tau-1, \tau]} \left(\gamma(\cdot) + \lambda(\cdot)\right)-
C \exp \left(-\tfrac{1}{64} \delta(\tau)^{-\frac{1}{50}}\right) \\
& \lambda(\tau-1) \geq c \lambda(\tau)-C \delta(\tau)^{\frac{1}{200}} \sup _{[\tau-1, \tau]} \gamma(\cdot)-C \exp \left(-\tfrac{1}{64} \delta(\tau)^{-\frac{1}{50}}\right) \\
\end{aligned}
\end{align}
where $c>1$ is a constant.
\end{lem}

\begin{proof}
We note that 
\begin{align}
    \|\hat{g}(\cdot, \tau')\|^2_{\mathcal{H}} = \int_{0}^{+\infty}  \xi^{k-1}e^{-\frac{\xi^2}{4}}|\hat{g}(\xi, \tau')|^2 d \xi \leq \sup _{[\tau-1, \tau]} \gamma(\cdot)
\end{align}
and 
\begin{align}
    \|\hat{h}(\cdot, \tau')\|^2_{{\mathscr{H}}} = \int_{0}^{+\infty}  \xi^{k-3}e^{-\frac{\xi^2}{4}}|\hat{f_\xi}(\xi, \tau')|^2 d \xi \leq \sup _{[\tau-1, \tau]} \lambda(\cdot)
\end{align}
for $\tau' \in[\tau-1,\tau]$. Using Proposition \ref{Error-E2-estimate}, Lemma \ref{g-C01}, and Lemma \ref{g-C23}, we obtain
\begin{align}
\begin{aligned}
& \quad\int_{0}^{+\infty} \xi^{k-1} e^{-\frac{\xi^2}{4}}\left|\hat{g}_\tau(\xi, \tau') - \mathcal{L}_1\hat{g}(\xi, \tau')\right|^2 d \xi \\
& \leq C \delta^{\frac{1}{100}}(\tau')\sup _{[\tau-1, \tau]} \left(\gamma(\cdot) + \lambda(\cdot)\right)+C \exp \left(-\tfrac{1}{32} \delta(\tau')^{-\frac{1}{50}}\right)
\end{aligned}
\end{align}
for $\tau' \in[\tau-1,\tau]$. On the other hand, by Proposition \ref{Error-E3-estimate} and Lemma \ref{apriori-f}. we obtain 
\begin{align}
\begin{aligned}\label{hatE_2 error rough}
& \quad\int_{0}^{+\infty} \xi^{k-3} e^{-\frac{\xi^2}{4}}\left|\hat{h}_\tau(\xi, \tau') - \mathcal{L}_2\hat{h}(\xi, \tau')\right|^2 d \xi \\
& \leq \epsilon \lambda(\cdot) +  C \delta^{\frac{1}{100}}(\tau')\sup _{[\tau-1, \tau]} \gamma(\cdot) +C \exp \left(-\tfrac{1}{32} \delta(\tau')^{-\frac{1}{50}}\right).
\end{aligned}
\end{align}
Consequently,
\begin{align}\label{eq-proj-g}
\begin{aligned}
\frac{d}{d \tau'}\left(\|\mathfrak{p}_+ \hat{g}(\cdot, \tau')\|^2_{\mathcal{H}}\right) & \geq \|\mathfrak{p}_+ \hat{g}(\cdot, \tau')\|^2_{\mathcal{H}} - C \delta(\tau')^{\frac{1}{200}}\sup _{[\tau-1, \tau]} \left(\gamma(\cdot) + \lambda(\cdot)\right)-C \exp \left(-\tfrac{1}{64} \delta(\tau')^{-\frac{1}{50}}\right) \\
\left|\frac{d}{d \tau'}\left(\|\mathfrak{p}_0 \hat{g}(\cdot, \tau')\|^2_{\mathcal{H}}\right)\right| & \leq  C \delta(\tau')^{\frac{1}{200}}\sup _{[\tau-1, \tau]} \left(\gamma(\cdot) + \lambda(\cdot)\right)+C \exp \left(-\tfrac{1}{64} \delta(\tau')^{-\frac{1}{50}}\right) \\
\frac{d}{d \tau'}\left(\|\mathfrak{p}_- \hat{g}(\cdot, \tau')\|^2_{\mathcal{H}}\right) & \leq-\|\mathfrak{p}_- \hat{g}(\cdot, \tau')\|^2_{\mathcal{H}}  +C \delta(\tau')^{\frac{1}{200}}\sup _{[\tau-1, \tau]} \left(\gamma(\cdot) + \lambda(\cdot)\right)+C \exp \left(-\tfrac{1}{64} \delta(\tau')^{-\frac{1}{50}}\right)
\end{aligned}
\end{align}
and by $\sigma(\mathcal{L}_2)=\{-(m+1)\}_{m=0}^{+\infty}$ in \eqref{negative spectrum of L3}, we have
\begin{align}\label{eq-proj-h}
\begin{aligned}
\frac{d}{d \tau'}\left(\|\hat{h}(\cdot, \tau')\|^2_{{\mathscr{H}}}\right) & \leq-(1-\epsilon)\|\hat{h}(\cdot, \tau')\|^2_{{\mathscr{H}}}  +C \delta(\tau')^{\frac{1}{200}}\sup _{[\tau-1, \tau]} \gamma(\cdot)+C \exp \left(-\tfrac{1}{64} \delta(\tau')^{-\frac{1}{50}}\right)
\end{aligned}
\end{align}
for $\tau' \in[\tau-1,\tau]$. We now denote  
\begin{align}
\tilde{h}(\xi,\tau) = h(\xi,\tau)\chi(\delta(\tau+1)^{\frac{1}{100}}\xi),\quad\tilde{g}(\xi,\tau) = g(\xi,\tau)\chi(\delta(\tau+1)^{\frac{1}{100}}\xi),
\end{align}
where $\chi$ is given by \eqref{cutoffchidef}.
Integrating \eqref{eq-proj-g} and \eqref{eq-proj-h} over the interval $[\tau-1, \tau]$ and using the monotonicity of $\delta(\tau)\geq \delta(\tau')$ for $\tau'\in [\tau-1, \tau]$, we have
\begin{align}
\begin{aligned}
&\|\mathfrak{p}_+ \hat{g}(\cdot, \tau-1)\|^2_{\mathcal{H}}\leq e^{-1} \|\mathfrak{p}_+ \hat{g}(\cdot, \tau)\|^2_{\mathcal{H}} + C \delta(\tau)^{\frac{1}{200}}\sup _{[\tau-1, \tau]} \left(\gamma(\cdot) + \lambda(\cdot)\right)+C \exp \left(-\tfrac{1}{64} \delta(\tau)^{-\frac{1}{50}}\right)\\
&\left|\|\mathfrak{p}_0 \hat{g}(\cdot, \tau-1)\|^2_{\mathcal{H}} - \|\mathfrak{p}_0 \hat{g}(\cdot, \tau)\|^2_{\mathcal{H}}\right|\leq C \delta(\tau)^{\frac{1}{200}}\sup _{[\tau-1, \tau]} \left(\gamma(\cdot) + \lambda(\cdot)\right)+C \exp \left(-\tfrac{1}{64} \delta(\tau)^{-\frac{1}{50}}\right)\\
&\|\mathfrak{p}_- \hat{g}(\cdot, \tau-1)\|^2_{\mathcal{H}} \geq e \, \|\mathfrak{p}_- \hat{g}(\cdot, \tau)\|^2_{\mathcal{H}} - C \delta(\tau)^{\frac{1}{200}}\sup _{[\tau-1, \tau]} \left(\gamma(\cdot) + \lambda(\cdot)\right)-C \exp \left(-\tfrac{1}{64} \delta(\tau)^{-\frac{1}{50}}\right)
\end{aligned}
\end{align}
and
\begin{align}
\|\hat{h}(\cdot, \tau-1)\|^2_{\mathscr{H}}\geq e^{1-\epsilon}\|\hat{h}(\cdot, \tau)\|^2_{\mathscr{H}} - C \delta(\tau)^{\frac{1}{200}}\sup _{[\tau-1, \tau]}  \lambda(\cdot)-C \exp \left(-\tfrac{1}{64} \delta(\tau)^{-\frac{1}{50}}\right).
\end{align}
Using Lemma \ref{g-C01} and Lemma \ref{apriori-f} and the monotonicity of $\delta(\tau)$, we obtain 
\begin{align}
\begin{aligned}
&\quad  \|\tilde{g}(\cdot, \tau-1)-  \hat{g}(\cdot, \tau-1)\|^2_{\mathcal{H}}\\
& \leq \int_{\frac{1}{2} \delta(\tau)^{-\frac{1}{100}}}^{\delta(\tau-1)^{-\frac{1}{100}}} \xi^{k-1} e^{-\frac{\xi^2}{4}}|g(\xi, \tau-1)|^2 d \xi \leq C \exp \left(-\tfrac{1}{32} \delta(\tau)^{-\frac{1}{50}}\right)
\end{aligned}
\end{align}
and
\begin{align}
\begin{aligned}
&\quad \|\tilde{h}(\cdot, \tau-1)- \hat{h}(\cdot, \tau-1)\|^2_{\mathscr{H}}\\
& \leq \int_{\frac{1}{2} \delta(\tau)^{-\frac{1}{100}}}^{\delta(\tau-1)^{-\frac{1}{100}}} \xi^{k-3} e^{-\frac{\xi^2}{4}}|h(\xi, \tau-1)|^2 d \xi \leq C \exp \left(-\tfrac{1}{32} \delta(\tau)^{-\frac{1}{50}}\right).
\end{aligned}
\end{align}

Putting these facts together, we conclude that
\begin{align}
\begin{aligned}
&\|\mathfrak{p}_+ \tilde{g}(\cdot, \tau-1)\|^2_{\mathcal{H}}\leq e^{-1} \|\mathfrak{p}_+ \hat{g}(\cdot, \tau)\|^2_{\mathcal{H}} + C \delta(\tau)^{\frac{1}{200}}\sup _{[\tau-1, \tau]} \left(\gamma(\cdot) + \lambda(\cdot)\right)+C \exp \left(-\tfrac{1}{64} \delta(\tau)^{-\frac{1}{50}}\right)\\
&\left|\|\mathfrak{p}_0 \tilde{g}(\cdot, \tau-1)\|^2_{\mathcal{H}} - \|\mathfrak{p}_0 \hat{g}(\cdot, \tau)\|^2_{\mathcal{H}}\right|\leq C \delta(\tau)^{\frac{1}{200}}\sup _{[\tau-1, \tau]} \left(\gamma(\cdot) + \lambda(\cdot)\right)+C \exp \left(-\tfrac{1}{64} \delta(\tau)^{-\frac{1}{50}}\right)\\
&\|\mathfrak{p}_- \tilde{g}(\cdot, \tau-1)\|^2_{\mathcal{H}}\geq e \, \|\mathfrak{p}_- \hat{g}(\cdot, \tau)\|^2_{\mathcal{H}} - C \delta(\tau)^{\frac{1}{200}}\sup _{[\tau-1, \tau]} \left(\gamma(\cdot) + \lambda(\cdot)\right)-C \exp \left(-\tfrac{1}{64} \delta(\tau)^{-\frac{1}{50}}\right)
\end{aligned}
\end{align}
and
\begin{align}
\|\tilde{h}(\cdot, \tau-1)\|^2_{\mathscr{H}}\geq e^{1-\epsilon}\|\hat{h}(\cdot, \tau)\|^2_{\mathscr{H}} - C \delta(\tau)^{\frac{1}{200}}\sup _{[\tau-1, \tau]}  \lambda(\cdot)-C \exp \left(-\tfrac{1}{64} \delta(\tau)^{-\frac{1}{50}}\right).
\end{align}
Thus, choosing some constant $1<c<e^{1-\epsilon}$, the assertion follows. This completes the proof of Lemma \ref{gamma-lambda-diff}.

\end{proof}

Now we define
\begin{align}\label{Lambda+Gamma defs}
\Lambda({\tau})=\sup_{\tau'\leq \tau}\lambda(\tau')\quad &\Gamma(\tau)=\sup_{\tau'\leq \tau}\gamma(\tau')\\
 \Gamma_{+}(\tau) &=\sup_{\tau' \leq \tau}\gamma_{+}(\tau')\\
        \Gamma_{0}(\tau)&=\sup_{\tau'\leq \tau}\gamma_{0}(\tau')\\
         \Gamma_{-}(\tau)&=\sup_{\tau'\leq \tau}\gamma_{-}(\tau')
\end{align}
Clearly, $\frac{1}{C}\,  \Gamma(\tau) \leq \Gamma^{+}(\tau)+\Gamma^0(\tau)+\Gamma^{-}(\tau) \leq C\,  \Gamma(\tau)$. It follows from Lemma \ref{g-C01} that $\gamma(\tau) \leq C \delta(\tau)^{\frac{1}{4}}$. Putting these facts together gives $\Gamma(\tau) \leq C \delta(\tau)^{\frac{1}{4}}$. In particular, $\Gamma(\tau) \rightarrow 0$ as $\tau \rightarrow-\infty$. Using Lemma \ref{gamma-lambda-diff}, we obtain
\begin{align}\label{diff-mode}
\begin{aligned}
& \Gamma^{+}(\tau-1) \leq c^{-1} \Gamma^{+}(\tau)+C \delta(\tau)^{\frac{1}{200}}\left( \Gamma(\tau)+\Lambda(\tau)\right)+C \exp \left(-\tfrac{1}{64} \delta(\tau)^{-\frac{1}{50}}\right), \\
& \left|\Gamma^0(\tau-1)-\Gamma^0(\tau)\right| \leq C \delta(\tau)^{\frac{1}{200}} \left( \Gamma(\tau)+\Lambda(\tau)\right)+C \exp \left(-\tfrac{1}{64} \delta(\tau)^{-\frac{1}{50}}\right), \\
& \Gamma^{-}(\tau-1) \geq c \Gamma^{-}(\tau)-C \delta(\tau)^{\frac{1}{200}} \left( \Gamma(\tau)+\Lambda(\tau)\right)-C \exp \left(-\tfrac{1}{64} \delta(\tau)^{-\frac{1}{50}}\right) 
\end{aligned}
\end{align}
and
\begin{align}
\Lambda(\tau-1) \geq c \Lambda(\tau)-C \delta(\tau)^{\frac{1}{200}} \Gamma(\tau)-C \exp \left(-\tfrac{1}{64} \delta(\tau)^{-\frac{1}{50}}\right) .
\end{align}
It follows from standard interpolation inequalities,  \eqref{g-0} and \eqref{g-xixi-0} , that we have $|g(0, \tau)| \leq C \gamma(\tau)^{\frac{1}{4}}$ and $|g_{\xi\xi}(0, \tau)| \leq C \gamma(\tau)^{\frac{1}{4}}$, hence $\sup _{\tau' \leq \tau}|g(0, \tau')| \leq C \Gamma(\tau)^{\frac{1}{4}}$ and $\sup _{\tau' \leq \tau}|g_{\xi\xi}(0, \tau')| 
\leq C \Gamma(\tau)^{\frac{1}{4}}$.  Hence, 
\begin{align}
\delta(\tau)=\sup _{\tau' \leq \tau}\left(|g(0, \tau')|+|g_{\xi\xi}(0, \tau')|\right) \leq C \, \Gamma(\tau)^{\frac{1}{4}}.
\end{align}
Consequently, $\exp \left(-\tfrac{1}{64} \delta(\tau)^{-\frac{1}{50}}\right) \leq C \, \delta(\tau)^5 \leq C\,  \delta(\tau) \Gamma(\tau)$. Putting these facts together, we conclude that
$$
\begin{aligned}
& \Gamma^{+}(\tau-1) \leq c^{-1} \Gamma^{+}(\tau)+C\, \delta(\tau)^{\frac{1}{200}} \left(\Gamma(\tau)+\Lambda(\tau)\right), \\
& \left|\Gamma^0(\tau-1)-\Gamma^0(\tau)\right| \leq C\, \delta(\tau)^{\frac{1}{200}} \left(\Gamma(\tau)+\Lambda(\tau)\right), \\
& \Gamma^{-}(\tau-1) \geq c \Gamma^{-}(\tau)-C\, \delta(\tau)^{\frac{1}{200}} \left(\Gamma(\tau)+\Lambda(\tau)\right),\\
& \Lambda(\tau-1) \geq c \Lambda(\tau)-C \, \delta(\tau)^{\frac{1}{200}} \Gamma(\tau).
\end{aligned}
$$

\begin{prop}\label{discrete-MZ}
We either have $\Gamma^0(\tau)+\Gamma^{-}(\tau) + \Lambda( \tau) \leq o(1) \Gamma^{+}(\tau)$, or $\Gamma^{+}(\tau)+ \Gamma^{-}(\tau) + \Lambda(\tau) 
\leq o(1) \Gamma^0(\tau)$,  as $\tau \rightarrow-\infty$.

\end{prop}

\begin{proof}
 By definition, the function $\Gamma^{-}(\cdot)$ and $\Lambda(\cdot)$ are monotone increasing. This implies 
\begin{align}
\Gamma^{-}(\bar{\tau}) \geq \Gamma^{-}(\bar{\tau}-1) \geq c \Gamma^{-}(\bar{\tau})-o(1) \left(\Gamma(\bar{\tau})+\Lambda(\bar \tau)\right)
\end{align}
and 
\begin{align}
\Lambda(\bar \tau) \geq \Lambda(\bar \tau-1) \geq c \Lambda(\bar \tau)-o(1) \Gamma(\bar{\tau}).
\end{align}
Thus, $\Lambda(\bar{\tau}) \leq o(1) \Gamma(\bar{\tau})$ and $\Gamma^{-}(\bar{\tau}) \leq o(1) \Gamma(\bar{\tau})$.

Let $I$ denote the set of all positive real numbers $\alpha$ with the property that the set $\left\{\bar{\tau}: \Gamma^0(\bar{\tau})<\alpha \Gamma^{+}(\bar{\tau})\right\}$ is bounded. Moreover, let $J$ denote the set of all positive real numbers $\alpha$ with the property that the set $\left\{\bar{\tau}: \Gamma^0(\bar{\tau}) \geq\right. \left.\alpha \Gamma^{+}(\bar{\tau})\right\}$ is unbounded. Note that $I \subset J$. Then, we claim that
\begin{claim}
If $\alpha \in J$, then $c^{\frac{1}{2}} \alpha \in J$ and $c^{-\frac{1}{2}} \alpha \in I$.    
\end{claim}
\begin{proof}[Proof of claim]
The proof of this claim is  the same as the computation and proof in \cite[Proposition 3.23 ]{ABDS}.
\end{proof}
Using the above claim, we conclude that either $J=\emptyset$ or $I=(0, \infty)$. If $I=(0, \infty)$, it follows that $\Gamma^{+}(\bar{\tau}) \leq o(1) \Gamma^0(\bar{\tau})$ as $\bar{\tau} \rightarrow-\infty$. On the other hand, if $J=\emptyset$, then $\Gamma^0(\bar{\tau}) \leq o(1) \Gamma^{+}(\bar{\tau})$ as $\bar{\tau} \rightarrow-\infty$. This completes the proof of the proposition.
\end{proof}

\section{Ruling out the dominance
of the unstable  modes of $g$}\label{rule out+}

In this section, we will show that the unstable (positive)  modes of  $\mathcal{L}_1$ (defined in \eqref{L2definition}) cannot dominate. Recall that $r_{\max }(t) \geq \sqrt{-2(n-k-1) t}$ for all $t$, where $r_{\max}(t)=G(0, t)=\sup_z G(z, t)$ as defined in \eqref{maximal-radiusrmax}. 

\begin{defn}
Given $0<\alpha<1$, we say that condition $\left(\star_\alpha\right)$ holds if 
\[ \left(\star_\alpha\right) \qquad  \qquad r_{\max }(t) \leq \sqrt{-2(n-k-1) t}\left(1+O(-t)^{-\alpha}\right).\]
\end{defn}

Then, the following estimates hold, and the argument is similar to those in \cite[Section 4]{ABDS} besides some technical estimates for new error term in our parabolic system situation.
\begin{prop}\label{star-Gz}
Suppose that $\left(\star_\alpha\right)$ holds for some $0<\alpha<1$. Then,  $G_z(z, t)^2 \leq C\, (-t)^{-\frac{\alpha}{1-\alpha}}$ holds, 
whenever $G(z, t) \geq \sqrt{-t}$ and $-t$ is sufficiently  large. 
\end{prop}
\begin{proof}
    The proof follows from the same barrier argument as in  \cite[Proposition 4.2]{ABDS} since by \eqref{u-equation}
    \begin{align}
    U_t \leq  UU_{rr} - \frac{1}{2}U_r^2 + \frac{n-k-1-U}{r}U_r + \frac{2(n-k-1)}{r^2}U(1-U), \notag
\end{align}
we can still apply Brendle's barrier $\psi_a$ in Proposition \ref{g-barrier}.
\end{proof}

\begin{cor}\label{diam-alpha}
Suppose that   $\left(\star_\alpha\right)$ holds  for some $0<\alpha<1$. Then,   
 $\operatorname{diam}\left(S^n, g(t)\right) \geq \frac{1}{C}(-t)^{\frac{1}{2(1-\alpha)}}$ holds, provided  $-t$ is sufficiently large.
\end{cor}
\begin{proof}
By argument in \cite[Corollary 4.3]{ABDS}, Corollary \ref{diam-alpha} follows from Proposition \ref{star-Gz}.
\end{proof}

\begin{lem}\label{Gamma-alpha}
Suppose $\left(\star_\alpha\right)$ holds for some $0<\alpha\leq\frac{3}{4}$. Then we have 
\begin{align}\label{eq-gamma-decay}
\Gamma(\tau) \leq C\,  \delta(\tau)^{\frac{1}{4}}\leq C(-t)^{-\frac{\alpha}{16(1-\alpha)}}
\end{align}
when $-t$ is sufficiently large.
\end{lem}

\begin{proof}
Since $\left(\star_\alpha\right)$ holds for some $0<\alpha\leq\frac{3}{4}$, we have $g(0, \tau)\leq C(-t)^{-\alpha}\leq C(-t)^{-\frac{\alpha}{4(1-\alpha)}}$, by Proposition \ref{star-Gz}, we have $\left|g_\xi(\xi, \tau)\right| \leq C(-t)^{-\frac{\alpha}{2(1-\alpha)}}$ whenever $G(z, t) \geq \sqrt{-t}$. By the same argument as Lemma \ref{g-C01}, \ref{g-C-m}, \ref{g-C23}, we have 
\begin{align}
    |g(0,\tau)|+|g_{\xi\xi}(0,\tau)|\leq  C(-t)^{-\frac{\alpha}{4(1-\alpha)}}.
\end{align}
Recall that the definition of $\delta(\tau)$ in \eqref{delta-def}, we get $\delta(\tau)\leq (-t)^{-\frac{\alpha}{4(1-\alpha)}}$. Then we have 
\begin{align}
\Gamma(\tau) \leq C\,  \delta(\tau)^{\frac{1}{4}}\leq C(-t)^{-\frac{\alpha}{16(1-\alpha)}}.
\end{align}
\end{proof}

\begin{prop}\label{F-derivative-estimate}
Suppose $\left(\star_\alpha\right)$ holds for some $0<\alpha\leq \frac{3}{4}$. Then we have 
\begin{equation}
    \sqrt{-t}\left|F_{zz}(z, t)\right| + (-t)\left|F_{zzz}(z, t)\right|  \leq C(R)(-t)^{-\frac{\alpha}{32(1-\alpha)}}
\end{equation}
whenever $z \leq R\sqrt{-t}$ with constant $0<R<+\infty$ and $-t$ is sufficiently large.  Moreover, we have 
\begin{align}
\left|\frac{\partial^m}{\partial z^m} {F}(z, t)\right| \leq C(m, R)\left(-t\right)^{-\frac{m-1}{2}}
\end{align}
whenever $z \leq R\sqrt{-t}$.
\end{prop}

\begin{proof}
Recall the definitions of $\Gamma(\tau)$ and $\Lambda(\tau)$ in \eqref{Lambda+Gamma defs}-\eqref{Lambda+Gamma defs}
and the definitions of $h=f_\xi$ and $\hat h$ from the previous section. By Lemma \ref{Gamma-alpha} we have 
\begin{align}
\Gamma(\tau) \leq C\,  \delta(\tau)^{\frac{1}{4}}\leq C(-t)^{-\frac{\alpha}{16(1-\alpha)}}.
\end{align}
By Proposition \ref{discrete-MZ}, we obtain 
\begin{align}\label{h-L2}
\Lambda(\tau)\leq o(1)\Gamma(\tau)\leq C(-t)^{-\frac{\alpha}{16(1-\alpha)}}.
\end{align}
 In particular, in any bounded domain $Q(2R)=[0, 2R]\times [\tau-4R^2, \tau]$ for any very negative $\tau<0$, $h(\xi, \tau)=\hat{h}(\xi, \tau)$ we have
\begin{equation}\label{hL2tdecay}
   \|h\|_{L^{2}_{{\rho}}(Q(2R))}\leq   2R\, {\sup_{s\in[\tau-4R^2,\tau]}}\|\hat{h}(\cdot,s)\|_{L^{2}_{{\rho}}([0, +\infty])}\leq 2R\Lambda(\tau)^{\frac{1}{2}}\leq C(R)(-t)^{-\frac{\alpha}{32(1-\alpha)}}
\end{equation}
where $\rho(\xi)=\xi^{k-3}e^{-\xi^2/4}$ and the $\|\cdot\|_{L^2_\rho(Q(2R))}$ norm is defined in \eqref{weight l2 norm definition}. 

Then, we note that
\begin{equation}
    \hat{h}_\tau(\xi, \tau) =\mathcal{L}_2\hat{h}(\xi, \tau)+\hat{E}_2. 
\end{equation}
{By  \eqref{hatE_2 error rough} we have}
\begin{align}
\begin{aligned}\label{hatE_2 error rough0}
\|\hat{E}_2(\cdot, \tau)\|^2_{L^2_\rho([0, +\infty])}&=\int_{0}^{+\infty} \xi^{k-3} e^{-\frac{\xi^2}{4}}\left|\hat{h}_\tau(\xi, \tau) - \mathcal{L}_2\hat{h}(\xi, \tau)\right|^2 d \xi \\
& \leq \epsilon \, \lambda(\cdot) +  C \delta^{\frac{1}{100}}(\tau)\sup _{[\tau-1, \tau]} \gamma(\cdot) +C \exp \left(-\frac{1}{32} \delta(\tau)^{-\frac{1}{50}}\right).
\end{aligned}
\end{align}
Using {\eqref{eq-gamma-decay}}, \eqref{h-L2},  \eqref{Lambda+Gamma defs} and Merle-Zaag type estimates in Proposition \ref{discrete-MZ}, {as well as the estimates $\exp \left(-\frac{1}{32} \delta({\tau})^{-\frac{1}{50}}\right) \leq C \delta({\tau})^5 \leq C \delta({\tau}) \Gamma({\tau})$, we have}
\begin{align}\label{E-L2}
\|\hat{E}_2(\cdot, \tau)\|^2_{L^2_\rho([0, +\infty])}=\int_{0}^{+\infty} \xi^{k-3} e^{-\frac{\xi^2}{4}}\left|\hat{h}_\tau(\xi, \tau) - \mathcal{L}_2\hat{h}(\xi, \tau)\right|^2 d \xi \leq C(-t)^{-\frac{\alpha}{16(1-\alpha)}}.
\end{align}
Note that in any bounded domain, we have $\hat{h}=h$ and $E_2=\hat{E}_2$. Hence,  we also have
\begin{equation}\label{E-L2Q}
    \|{E}_2\|^2_{L^2_\rho(Q(2R))}=\|\hat{E}_2\|^2_{L^2_\rho(Q(2R))}\leq 2R\, {\sup_{s \in [\tau - 4R^2,\tau]}}\|\hat{E}_2(\cdot, s)\|^2_{L^2_\rho([0, +\infty])} \leq C(R)(-t)^{-\frac{\alpha}{16(1-\alpha)}}. 
\end{equation}
By \eqref{hL2tdecay}, \eqref{E-L2Q} and the parabolic estimates and interpolation estimates summarized in Lemma \ref{decayDlh} with $\theta\in (0, 1)$ small enough,  we have 
\begin{align}
\left|{h}(\xi,s)\right|_{C^m(Q(R))}=\left|\hat{h}(\xi,s)\right|_{C^m(Q(R))}\leq C_m(-t)^{-\frac{\alpha}{32(1-\alpha)}}
\end{align}
where $Q(2R)=[0, 2R]\times [\tau - 4R^2, \tau]$. Since the constants in above estimates are independent of $\tau \ll -1$,  then the proposition follows directly.
\end{proof}

\begin{prop}\label{Laplace-G}
Suppose that $\left(\star_\alpha\right)$ holds for some $0<\alpha<\frac{16}{31}$. Moreover, suppose that $t_0$ is the time such that $r_{\max}(t_0)\geq \sqrt{-2(n-k-1)t_0}$. If $-t_0$ is sufficiently large, then $-\left(G^2\right)_{z z} - \frac{k-1}{z} \left(G^2\right)_{z}\leq C\left(-t_0\right)^{-\left(1+\frac{\alpha^2}{200}\right) \alpha}$ at the point $\left(0, t_0\right)$.
\end{prop}

\begin{proof}
 The proof closely follows the proof of  \cite[Proposition 4.4 ]{ABDS}.
By \eqref{G-equation}, we have
\begin{align}
    G_t = G_{zz} + \frac{k-1}{z}G_z - (n-k-1)\frac{1-G_z^2}{G} + E(G,F) 
\end{align}
where 
\begin{align}
\begin{aligned}
E(G,F) = (k-1)\frac{F_zz - F}{Fz}G_z
- G_z \int_0^{z(t)} (k-1)\frac{F_{zz}}{F}(x,t)+(n-k)\frac{G_{zz}}{G}(x,t) dx.
\end{aligned}
\end{align}
We next consider the parabolic cylinder
\begin{align}
Q:=\left[0,\left(-t_0\right)^{\frac{1+\varepsilon}{2}}\right] \times\left[t_0-\left(-t_0\right)^{1+\varepsilon}, t_0\right] .
\end{align}
In the following, we put $\varepsilon:=\frac{\alpha^2}{100}$. Using the estimate for $G_z(z, t)^2$ in Proposition \ref{star-Gz} and Lemma \ref{r-low}, we obtain
\begin{align}\label{Gstar-low}
{G}(z, t)^2 \geq(-2(n-k-1) t)\left(1-C(-t)^{-\frac{\alpha}{8}}\right)
\end{align}
for all $t \leq t_0$ and all $z \in\left[0,(-t)^{\frac{1+\varepsilon}{2}}\right]$. On the other hand, the condition  $(\star_\alpha)$  gives
\begin{align}\label{Gstar-up}
{G}(z, t)^2 \leq(-2(n-k-1) t)\left(1+C(-t)^{-\alpha}\right)
\end{align}
for all $t \leq t_0$ and all $z$.
Moreover, we define a function $H$ by
\begin{align}\label{Hztdef}
{H}(z, t):=\frac{1}{2} {G}(z, t)^2+(n-k-1)t.
\end{align}
We have ${H}\left(0, t\right) \geq 0$ for all $-t\gg 1$. Moreover, by \eqref{Gstar-low} and \eqref{Gstar-up}, we can find a positive constant $L$ such that
\begin{align}
-L(-t)^{1-\frac{\alpha}{8}} \leq {H}(z, t) \leq L(-t)^{1-\alpha}
\end{align}
in $Q$. In particular,
\begin{align}
-L\left(-2 t_0\right)^{(1+\varepsilon)\left(1-\frac{\alpha}{8}\right)} \leq {H}(z, t) \leq L\left(-2 t_0\right)^{(1+\varepsilon)(1-\alpha)}
\end{align}
in $Q$.
The function ${H}$ satisfies an equation of the form
\begin{align}
{H}_t(z, t)-{H}_{z z}(z, t)-\frac{k-1}{z}H_z(z,t)=S(z, t), 
\end{align}
where the source term $S$ is defined by
\begin{align}
S(z, t) = -(n-k-2){G}_z(z, t)^2 + G \, E(F,G).
\end{align}
{By Proposition  \ref{star-Gz}  we have ${G}_z(z, t)^2 \leq C(-t)^{-\frac{\alpha}{1-\alpha}}$ at each point in $Q$}. We know that $\frac{1}{2}z\leq F\leq z$ on $z\leq L\sqrt{-t}$ for some large $L$ and $-t$ sufficiently large $-t$. We notice that 
{\begin{align}
\left|\frac{F_zz - F}{Fz}\right|\leq \left|\frac{F_{zz}(\theta)}{F}\right|\leq 2  \frac{\left|F_{zz}(\theta)\right|}{z}  \le 2\frac{|F_{zzz}(\bar{\theta})| \theta}{z} \le 2 |F_{zzz}(\bar{\theta})|
\end{align}
}
for some $\bar{\theta} \in [0,z]$. By same argument, we also obtain
\begin{align}
\left|\int_0^{z(t)} \frac{F_{zz}}{F}(x,t)dx\right|\leq {2\,z\,|F_{zzz}(\theta')|\,}
\end{align}
for some $\theta' \in [0,z]$. Thus, by Proposition \ref{F-derivative-estimate}, we obtain
\begin{align}
\left|\frac{F_zz - F}{Fz} - \int_0^{z(t)} \frac{F_{zz}}{F}(x,t)dx\right|\leq  C(-t)^{-\frac{1}{2}-\frac{\alpha}{32(1-\alpha)}}
\end{align}
in $Q$.
Since $\alpha\in (0, \frac{16}{31})$, this implies
\begin{align}
\begin{aligned}
|S(z, t)| &\leq C(-t)^{-\frac{\alpha}{1-\alpha}} + C(-t)^{-\frac{1}{2}-\frac{17\alpha}{32(1-\alpha)}} \\
&\leq C(-t_0)^{-\frac{\alpha}{1-\alpha}} + C(-t_0)^{-\frac{1}{2}-\frac{17\alpha}{32(1-\alpha)}}\\
&\leq C(-t_0)^{-\frac{\alpha}{1-\alpha}}
\end{aligned}
\end{align}
at each point in $Q$. Moreover, by Proposition \ref{F-derivative-estimate}, the higher order derivatives of ${F}$ satisfy the estimate $\left|\frac{\partial^m}{\partial z^m} {F}(z, t)\right| \leq C(m)\left(-t_0\right)^{-\frac{m-1}{2}}$ in the parabolic cylinder $\left[0,\left(-t_0\right)^{\frac{1}{2}}\right] \times \left[2 t_0, t_0\right]$. Then by the pointwise curvature derivative estimate of $G$ in proposition \ref{Gmorder estimates}. We have $\left|\frac{\partial^m}{\partial z^m} S(z, t)\right| \leq C(m)\left(-t_0\right)^{-\frac{m}{2}}$ in the parabolic cylinder $\left[0,\left(-t_0\right)^{\frac{1}{2}}\right] \times \left[2 t_0, t_0\right]$. 
Using standard interpolation inequalities, we obtain
\begin{align}\label{S-estimates}
\begin{aligned}
& \left|\frac{\partial}{\partial z} S(z, t)\right| \leq C\left(-t_0\right)^{-\frac{1}{2}-\frac{\alpha}{1-\alpha}+\varepsilon} \\
& \left|\frac{\partial^2}{\partial z^2} S(z, t)\right| \leq C\left(-t_0\right)^{-1-\frac{\alpha}{1-\alpha}+\varepsilon} \\
& \left|\frac{\partial}{\partial t} S(z, t)\right| \leq C\left(-t_0\right)^{-1-\frac{\alpha}{1-\alpha}+\varepsilon}
\end{aligned}
\end{align}
in the parabolic cylinder $\left[0,\left(-t_0\right)^{\frac{1}{2}}\right] \times\left[2 t_0, t_0\right]$.
We now introduce two auxiliary functions ${H}^{(1)}$ and ${H}^{(2)}$ on the parabolic cylinder $Q$. Let ${H}^{(1)}$ denote the solution of the linear heat equation
\begin{align}
{H}_t^{(1)}(z, t)-{H}_{z z}^{(1)}(z, t) -\frac{k-1}{z}{H}_{z}^{(1)}(z, t)=S(z, t)
\end{align}
on $Q$ with Dirichlet boundary condition ${H}^{(1)}=0$ on the parabolic boundary of $Q$. Moreover, let ${H}^{(2)}$ denote the solution of the linear heat equation
\begin{align}
{H}_t^{(2)}(z, t)-{H}_{z z}^{(2)}(z, t) -\frac{k-1}{z}{H}_{z}^{(2)}(z, t)=0
\end{align}
on $Q$ with Dirichlet boundary condition ${H}^{(2)}=L\left(-2 t_0\right)^{(1+\varepsilon)(1-\alpha)}-{H}$ on the parabolic boundary of $Q$.
Clearly, ${H}^{(2)}$ is nonnegative, and we have
\begin{align}\label{H12+=constant}
{H}^{(1)}(z, t)+{H}^{(2)}(z, t)+{H}(z, t)=L\left(-2 t_0\right)^{(1+\varepsilon)(1-\alpha)}
\end{align}
at each point in $Q$. In particular, by \eqref{H12+=constant} we have
\begin{align}
0&={H}_{z z}^{(1)}\left(0, t_0\right)+\frac{k-1}{z}{H}_{z}^{(1)}\left(0, t_0\right)\\\notag
&+{H}_{z z}^{(2)}\left(0, t_0\right)+ \frac{k-1}{z}{H}_{z}^{(2)}\left(0, t_0\right)\\\notag
&+{H}_{z z}\left(0, t_0\right)+ \frac{k-1}{z}{H}_{z}\left(0, t_0\right).
\end{align}
Therefore, in order to estimate $\left|{H}_{z z}\left(0, t_0\right) + \frac{k-1}{z}{H}_{z}\left(0, t_0\right)\right|$, it suffices to  estimate the bounds of $\left|{H}_{z z}^{(1)}\left(0, t_0\right)+\frac{k-1}{z}{H}_{z}^{(1)}\left(0, t_0\right)\right|$ and $\left|{H}_{z z}^{(2)}\left(0, t_0\right)+\frac{k-1}{z}{H}_{z}^{(2)}\left(0, t_0\right)\right|$.
By the same argument in \cite[Propostion 4.4]{ABDS}, we  apply \eqref{S-estimates}, and standard interior estimates for   $H^{(1)}$ estimates and  we also apply Proposition \ref{refine-Laplace-esitimate} to  $H^{(2)}$ estimates (we notice that ${H}^{(2)}
\leq C\left(-t_0\right)^{(1+\varepsilon)\left(1-\frac{\alpha}{8}\right)}$ on the parabolic boundary of $Q$), then we sum up the two estimates and conclude that
\begin{align}
\begin{aligned}
\left|{H}_{z z}\left(0, t_0\right)+\frac{k-1}{z}{H}_{z}\left(0, t_0\right)\right| \leq C\left(-t_0\right)^{-\left(1+\frac{\varepsilon}{2}\right) \alpha}
\end{aligned}
\end{align}
This gives $-\left(G^2\right)_{z z} - \frac{k-1}{z} \left(G^2\right)_{z} \leq C\left(-t_0\right)^{-\left(1+\frac{\varepsilon}{2}\right) \alpha}$ at the point $\left(0, t_0\right)$. This completes the proof of the Proposition \ref{Laplace-G}.

\end{proof}

\begin{cor}\label{refine-max-radius}
Suppose that $\left(\star_\alpha\right)$ holds for some $0<\alpha<\frac{16}{31}$. If $-t$ is sufficiently large, then $-\frac{1}{2} \frac{d}{d t}\left(r_{\max }(t)^2\right) \leq (n-k-1)+C(-t)^{-\left(1+\frac{\alpha^2}{200}\right) \alpha}$.
\end{cor}

\begin{proof}
Consider the point where the radius is maximal. At that point, $G=r_{\max }(t) \geq \sqrt{-2(n-k-1) t}, G_z=0$, and $-\left(G^2\right)_{z z} - \frac{k-1}{z} \left(G^2\right)_{z} \leq C(-t)^{-\left(1+\frac{\alpha^2}{200}\right) \alpha}$ by Proposition \ref{Laplace-G}. Using the evolution equation for $G$, we obtain $-\frac{1}{2(n-k-1)} \frac{\partial}{\partial t}\left(G^2\right) \leq 1+C(-t)^{-\left(1+\frac{\alpha^2}{200}\right) \alpha}$ at the point where the radius is maximal. From this, the assertion follows.
\end{proof}

\begin{cor}\label{iterating-alpha}
Suppose that $\left(\star_\alpha\right)$ holds for some $0<\alpha<\frac{16}{31}$. If $0<\tilde{\alpha}< \min \left\{\left(1+\frac{\alpha^2}{200}\right) \alpha, \frac{16}{31}\right\}$, then $\left(\star_{\tilde{\alpha}}\right)$ holds.
\end{cor}

\begin{proof}
The proof is the same as \cite[Corollary 4.6]{ABDS}.
\end{proof}

After these preparations, we can now rule out the case in which the positive modes dominate. Suppose that the positive modes dominate, that is, by Proposition \ref{discrete-MZ} we have $\Gamma^0({\tau})+\Gamma^{-}({\tau})+ \Lambda( \tau) \leq o(1) \Gamma^{+}({\tau})$. By \eqref{diff-mode} we would have $\Gamma^{+}({\tau}-1) \leq e^{-1} \Gamma^{+}({\tau})+C \delta({\tau})^{\frac{1}{200}} \Gamma^{+}({\tau})$. Iterating this inequality would yield $\Gamma^{+}({\tau}) \leq O\left(e^{\frac{\ln c}{2}{\tau}}\right)$. In particular, $\rho_{\max }(\tau) \leq O\left(e^{\frac{\ln c}{16}\tau}\right)$. Equivalently, $\frac{r_{\max }(t)}{ \sqrt{-2(n-k-1) t}}\leq\left(1+O\left((-t)^{-\frac{\ln c}{16}}\right)\right.$. Therefore, $\left(\star_\alpha\right)$ holds for $\alpha=\frac{\ln c}{16}$. Iterating Corollary \ref{iterating-alpha} finitely many times, we would conclude that $\left(\star_\alpha\right)$ holds for some $\frac{16}{31}<\alpha\leq \frac{3}{4}$. Using Corollary \ref{diam-alpha}, we would  obtain
\begin{align}\label{eq-contra100}
\operatorname{diam}\left(S^n, g(t)\right)>c(-t)^{\frac{1}{2(1-\alpha)}} > c(-t)^{\frac{31}{30}},
\end{align}
if $(-t)$ is sufficiently large. On the other hand, standard estimates for the change of distances under Ricci flow imply that $-\frac{d}{d t} \operatorname{diam}\left(S^n, g(t)\right) \leq C \sqrt{R_{\text {max }}(t)}$. Since $R_{\max }(t)$ is uniformly bounded from above by Hamilton's Harnack inequality, we conclude that
\begin{align}
\limsup _{t \rightarrow-\infty} (-t)^{-1}\, \operatorname{diam}\left(S^n, g(t)\right)<C.
\end{align}
This contradicts \eqref{eq-contra100}. Thus, we have ruled out the case that the positive mode dominates.


\section{Unique asymptotics of $G(z, t)$ and unique dependence of $F(z, t)$}\label{Unique sharp asymptotics of  G}
In this section, we prove Theorem \ref{unique asymptotics theorem} on unique sharp asymptotics of profile of $G(z, t)$ and Theorem \ref{GuniqueimpliesFunique} that the uniqueness of $G(z, t)$ implies the uniqueness of $F(z, t)$  for  $SO(k)\times SO(n-k+1)$-symmetric ancient ovals.

\subsection{Unique sharp asymptotics in the parabolic region}
In view of the preceding discussion, we now focus on the case that neutral mode dominates, that is, when $\Gamma^+(\tau) + \Gamma^-(\tau) + \Lambda(\tau) \leq o(1) \, \Gamma^0(\tau)$. Recall that
\begin{equation}
    ||g||_{\mathcal{H}}^2=\int_{0}^{+\infty} g^2 \xi^{k-1}e^{-\frac{\xi^2}{4}}d\xi.
\end{equation}
Let us define 
\begin{equation}
    ||g||_{\mathcal{D}}^2=\int_{0}^{+\infty} (g^2+g_\xi^2) \xi^{k-1}e^{-\frac{\xi^2}{4}}d\xi.
\end{equation}
Note that $\|g\|_{\mathcal{D}}^2 = \langle g,(2-\mathcal{L})g \rangle_{\mathcal{H}}$ if $g$ is compactly supported.

Recall that the subspace $\mathcal{H}_0$ is one-dimensional and is spanned by the Hermite polynomial $\xi^2-2k$. We consider the projection of the function 
\begin{align}
\hat{g}(\xi,\tau) := g(\xi,\tau) \, \chi(\delta(\tau)^{\frac{1}{100}} \xi)
\end{align}
onto the subspace $\mathcal{H}_0$. More precisely, we write 
\begin{align}
\mathfrak{p}_0(\hat{g}(\xi,\tau)) =  \alpha(\tau) \, (\xi^2-2k),
\end{align}
where 
\begin{align}
\alpha(\tau) := \frac{1}{\|\xi^2-2k\|^2_{\mathcal{H}}} \int_{0}^\infty (\xi^2-2k) \, \hat{g}(\xi,\tau) \, \xi^{k-1} e^{-\frac{\xi^2}{4}} \, d\xi.
\end{align}
Furthermore, we define 
\begin{align}
A(\tau) := \sup_{\tau' \leq \tau} |\alpha(\tau')|.
\end{align}
Clearly, $\frac{1}{C} \, A(\tau)^2 \leq \Gamma^0(\tau) \leq C \, A(\tau)^2$ for some constant $C$. This implies $\frac{1}{C} \, A(\tau)^2 \leq \Gamma(\tau) \leq C \, A(\tau)^2$. Since $\delta(\tau) \leq C \, \Gamma(\tau)^{\frac{1}{8}}$, we conclude that $\delta(\tau) \leq C \, A(\tau)^{\frac{1}{4}}$. Moreover, by Proposition \ref{discrete-MZ}, we have 
$\Lambda(\tau)\leq o(1) \Gamma(\bar\tau) \leq o(1)A(\bar\tau)^2$.

\begin{lem}\label{neutral-dominate-positive}
We have $\|\mathfrak{p}_+ \hat{g}(\xi,\tau)\|_{\mathcal{H}} \leq o(1) \, A(\tau)$.
\end{lem}

\begin{proof}
We have 
\begin{align}\|\mathfrak{p}_+ \hat{g}(\xi,\tau)\|_{\mathcal{H}}^2 \leq \Gamma^+(\tau) \leq o(1) \, \Gamma(\tau) \leq o(1) \, A(\tau)^2.
\end{align}
This proves the assertion.
\end{proof}

\begin{lem}\label{neutral-dominate-negative}
We have $\|\mathfrak{p}_- \hat{g}(\xi,\tau)\|_{\mathcal{H}} \leq C \, \delta(\tau)^{\frac{1}{400}} \, A(\tau)$.
\end{lem}

\begin{proof}
The proof is similar to \cite[Lemma 5.2]{ABDS}.
\end{proof}

\begin{lem}\label{neutral-dominate-negative-D}
We have $\|\mathfrak{p}_- \hat{g}(\xi,\tau)\|_{\mathcal{D}} \leq C \, \delta(\tau)^{\frac{1}{400}} \, A(\tau)$.
\end{lem}

\begin{proof}
Using Proposition \ref{evolution of (f, g)}, Proposition \ref{Error-E2-estimate}, Lemma \ref{derivative-delta} and Lemma \ref{neutral-dominate-negative}, the proof is similar to \cite[Lemma 5.3]{ABDS}.

\end{proof}

\begin{lem}\label{bound-integrals}
We have 
\begin{align}
\int_{0}^\infty  |\hat{g}(\xi,\tau) -  \alpha(\tau) \, (\xi^2-2k)|^2 \, \xi^{k-1} e^{-\frac{\xi^2}{4}}  \, (1+0 \leq \xi)^4 \, d\xi \leq o(1) \, A(\tau)^2
\end{align}
and 
\begin{align}
\int_{0}^\infty  |\hat{g}_\xi(\xi,\tau) -  2 \,\alpha(\tau)\xi|^2 \, \xi^{k-1} e^{-\frac{\xi^2}{4}}  \, (1+0 \leq \xi)^4 \, d\xi \leq o(1) \, A(\tau)^2
\end{align} 
\end{lem}

\begin{proof}
Using Lemma \ref{neutral-dominate-positive} and  Lemma \ref{neutral-dominate-negative-D}, the proof follows similarly as in \cite[Lemma 5.4]{ABDS}. 
\end{proof}



\begin{prop}
\label{projection-E2-neutral}
The function $g$ satisfies $\frac{\partial}{\partial \tau} g(\xi,\tau) = \mathcal{L}_1 g(\xi,\tau) + E_1(\xi,\tau)$, where  
\begin{align}
\intde &\quad E_1(\xi,\tau) \, \frac{\xi^2-2k}{\|\xi^2-2k\|^{2}_{\mathcal{H}}} \, \xi^{k-1}\, e^{-\frac{\xi^2}{4}} \, d\xi \\
&=-\beta_{n, k}^{-1}\alpha(\tau)^2 + o(A(\tau)^2)
\end{align}  
where
\begin{align}
    \beta_{n, k}&=\frac{\sqrt{(n-k-1)}}{4\sqrt{2}}.
\end{align}
\end{prop}

\begin{proof}
Recall that the source term $E_1(\xi,\tau)$ is given by  {\eqref{g-error-expansion}}.
Let us write it as
\begin{align}
\begin{aligned}
    E_1(\xi,\tau) &= -\frac{1}{2\sqrt{2(n-k-1)}} \, g(\xi,\tau)^2 - \frac{1}{\sqrt{2(n-k-1)}} \, g_\xi(\xi,\tau)^2\\
    &+ E_1^1(\xi,\tau) + E^2_1(\xi,\tau) + E_1^3(\xi,\tau) + E_1^4(\xi,\tau)
\end{aligned}
\end{align}
where
\begin{align}
\begin{aligned}
E_1^1(\xi,\tau) &= -g(\xi,\tau) + \frac{1}{2\sqrt{2(n-k-1)}} \, g(\xi,\tau)^2 -   \frac{(n-k-1)}{g(\xi,\tau) + \sqrt{2(n-k-1)})} \\
&+\frac{1}{2}(g(\xi,\tau) + \sqrt{2(n-k-1)}),
\end{aligned}
\end{align}
\begin{align}
E_1^2(\xi,\tau) = \frac{g_\xi^2(\xi,\tau)}{\sqrt{2(n-k-1)}} -\frac{g_\xi^2(\xi,\tau)}{g(\xi,\tau) + \sqrt{2(n-k-1)})},
\end{align}
\begin{align}
E_1^3(\xi,\tau) = -(n-k)g_{\xi}(\xi,\tau)\int_0^\xi \frac{g_{\eta}^2(\eta,\tau)}{(g(\eta,\tau) + \sqrt{2(n-k-1)})^{2}} 
\end{align}
\begin{align}
E_1^4(\xi,\tau) = -(k-1)g_\xi(\xi,\tau)\left( \frac{f(\xi,\tau)}{\xi(f(\xi,\tau)+\xi)} +\int_{0}^{\xi}
     \frac{f_{\eta}(\eta,\tau)(f_{\eta}(\eta,\tau) +1) }{(f(\eta,\tau)  + \eta)^2} d\eta\right).
\end{align}
Using Lemma \ref{bound-integrals} and the fact
\begin{equation}
    \frac{\langle(\xi^2-2k)^2, \xi^2-2k\rangle_{\mathcal{H}}}{\|\xi^2-2k\|^2_{\mathcal{H}}}=8
\end{equation}
 we obtain 
\begin{align} 
\begin{aligned}
&\quad -\|\xi^2-2k\|^{-2}_{\mathcal{H}}\frac{1}{2\sqrt{2(n-k-1)}} \int_{0}^\infty \hat{g}(\xi,\tau)^2 \, (\xi^2-2k)\, \xi^{k-1} \, e^{-\frac{\xi^2}{4}} \, d\xi \\ 
&= -\alpha(\tau)^2  \frac{1}{2\sqrt{2(n-k-1)}}\frac{\langle(\xi^2-2k)^2, \xi^2-2k\rangle_{\mathcal{H}}}{\|\xi^2-2k\|^2_{\mathcal{H}}}+ o(1) \, A(\tau)^2\\
&=-\frac{8}{2\sqrt{2(n-k-1)}} \alpha(\tau)^2 + o(1) \, A(\tau)^2 \\
&=-\frac{4}{\sqrt{2(n-k-1)}} \alpha(\tau)^2 + o(1) \, A(\tau)^2
\end{aligned}
\end{align}
and  by 
\begin{equation}
    \langle \xi^2-2k, 1\rangle_{\mathcal{H}}=0
\end{equation}
\begin{align} 
\begin{aligned}
&\quad-\|\xi^2-2k\|^{-2}_{\mathcal{H}}\frac{1}{\sqrt{2(n-k-1)}} \int_{0}^\infty \hat{g}_\xi(\xi,\tau)^2 \, (\xi^2-2k)\, \xi^{k-1} \, e^{-\frac{\xi^2}{4}} \, d\xi \\ 
&=-\|\xi^2-2k\|^{-2}_{\mathcal{H}}\frac{1}{\sqrt{2(n-k-1)}}  4\alpha(\tau)^2 \int_{0}^\infty  \xi^2\, (\xi^2-2k) \, \xi^{k-1} \, e^{-\frac{\xi^2}{4}} \, d\xi + o(1) \, A(\tau)^2\\
&=-\|\xi^2-2k\|^{-2}_{\mathcal{H}}\frac{1}{\sqrt{2(n-k-1)}} \times   4\|\xi^2-2k\|^{2}_{\mathcal{H}}\, \alpha(\tau)^2\\
&=- \frac{4}{\sqrt{2(n-k-1)}} \, \alpha(\tau)^2 + o(1) \, A(\tau)^2 .
\end{aligned}
\end{align}

Finally, we estimate the terms $E_1^1(\xi,\tau)$, $E_1^2(\xi,\tau)$, $E_1^3(\xi,\tau)$ and $E_1^4(\xi,\tau)$. The term $E_1^1(\xi,\tau)$ satisfies the pointwise estimate $|E_1^1(\xi,\tau)| \leq C \, |g(\xi,\tau)|^3 \leq o(1) \, g(\xi,\tau)^2$ for $0 \leq \xi \leq \delta(\tau)^{-\frac{1}{100}}$. Using Lemma \ref{bound-integrals}, by the similar argument as in the proof of \cite[Proposition 5.6]{ABDS}, we obtain the estimates of $E_1^i$, for $i\in \{1,2,3\}$ below
\begin{align} 
\begin{aligned}
&\bigg | \intde (E_1^1+E_1^2+E_1^3)(\xi,\tau) \, (\xi^2-2k) \, \xi^{k-1} \, e^{-\frac{\xi^2}{4}} \, d\xi \bigg | \leq o(1) \, A(\tau)^2.
\end{aligned}
\end{align}


At last, we analyze the term $E_1^4(\xi, \tau)$. By Lemma \ref{gequation errorf_1}, Lemma \ref{gequation errorf_2} and Lemma \ref{bound-integrals},  We conclude that 
\begin{align}
\begin{aligned}
&\bigg | \intde{E_1^4}(\xi,\tau) \, (\xi^2-2k) \, \xi^{k-1} \, e^{-\frac{\xi^2}{4}} \, d\xi \bigg | \\ 
& \leq \bigg | \intde{E_1^4}(\xi,\tau) \,  \left((\xi^2-2k)-g(\xi,\tau)\right) \, \xi^{k-1} \, e^{-\frac{\xi^2}{4}} \, d\xi \bigg | \\
& + \bigg | \intde{E_1^4}(\xi,\tau) \,   g(\xi,\tau) \, \xi^{k-1} \, e^{-\frac{\xi^2}{4}} \, d\xi \bigg | \\
& \leq o(1)A(\tau)\|\hat{E}^4_1(\xi,\tau)\|_{\mathcal{H}} + \|\hat{g}(\xi,\tau)\|_{\mathcal{H}}\|\hat{E}^4_1(\xi,\tau)\|_{\mathcal{H}}\\
&\leq o(1) \, A(\tau)^2,
\end{aligned}
\end{align} 
where
\begin{align}
\hat{E}_1^4(\xi,\tau) := {E}_1^4(\xi,\tau) \, \chi(\delta(\tau)^{\frac{1}{100}} \xi)
\end{align}
and $\chi$ is given by \eqref{cutoffchidef}.
Putting these facts together, the assertion follows.
\end{proof}
\begin{cor}\label{CorA=a}
    If $-\tau$ is sufficiently negative, then $A(\tau)=|\alpha(\tau)|$.
\end{cor}
\begin{proof}
    The proof is similar as in  \cite[Corollary 5.8]{ABDS}.
\end{proof}
{Then, Proposition \ref{projection-E2-neutral} and Corollary \ref{CorA=a} immediately imply the following corollary.}
\begin{cor}
The spectral function $\alpha(\tau)$ satisfies
\begin{equation}
    \frac{d\alpha(\tau)}{d\tau}=-(\beta_{n, k}^{-1}+o(1))\alpha(\tau)^2,
\end{equation}
where
\begin{align}
    \beta_{n, k}&=\frac{\sqrt{(n-k-1)}}{4\sqrt{2}}.
\end{align}
In particular
\begin{equation}\label{spectral asy}
    \alpha(\tau)=\frac{\sqrt{(n-k-1)}}{4\sqrt{2}\tau}+o(\tau^{-1})
\end{equation}
as $\tau\to -\infty$.
\end{cor}
Then we can obtain unique sharp asymptotics in parabolic region below.
\begin{prop}[Unique sharp asymptotics in parabolic region]\label{convergence-C_loc^infty}
    We have 
    \begin{equation}
        \tau \, g(\xi, \tau)\to \frac{\sqrt{(n-k-1)}}{4\sqrt{2}}(\xi^{2}-2k)\quad \textrm{in}\quad C^{\infty}_{\textrm{loc}}([0, +\infty)).
    \end{equation}
    In particular,  for the profile function $G(z, t)$, as $t\to -\infty$ we have
    \begin{equation}
         G(z,t)^2 = 2(n-k-1)(-t)   - \frac{(n-k-1)(z^2+2kt)}{2\log(-t)}  + o\left(\frac{-t}{\log(-t)}\right) 
    \end{equation}
in $C^{\infty}_{loc}$-sense    for all $0\leq z \leq L\sqrt{-t}$.
\end{prop}
\begin{proof}
    The proof of asymptotics directly follows from \eqref{spectral asy}, the argument in the proof of  \cite[Proposition 5.10]{ABDS}, and transforming everything back to original variables.
\end{proof}

\begin{cor}\label{domain.of.definition.of.G}
The domain of definition of the function $\xi\mapsto g(\xi, \tau) $ is an interval of length at most $o(1)(-\tau)$
\end{cor}

\begin{proof}
 Recall that the function of $g(\xi, \tau)\geq -\sqrt{2(n-k-1)}$ is concave in $\xi$, the assertion follows from Proposition \ref{convergence-C_loc^infty}.
\end{proof}

\subsection{Unique sharp asymptotics in the intermediate region}
We next study the asymptotics in the intermediate region where $z\geq M\sqrt{-t}$ for some large constant $M$, and $G\geq \theta\sqrt{-2(n-k-1)t}$ for some small constant $\theta$. First, using barrier arguments to derive the derivative estimates for $G(z,t)$ at intermediate region.

\begin{lem}\label{GzGtupper bound}
Fix a small number $\theta \in (0,\frac{1}{2})$ and a large number $M \geq 10k$. If $-t$ is sufficiently large (depending on $\theta$ and $M$), then  we have
\begin{align}\label{Gz-intermediate}
G_z(z,t)^2 \leq \frac{M^2+C(\theta)k}{M^2-2k} \cdot  \frac{(n-k-1)}{2 \log(-t)} \cdot  \Big ( \frac{-2(n-k-1)t}{G(z,t)^2} - 1 \Big )
\end{align}
and
\begin{align}\label{Gt-intermediate}
G_t(z,t) \leq -(n-k-1)G(z,t)^{-1} \, \Big ( 1-\frac{C(\theta)k}{\log(-t)} \Big )
\end{align}
whenever $0\leq z  \leq M\sqrt{-t}$ and $G(z,t) \geq \theta \sqrt{-2(n-k-1)t}$. 
\end{lem}

\begin{proof}
The proof of \eqref{Gz-intermediate} is based on applying barrier in Proposition \ref{g-barrier} as in the proof of \cite[Proposition 6.1]{ABDS} up to the change of some constants depending on $n$ and $k$.  The proof of \eqref{Gt-intermediate} is similar  to the proof of \cite[Corollary 6.2]{ABDS} by an argument of differential inequality of $G$. Here, we note that the new issue different from \cite[Corollary 6.2]{ABDS} is that the evolution equation of $G$ in \eqref{G-equation} is coupled with $F$-related terms. To deal with it, we recall that we have nonnegative curvature condition, it follows that 
\begin{align}
    F_{zz}\leq 0,\quad G_{zz}\leq 0, \quad F_zG_z\leq 0 \quad \text{and} \quad -1\leq G_z\leq 0.
\end{align}
at each point in space-time. Using the evolution equation for $G$, we obtain 
\begin{align}
\begin{aligned}\label{Godeinequality}
    G_t &= G_{zz} - (n-k-1)\frac{1-G_z^2}{G} + (k-1)\frac{F_zG_z}{F}\\
    &- G_z \int_0^{z(t)} (k-1)\frac{F_{zz}}{F}(x,t)+(n-k)\frac{G_{zz}}{G}(x,t) dx\\
    &\leq - (n-k-1)\frac{1-G_z^2}{G}
\end{aligned}
\end{align}
at each point in space-time. Then as in \cite[Corollary 6.2]{ABDS}, \eqref{Gt-intermediate} follows from  \eqref{Gz-intermediate} with $M=10k$ and \eqref{Godeinequality}. This completes the proof of Lemma \ref{GzGtupper bound}.

\end{proof}

\begin{prop}[intermediate region asymptotics]\label{intermediate region asymptotics}
Fix a small number $\theta \in(0, \frac{1}{2})$. If $-t$ is sufficiently large (depending on $\theta$), then  $\{z: G(z, t) \geq \theta \sqrt{-2(n-k-1) t}\}$ is an interval $\left[0, \bar{z}(\theta, t)\right]$, and
\begin{align}\label{domain-size-estimate}
\bar{z}(\theta, t)=(2+o(1)) \sqrt{1-\theta^2} \sqrt{(-t) \log (-t)}
\end{align}
Moreover, we have 
\begin{align}\label{G-intermediate-estimate}
G(z, t)^2=-2(n-k-1) t-(n-k-1)\frac{z^2}{2 \log (-t)}+o(-t)
\end{align}
in $C^{\infty}_{loc}$-sense  for $0\leq z \leq 2\sqrt{1-\theta^2} \sqrt{(-t) \log (-t)}$.
\end{prop}

\begin{proof}
Using Proposition \ref{convergence-C_loc^infty}-Lemma \ref{GzGtupper bound},  the proof follows similarly as in the proof of  \cite[Proposition 6.3, Proposition 6.4, Corollary 6.5 and Corollary 6.6]{ABDS}.  
\end{proof}

\subsection{Unique sharp asymptotics in the tip region}
In this subsection, we analyze the asymptotics of the solution near the tip region based on the discussion of intermediate region asymptotics in Proposition \ref{intermediate region asymptotics} and the blowdown analysis in Section \ref{blowdown analysis section}.

For each $t$, the function $z \mapsto G(z,t)$ is defined on the interval $[0,d_{\text{\rm tip}}(t)]$, where $d_{\text{\rm tip}}(t)$ denotes the distance of the fixed point of the $SO(k)$-action $\mathcal{O}$ to tip fiber. We  derive the asymptotics in the tip region below.
\begin{prop}[Tip region asymptotics]\label{Tip region asymptotics}
We have 
\begin{align}\label{dtip}
\lim_{t \to -\infty} \frac{d_{\text{\rm tip}}(t)}{\sqrt{(-t) \, \log(-t)}} = 2.
\end{align}
and
the scalar curvature at each  tip point on $\{G(z, t)=0\}$  satisfies
\begin{align}\label{Rtip}
R_{\text{\rm tip}}(t) = (1+o(1)) \, \frac{\log (-t)}{(-t)}.
\end{align}
Finally, if
we consider any  $g_{i}(t)=R(p_i, t_i)g(t_i +R(p_i, t_i)^{-1}t)$ on $M$ obtained by rescaling  ancient solution $g(t)$ at the tip points on $(p_i, t_i)\in \{G(z, t_i)=0\}$ with $t_i\to -\infty$, then the rescaled solutions $(M, g_i(t))$
converge to the unique Bryant soliton times $\mathbb{R}^{k-1}$.
\end{prop}
\begin{proof}
{
We notice that for any $\theta>0$, by \eqref{domain-size-estimate} and convexity of $G$, we have for $\bar{z}$ defined in Proposition \ref{intermediate region asymptotics},
\begin{align}
\frac{1}{2}\left(\bar z(2\theta, t) + d_{\text{\rm tip}}(t)\right)\leq \bar z(\theta, t) \leq d_{\text{\rm tip}}(t).
\end{align}
Using intermediate region asymptotics from Proposition \ref{intermediate region asymptotics} and taking $\theta\to 0$, we obtain \eqref{dtip} as in \cite[Proposition 7.1]{ABDS}. By Proposition \ref{tip-asymptotics}, we know the sequence of rescaled flows $(M, g_i(t), p_i)$ subsequentially converges to  $\mathbb{R}^{k-1}$ times Bryant soliton with tip scalar curvature $1$. Since all subsequential limits are the same, we obtain full convergence to $\mathbb{R}^{k-1}$ times Bryant soliton. Lastly, since the limit is the Bryant soliton, we have $\int_0^\infty \text{\rm Ric}(\gamma'(s),\gamma'(s)) \, ds = 1$ as in \cite[Lemma 7.2]{ABDS}, when the scalar curvature is normalized to $1$. Using the above facts, Corollary \ref{scalar.curvature.at.tips.2} and \eqref{dtip}, we can obtain  $-\frac{d}{dt} d_{\text{\rm tip}}(t) = (1+o(1)) \, R_{\text{\rm tip}}(t)^{\frac{1}{2}}$  as in \cite[Lemma 7.3]{ABDS}. The proof of  \eqref{Rtip} is now the same as the proof of in \cite[Proposition 7.4 ]{ABDS}. 
}
\end{proof}


{We are now ready to prove Theorem \ref{unique asymptotics theorem}.
\begin{proof}[Proof of Theorem \ref{unique asymptotics theorem}]
    Theorem \ref{unique asymptotics theorem} follows directly from Proposition \ref{convergence-C_loc^infty}, Proposition \ref{intermediate region asymptotics}, Proposition \ref{Tip region asymptotics}.
  \end{proof}}

\subsection{Uniqueness of profile $G(z, t)$ implies uniqueness of profile $F(z, t)$}\label{secGimplyF}
In this subsection, we show that uniqueness of $G(z, t)$ implies uniqueness of $F(z, t)$.
\begin{proof}[Proof of Theorem \ref{GuniqueimpliesFunique}]
    By the equation of $G$ in \eqref{G-equation},  for 
    \begin{equation}
\Phi^{G}(z,t)
:=
\frac{G_t-G_{zz}+(n-k-1)\frac{1-G_z^2}{G}}{G_z} \quad z>0 
    \end{equation}
we have
\begin{equation}\label{Phi-identity}
\Phi^{G}=(k-1)\frac{F_z}{F}-\mathcal J.
\end{equation}
where 
\begin{equation}
    \mathcal J(z,t)=
\int_0^z \left((k-1)\frac{F_{zz}}{F}(\tilde{z},t)+(n-k)\frac{G_{zz}}{G}(\tilde{z},t)\right)\,d\tilde{z}.
\end{equation}
Differentiate \eqref{Phi-identity} to obtain
\begin{align}
    \Phi^{G}_z
&=
(k-1)\frac{F_{zz}}{F}
-(k-1)\frac{F_z^2}{F^2}
-
(k-1)\frac{F_{zz}}{F}
-(n-k)\frac{G_{zz}}{G}\\
&=
-(k-1)\frac{F_z^2}{F^2}
-(n-k)\frac{G_{zz}}{G}.
\end{align}

Consequently,
\begin{align}\label{F-log-derivative-square}
((\log F)_z)^2&=\left(\frac{F_z}{F}\right)^2=Q^G\geq 0
\end{align}
where 
\begin{align}
Q^G&=
-\frac{1}{k-1}
\left(
\Phi^{G}_z+(n-k)\frac{G_{zz}}{G}
\right)\\
&=-\frac{1}{k-1}
\left[
\partial_z\!\left(
\frac{G_t-G_{zz}+(n-k-1)\frac{1-G_z^2}{G}}{G_z}
\right)
+
(n-k)\frac{G_{zz}}{G}\right]
\end{align}
Hence by $F_z/F>0$ from \eqref{gradient-sign}, we have
\begin{equation}\label{F-log-derivative}
\frac{F_z}{F}
=\sqrt{Q^G}=\sqrt{
-\frac{1}{k-1}
\left[
\partial_z\!\left(
\frac{G_t-G_{zz}+(n-k-1)\frac{1-G_z^2}{G}}{G_z}
\right)
+
(n-k)\frac{G_{zz}}{G}
\right]
}.
\end{equation}
By integration from $z_*$ to $z >0$
\begin{align}\label{Fz*t0}
F(z, t)&= F(z_*, t)\exp \left(\int_{z_*}^{z}\sqrt{Q^G(\tilde{z}, t)}d\tilde{z}\right) \\
&=z\frac{F(z_*, t)}{z_*}\exp \left(\int_{z_*}^{z}\left[\sqrt{Q^G(\tilde{z}, t)}-\frac{1}{\tilde{z}}\right]d\tilde{z}\right)\notag
\end{align}
Sending $z_*\to 0^{+}$ in \eqref{Fz*t0} and noticing $\lim_{z_*\to 0^+}F(z_*, t)/z_*\to 1$, we have
\begin{align}\label{Fz*t}
F(z, t)
&=z\exp \left(\int^{z}_{0}\left[\sqrt{Q^G(\tilde{z}, t)}-\frac{1}{\tilde{z}}\right]d\tilde{z}\right) 
\end{align}
where 
\begin{align}
    \lim_{z\to 0^+}\left|\sqrt{Q^G}(z, t)-\frac{1}{z}\right|=\lim_{z\to 0^+}\left|\frac{F_z}{F}-\frac{1}{z}\right|= \lim_{z\to 0^+}\left|\frac{zF_z-F}{Fz}\right|= \lim_{z\to 0^+}\left|\frac{zF_{zz}}{zF_z+F}\right|=0.
\end{align}
The claim now immediately follows.
\end{proof}

\bigskip


\begin{appendix}
\section{\quad\quad\quad A Gaussian weighted Hardy inequality}\label{appendix weights varied Gaussian Hardy}
For $a> 1$,  $A >0$ large, and $\xi \in (0, A]$ 
we define 
    \begin{equation}\label{rho weight}
    \rho(\xi)= \frac{
  \xi^{-a} e^{\frac {\xi^2}4} }{A^{-a}  e^{\frac {A^2}4}  + 2\int^{A}_{\xi} \sigma^{-a}  e^{\frac {\sigma^2}4} \,   \mathrm{d}\sigma}, 
 \qquad \xi\in (0, A]. 
    \end{equation}
One easily checks that $\rho$ is well defined on the interval $(0, A]$,  it  solves
\begin{equation}\label{eqn-rho}
    \rho'+(\frac{a}{\xi}-\frac{\xi}{2})\rho-2\rho^2=0
\end{equation}
and satisfies  
\begin{equation}\label{eqn-rho2}
 \rho>0,\quad \rho(A)=1,\quad  \lim_{\xi\to 0+}\xi \, \rho(\xi)=\tfrac{a-1}{2}>0.
\end{equation}

We will use the following elementary Hardy inequality, which we prove for the reader's convenience. 
 
\begin{lem}[{\cite[Lemma 3.4]{Hardy_Gaussian}}]\label{weights varied Gaussian Hardy}
For $\tilde{f}\in C^{1}([0, A])$, $\tilde{f}(0)=0$ and $a>1$,  we have
    \begin{equation}\label{eqn-Hardy} 
        \int_{0}^{A} \tilde{f}(\xi)^2 \rho^2(\xi) \xi^{a} e^{-\frac{\xi^2}{4}}  d\xi\leq   \int_{0}^{A} \tilde{f}' (\xi)^2 \xi^{a} e^{-\frac{\xi^2}{4}} d\xi+ \tilde{f} (A)^2 A^{a} e^{-\frac{A^2}{4}} 
    \end{equation}
\end{lem}
\begin{proof} First we observe that  $(\tilde{f}'+ \rho\tilde{f})^2 \geq 0$, implies  $ (\tilde{f}')^2 \geq -  \rho \, ( \tilde f^2)'- \rho^2 \tilde f^2$. 
Integrating the last inequality on  $[0, A]$ against the  weight $\xi^{a} e^{-\frac{\xi^2}{4}} d\xi$, while using integration by parts, equation \eqref{eqn-rho} and  conditions \eqref{eqn-rho2},  we obtain
       \begin{equation}
  \int_{0}^{A} \tilde{f}' (\xi)^2 \xi^{a} e^{-\frac{\xi^2}{4}} d\xi 
 \geq      - \tilde{f} (A)^2 A^{a} e^{-\frac{A^2}{4}} +\int_{0}^{A} |\tilde{f}(\xi)|^2 \rho^2(\xi) \xi^{a} e^{-\frac{\xi^2}{4}}  d\xi\    \end{equation}
which readily gives the desired inequality. 
\end{proof}



\section{\quad\quad\quad  Heat kernel estimates}
Recall that one can use the infinite  reflection method to show that the Dirichlet heat kernel on the domain  $\Omega := [-1,1]^k$ (Green's function of  the heat equation with Dirichlet boundary condition) is given by
\begin{align}
    K_t({\bf x},{\bf y}) =\!\!\!\!\!\!\sum_{\delta\in \{1,-1\}^{k},l\in\mathbb{Z}^{k}}\!\! \frac{(-1)^{\frac{\delta_{1}+\dots+\delta_{k}-k}{2}}}{(4\pi t)^{\frac{k}{2}}}\prod_{i=1}^{k}
    \exp\left(-\frac{\left(x_i-\delta_iy_i-(4l_{i}+1-\delta_{i})\frac{1}{4\sqrt{k}}\right)^2}{4t}\right),
\end{align}
for all ${\bf x}, {\bf y} \in \Omega = [-1,1]^k$, and all $t>0$. Note that we have boundary condition $\left.K_t({\bf x},{\bf y})\right|_{{\bf y}\in \partial \Omega} = 0$. We have the   following elementary estimates for Dirichlet heat kernel $K_t({\bf x},{\bf y})$.

\begin{lem}\label{HK-estimate}
There exists  a large constant $C$ such that the following hold on $\Omega:=(-1,1)^k$\\
(i) $K_1({\bf 0}, \bf   y) \geq \frac{1}{C} \Pi_{  i=1}^k\cos \frac{\pi y_i}{2}$ for all ${\bf y} \in\Omega$.\\
(ii) $\left|\left.\Delta_{{{\bf x}}} K_1({{\bf x}}, {\bf y})\right|_{{{\bf x}}={\bf 0}} \right\rvert\, \leq C \Pi_{  i=1}^k\cos \frac{\pi y_i}{2}$ for all ${\bf y} \in\Omega$.\\
(iii) $-\left.\partial_{\nu_{{\bf y}}} K_t({\bf 0}, {\bf y})\right|_{\partial \Omega} \geq \frac{1}{C} t^{-\frac{k}{2}-1} e^{-\frac{1}{4 t}}$ for all $t \in(0,1]$.\\
(iv) $\left|\left.\Delta_{{{\bf x}}} \partial_{\nu_{{\bf y}}} K_t({{\bf x}}, {\bf y})\right|_{{{\bf x}}={\bf 0}, \partial \Omega} \right\rvert\, \leq C t^{-\frac{k}{2}-3} e^{-\frac{1}{4 t}}$  for all $t \in(0,1]$.
\end{lem}

\begin{proof}
We can find a small constant $\tau \in(0,1)$ such that the following holds:
\begin{equation}
    -\left.\partial_{\nu_{{\bf y}}} K_t({\bf 0}, {\bf y})\right|_{\partial \Omega} \geq C^{-1} \,  t^{-\frac{k}{2}-1} e^{-\frac{1}{4 t}}
\end{equation}
and 
 \begin{equation}
     \left|\left.\Delta_{{{\bf x}}} \partial_{\nu_{{\bf y}}} K_t({{\bf x}}, {\bf y})\right|_{{{\bf x}}={\bf 0}, \partial \Omega} \right\rvert\, \leq C\,  t^{-\frac{k}{2}-3} e^{-\frac{1}{4 t}}
 \end{equation}
 for all $t \in(0, \tau]$.
In particular, $-\left.\partial_{\nu_{{\bf y}}} K_t({\bf 0}, {\bf y})\right|_{\partial \Omega}$ is a positive number. Since $K_\tau({\bf 0}, {\bf y})>0$ for all ${\bf y} \in\Omega$, we can find a small number $\varepsilon>0$ such that $K_\tau({\bf 0}, {\bf y}) \geq \varepsilon \Pi_{  i=1}^k\cos \frac{\pi y_i}{2}$ for all ${\bf y} \in\Omega$. The maximum principle now implies $K_t({\bf 0}, {\bf y}) \geq \varepsilon e^{-\frac{\pi^2 t}{4}} \Pi_{  i=1}^k\cos \frac{\pi y_i}{2}$ for all ${\bf y} \in\Omega$ and all $t \in[\tau, 1]$. In particular, $K_1({\bf 0}, {\bf y}) \geq \varepsilon e^{-\frac{\pi^2}{4}} \Pi_{  i=1}^k\cos \frac{\pi y_i}{2}$. From this, statements (i), (iii), and (iv) follow.

To prove statement (ii), we observe that the function $y \mapsto \left.\Delta_{{{\bf x}}} K_1({{\bf x}}, {\bf y})\right|_{{{\bf x}}={\bf 0}}$ is smooth, and vanishes at $\partial \Omega$. Consequently, $\left|\left.\Delta_{{{\bf x}}} K_1({{\bf x}}, {\bf y})\right|_{{{\bf x}}={\bf 0}} \right\rvert\, \leq C \Pi_{  i=1}^k\cos \frac{\pi y_i}{2}$ for all ${\bf y} \in\Omega$. This proves (ii).
\end{proof}

\begin{prop}\label{refine-Laplace-esitimate}
Let $h({{\bf x}}, t)$ be a nonnegative solution of the heat equation $h_t=\Delta h$ on the rectangle $\Omega\times [-1,0] = [-1,1]^k \times[-1,0]$. Then, for each $\mu \in(0,1)$,
$$
\left|\Delta h({\bf 0},0)\right| \leq C \mu^{-2} h({\bf 0},0)+C e^{-\frac{1}{8 \mu}} \sup _{\partial \Omega \times[-1,0]} h,
$$
where $C$ is a constant.
\end{prop}

\begin{proof}
For any point ${{\bf x}}\in \Omega$, and each $t \in(0,1]$, we define
\begin{align}
I(t):=\int_{\Omega} K_t({{\bf x}}, {\bf y}) h({\bf y},-t) d {\bf y}.
\end{align}
Then
\begin{align}
\begin{aligned}
I^{\prime}(t) & =\int_{\Omega}\left[h({\bf y},-t) \frac{\partial}{\partial t} K_t({{\bf x}}, {\bf y})+K_t({{\bf x}}, {\bf y}) \frac{\partial}{\partial t} h({\bf y},-t)\right] d {\bf y} \\
& =\int_{\Omega}\left[h({\bf y},-t) \Delta_{{\bf y}} K_t({{\bf x}}, {\bf y})-K_t(x, y) \Delta_{{\bf y}} h({\bf y},-t)\right] d {\bf y} \\
& =\int_{\partial\Omega}h({\bf y},-t) \partial_{\nu_{{\bf y}}} K_t(x, y)d\sigma.
\end{aligned}
\end{align}
We now integrate this identity over $t \in(0,1]$. Since $\lim _{t \rightarrow 0} I(t)=h({{\bf x}}, 0)$, it follows that
\begin{align}
\begin{aligned}
h({{\bf x}}, 0) & =\int_{\Omega} K_1({{\bf x}}, {\bf y}) h({\bf y},-1) d {\bf y} \\
& -\int_0^1 \int_{\partial\Omega}h({\bf y},-t) \partial_{\nu_{{\bf y}}} K_t({{\bf x}}, {\bf y})d\sigma d t
\end{aligned}
\end{align}
for ${{\bf x}} \in \Omega$. We now put ${{\bf x}}={\bf 0}$. Using part (i) and (iii) of Lemma \ref{HK-estimate}, we obtain
\begin{align}
\begin{aligned}
h({\bf 0},0) & \geq \frac{1}{C} \int_{\Omega} \Pi_{  i=1}^k\cos \frac{\pi y_i}{2} h({\bf y},-1) d {\bf y} \\
& +\frac{1}{C} \int_0^1 t^{-\frac{k}{2}-1} e^{-\frac{1}{4 t}}\int_{\partial\Omega}h({\bf y},-t) d\sigma d t
\end{aligned}
\end{align}
Similarly, using part (ii) and (iv) of Lemma \ref{HK-estimate}, we obtain
\begin{align}
\begin{aligned}
\left|\Delta h({\bf 0},0)\right| & \leq C \int_{\Omega} \Pi_{i=1}^k\cos \frac{\pi y_i}{2} h({\bf y},-1) d {\bf y} \\
& +C \int_0^1 t^{-\frac{k}{2}-3} e^{-\frac{1}{4 t}}\int_{\partial\Omega}h({\bf y},-t) d\sigma d t.
\end{aligned}
\end{align}
Putting these facts together, we conclude that
\begin{align}
\begin{aligned}
\left|\Delta h({\bf 0},0)\right| & \leq C \int_{\Omega}  \Pi_{i=1}^k\cos \frac{\pi y_i}{2} h({\bf y},-1) d {\bf y} \\
& +C \int_\mu^1 t^{-\frac{k}{2}-3} e^{-\frac{1}{4 t}}\int_{\partial\Omega}h({\bf y},-t) d\sigma d t \\
& +C \int_0^\mu t^{-\frac{k}{2}-3} e^{-\frac{1}{4 t}}\int_{\partial\Omega}h({\bf y},-t) d\sigma d t \\
& \leq C \int_{\Omega}  \Pi_{i=1}^k\cos \frac{\pi y_i}{2} h({\bf y},-1) d {\bf y} \\
& +C \mu^{-2} \int_\mu^1 t^{-\frac{k}{2}-1} e^{-\frac{1}{4 t}}\int_{\partial\Omega}h({\bf y},-t) d\sigma d t \\
& +C e^{-\frac{1}{8 \mu}} \int_0^\mu \int_{\partial\Omega}h({\bf y},-t) d\sigma d t \\
& \leq C \mu^{-2} h({\bf 0},0)+C e^{-\frac{1}{8 \mu}} \sup _{\partial \Omega \times[-1,0]} h
\end{aligned}
\end{align}
for each $\mu \in(0,1)$. This completes the proof.
\end{proof}

\section{\quad\quad\quad   Regularity estimate of  $h=f_{\xi}$}\label{regularity estimates of tildeh}
In this section, we prove  certain  regularity estimates for $h=f_{\xi}$ in the parabolic region $z\leq C\sqrt{-t}$,  for some constant $C>0$. {
These estimates are  applied to evaluate the error terms of the evolution equation for $\frac{1}{2}G^2 + (n-k-1)t$ in \eqref{Hztdef} of Section \ref{rule out+}, which is crucial for ruling out the case of dominant positive mode.}

First, for any
positive weight function $\rho$ and $q>0$, we define the norm $\|\cdot\|_{L^{2}_{{\rho}}(Q(R))}$
\begin{equation}\label{weight l2 norm definition}
   \|h\|_{L^{q}_{{\rho}}(Q(R))}=\left(\int_{Q(R)} h(\xi, \tau)^q\rho(\xi) d\xi d\tau\right)^{\frac{1}{q}}
\end{equation}
where  $Q(R)=[0, R]\times [-R^2, 0]$ for  $R>0$. Then, we have the   following $C^\alpha$-estimates for $h$.

\begin{lem}[$C^{\alpha}$-estimates]\label{Linfinity estimates}
    For $h\leq 0$ with $h(0, \tau)=0$, $h(0, \tau)=0$ satisfying 
    \begin{equation}
\partial_{\tau}h=\mathcal{L}_2h+E_2,
\end{equation}
where $E_2$ is defined in \eqref{h-error-expansion}, we have 
   \begin{equation}
       \|\xi^{-2}h\|_{L^{\infty}(Q(R))}+\|\xi^{-2}h\|_{C^{\alpha}(Q(R))}\leq C(k, R)\left(\|h\|_{L^{2}_{{\rho}}(Q(2R))}+\|\xi^{-2}E_2\|_{L^q_{\tilde{\rho}}(Q(2R))}\right)
    \end{equation}
where $q>\frac{k}{2}+1$,   and
\begin{equation}
    {\rho}(\xi)=\xi^{k-3} e^{-\xi^2/4} \quad \mbox{and} \quad 
    \tilde{\rho}(\xi)=\xi^{k+1} e^{-\xi^2/4},  \quad \text{for}\,\, k\geq 2.
\end{equation}
\end{lem}
\begin{proof}
Noticing that $h(0, \tau)=0$,  $h(0, \tau)=0$, we consider $\tilde{h}$ satisfying
\begin{equation}
    h(\xi, \tau)=\xi^2\, \tilde{h}(\xi, \tau),
\end{equation}
and observe that  satisfies the evolution  equation,
\begin{align}\label{tildehequation}
\partial_{\tau}\tilde{h}&=\tilde{\mathcal{L}_2}\tilde{h}+\tilde{E}_2=\tilde{h}_{\xi\xi}+\left(\frac{k+1}{\xi}-\frac{\xi}{2}\right)\tilde{h}_{\xi}-\tilde{h}+\tilde{E}_2
\end{align}
where
\begin{equation}
    \tilde{\mathcal{L}_2} \tilde{h}=\tilde{h}_{\xi\xi}+\left(\frac{k+1}{\xi}-\frac{\xi}{2}\right)\tilde{h}_{\xi}-\tilde{h}
\end{equation}
\begin{equation}
    \tilde{E}_2=\frac{E_2}{\xi^2}.
\end{equation}

We note that both $T\tilde{h}=\xi^2\tilde{h}$ and its inverse $T^{-1}h=\xi^{-2}h$ are well defined in the function space $\mathscr{H}$ defined in \eqref{mathscrHdef},  and by  $h(0, \tau)=0$,  $h(0, \tau)=0$, we have the following self-similar conjugate relation
\begin{equation}
    \tilde{\mathcal{L}}_2 h=T^{-1}\circ \mathcal{L}_2\circ T(h). 
\end{equation}
Hence $\mathcal{L}_2$ has the same spectrum as $\tilde{\mathcal{L}}$, that is 
\begin{equation}\label{negative spectrum of L3}
    \sigma(\mathcal{L}_2)= \sigma(\tilde{\mathcal{L}_2})=\{-(m+1)\}_{m=0}^{\infty}
\end{equation}

We can  write the equation of $\tilde{h}$ in divergence form
\begin{equation}\label{eqn-C10}
\partial_{\tau}\tilde{h}=\tilde{\rho}^{-1}\left(
\tilde{\rho}\tilde{h}_\xi\right)_\xi
-\tilde{h}+ \tilde{E}_2.
\end{equation}
{Since $h \le 0$, we also have that $\tilde{h}\leq 0$.} Hence the  nonnegative solution 
\begin{equation}
    \hat{h}:=-\tilde{h}\geq 0,
\end{equation}
of equation \eqref{eqn-C10},  is a  subsolution of following equation
\begin{align}\label{hath-equation}
\partial_{\tau}\hat{h} \leq \tilde{\rho}^{-1}\left(
\tilde{\rho}\hat{h}_\xi\right)_\xi
- \tilde{E}_2
\end{align}
Note that the weight $\tilde{\rho}$ is non-singular for {$k\geq 2$}, and the operator $\tilde{\rho}^{-1}\left(
\tilde{\rho}\tilde{h}_\xi\right)_\xi$ is the radial symmetric $k+2$ dimensional self-adjoint Ornstein-Uhlenbeck operator acting on the $k+2$ dimensional Gaussian weighted space $L(\mathbb{R}^{k+2}; e^{-\frac{|y|^2}{4}})$. 
Hence,  the function 
\begin{equation}
\bar{h} :=\tilde{\rho}\, \hat{h}
\end{equation}
satisfies 
\begin{align}
\partial_{\tau}\bar{h}&=\left(\tilde{\rho}\left(
\tilde{\rho}^{-1}\bar {h}
\right)_\xi\right)_\xi
-\bar{h}- \tilde{\rho}\tilde{E}_2
\leq \left(\tilde{\rho}\left(
\tilde{\rho}^{-1}\bar {h}
\right)_\xi\right)_\xi
- \tilde{\rho}\tilde{E}_2
\end{align}
Therefore, by the similar parabolic Nash-Moser iteration and oscillation  argument as in  \cite[Lem A.1, Lem A.2]{nowingmcf} or \cite[Thm 6.17, Thm 6.9]{Lib}, we obtain
  \begin{equation}
       \|\hat{h}\|_{L^{\infty}(Q(R))}+\|\hat{h}\|_{C^{\alpha}(Q(R))}\leq C(k, q, R)\left(\|\hat{h}\|_{L^{2}_{\tilde{\rho}}(Q(2R))}+\|\tilde{E}_2\|_{L^q_{\tilde{\rho}}(Q(2R))}\right)\footnote{The parabolic Sobolev inequality in \cite[Thm 6.9]{Lib} can be easily extended to the $\tilde{\rho}$-Gaussian weighted parabolic Sobolev inequality on a bounded domain which we use in the Nash-Moser iteration process as in \cite[Appendix A]{nowingmcf}.},
    \end{equation}
where $q>\frac{k}{2}+1$. This implies
    \begin{align}
       &\quad \|\xi^{-2}h\|_{L^{\infty}(Q(R))}+\|\xi^{-2}h\|_{C^{\alpha}(Q(R))}\\\nonumber
       &\leq C(k, R)\left(\|\xi^{-2}h\|_{L^{2}_{\tilde{\rho}}(Q(2R))}+\|\xi^{-2}E_2\|_{L^q_{\tilde{\rho}}(Q(2R))}\right)\\\nonumber
       &= C(k, R) \left(\|h\|_{L^{2}_{{\rho}}(Q(2R))}+\|\xi^{-2}E_2\|_{L^q_{\tilde{\rho}}(Q(2R))}\right).\\\nonumber
    \end{align}

\end{proof}

We note that $\|h\|_{L^{2}_{{\rho}}(Q(2R))}$ 
has good decay estimates {as in \eqref{h-L2}}  by Merle-Zaag type estimates in Proposition \ref{discrete-MZ} and 
\begin{equation}
\|\xi^{-2}E_2\|_{L^q_{\tilde{\rho}}(Q(2R))}=\|\xi^{{\frac 4q-2}}E_2\|_{L^q_{{\rho}}(Q(2R))}
\end{equation}
and in particular
\begin{equation}
\|\xi^{-2}E_2\|_{L^2_{\tilde{\rho}}(Q(2R))}=\|E_2\|_{L^2_{{\rho}}(Q(2R))}
\end{equation}
which has good decay estimates by Proposition \ref{Error-E3-estimate}. Next we prove the estimate that we need in order to obtain later   Schauder type estimates.
\begin{lem}\label{x-2E3Calpha} We have 
    \begin{equation}
        \xi^{-2}E_{2}\in C^{\alpha}(Q(R)),
    \end{equation}
and
\begin{equation}\label{Calphax-2E3}
     \|\xi^{-2} E_{2}\|_{C^{\alpha}}\leq C
\end{equation}
where $Q(R)=[0, R]\times [-R^2, 0]$ for  $R>0$. 
\end{lem}
\begin{proof}
To simplify the notation, we will use the abbreviation $C^{0}=C^0(Q(R))$ and $C^{\alpha}=C^\alpha (Q(R))$.  We   first estimate  the $C^{0}$ norm of $\xi^{-2}\, E_{2}$ term by term,
where  $E_2$ is defined  in \eqref{h-error-expansion}. 

For $\xi\in [0, R]$, 
by \eqref{3.3E31}, $f(0)=0, f_{\xi}(0)=0, f_{\xi\xi}(0)=0$ in \eqref{fgx=0} and \eqref{fxx=0}, Lemma \ref{apriori-f} we have
\[
   \quad \left |  \xi^{-2}\, f_{\xi \xi } \, \Big[ \frac{1}{f+\xi }-\frac{1}{\xi}\Big] \right| = \left |\frac{f_{\xi\xi}}{f+\xi} \cdot \frac{f}{\xi^3}\right|
   =\left |\frac{f_{\xi\xi}}{f+\xi} \cdot f_{\xi\xi\xi}((\vartheta) \xi)\right |\leq C
\]
where $\vartheta=\vartheta(\xi)\in [0, 1]$  
may  vary  in each time we  apply the  intermediate value theorem. 
 { By \eqref{fxi/f+xi2}, \eqref{f-slope}, and the intermediate value theorem we have
  \begin{align}
&\quad\left|\xi^{-2}\left[-\frac{2f_{\xi}f_{\xi\xi}}{f+\xi}+\frac{f^2_{\xi}(f_{\xi}+1)}{(f+\xi)^2}\right]\right|\\\nonumber
&\le\left|\frac{2f_{\xi} f_{\xi\xi}}{\xi^2(f+\xi)}\right| + \left|\frac{f_{\xi}^2 (f_{\xi} + 1)}{\xi^2\,(f+\xi)^2}\right|\\\nonumber
&\le 2\frac{|f_{\xi}|}{(f+\xi)^2}\left(\frac{f}{\xi}+1\right)\, \frac{|f_{\xi\xi}|}{\xi} + \frac{f_{\xi}^2}{(f+\xi)^4}\left(1+\frac{f}{\xi}\right)^2\, (1+f_{\xi}) \\\nonumber
&\le C\, (|f_{\xi\xi\xi}(\vartheta\xi)| + 1) \le C.
  \end{align}}
{By Lemma \ref{apriori-f} and \eqref{fxi/f+xi2} we have}
  \begin{align}
\begin{aligned}
\xi^{-2}\left|\frac{f^2_{\xi}(f_{\xi}+1)}{(f+\xi)^2}-\frac{2f_{\xi}f_{\xi\xi}}{f+\xi}\right|\leq \epsilon|\xi^{-2}f_\xi|\leq  |f_{\xi\xi\xi}(\vartheta\xi)| \leq C.
\end{aligned}
\end{align}

{By the intermediate value theorem and \eqref{fxi/f+xi2} we have
\begin{align}
\begin{aligned}
    &\quad \left|\xi^{-2}f_{\xi\xi}(\xi,\tau)\int_0^\xi \frac{f_{\eta}(\eta,\tau)(f_{\eta}(\eta,\tau)+1)}{(f(\eta,\tau)+\eta)^2} d\eta\right|_{C^0(Q(R))}\\
    &\leq C\, \epsilon\, |f_{\xi\xi\xi}|_{C^0(Q(R))}\\
    &\leq C,
\end{aligned}
\end{align}
and also
\begin{align}
    &\quad\xi^{-2}\left|-(k-1)f_\xi(\xi,\tau)\frac{f_{\xi}(\xi,\tau)(f_{\xi}(\xi,\tau)+1)}{(f(\xi,\tau)+\xi)^2}\right|\\\nonumber
    &\leq C\, \epsilon\, |f_{\xi\xi\xi}|^2_{C^0(Q(R))}\\ \nonumber
    &\leq C,
\end{align}}
and 
\begin{align}
    &\quad\xi^{-2}\left|-(n-k)f_{\xi\xi}(\xi,\tau)\int_0^\xi \frac{g_{\eta}^2(\eta,\tau)}{(g(\eta,\tau) + \sqrt{2(n-k-1)})^2} d\eta \right|\\\nonumber
    &\leq \frac{2(n-k)}{(n-k-1)}|f_{\xi\xi\xi}|_{C^0(Q(R))} \, |g_{\xi}|^2_{C^0(Q(R))} \leq C
\end{align}
and
\begin{align}
    &\quad \xi^{-2}\left|-(n-k)(f_\xi(\xi,\tau) + 1) \frac{g_{\xi}^2(\xi,\tau)}{(g(\xi,\tau) + \sqrt{2(n-k-1)})^2}\right|\\\nonumber
    &\leq \frac{2(n-k)}{2(n-k-1)}\, |g_{\xi\xi}|^2_{C^0(Q(R))}\leq C. 
\end{align}
Combining all the above estimates, we have
\begin{equation}
{\|    \xi^{-2}E_{2} \|_{C^0(Q(R))} \leq C. } 
\end{equation}
Hence, by Lemma \ref{Linfinity estimates} we conclude that 
\begin{equation}\label{Calpha x-2h}
    \xi^{-2}f_{\xi}= \xi^{-2}h\in C^{\alpha}. 
\end{equation}

Next  we want to use \eqref{Calpha x-2h} and 
the following  estimates
\begin{equation}\label{eq:Ckth-product}
\|fg\|_{C^{\ell,\alpha}}
\le
C(n,\ell)
\Bigg(
\|f\|_{C^0}\|g\|_{C^{\ell,\alpha}}
+
\|g\|_{C^0}\|f\|_{C^{\ell,\alpha}}
+
\sum_{j=1}^{\ell}
\|f\|_{C^{j}}\|g\|_{C^{\ell-j}}
\Bigg),
\end{equation}
when $\ell=0$ to show that
\begin{equation}
     \|\xi^{-2} E_{2}\|_{C^{\alpha}}\leq C
\end{equation}
and  we will  estimate the $C^{\alpha}$ norm of $\xi^{-2}E_{2}$ term by term. We rewrite $E_2$ in \eqref{h-error-expansion} in the following form
\begin{align}
    E_{2}&=(k-3) {f_{\xi \xi }(\xi ,\tau)}\Big[\frac{1}{f(\xi ,\tau)+\xi }-\frac{1}{\xi}\Big] \\
    &-(2k-4){h(\xi,\tau)}\Big[\frac{h+1}{(f(\xi,\tau)+\xi)^2}-\frac{1}{\xi^2}\Big] \notag\\
    &-\frac{2hf_{\xi\xi}}{f+\xi}+\frac{h^2(h+1)}{(f+\xi)^2}-(k-1)f_{\xi\xi}(\xi,\tau)\int_0^\xi \frac{h(\eta,\tau)(h(\eta,\tau)+1)}{(f(\eta,\tau)+\eta)^2} d\eta\notag\\
    &-hf_\xi(\xi,\tau)\frac{h(\xi,\tau)(h(\xi,\tau)+1)}{(f(\xi,\tau)+\xi)^2} -(n-k)f_{\xi\xi}(\xi,\tau)\int_0^\xi \frac{g_{\eta}^2(\eta,\tau)}{(g(\eta,\tau) + \sqrt{2(n-k-1)})^2} d\eta  \notag\\
    & -(n-k)(h(\xi,\tau) + 1) \frac{g_{\xi}^2(\xi,\tau)}{(g(\xi,\tau) + \sqrt{2(n-k-1)})^2} \notag.
\end{align}
First observe that \eqref{Calpha x-2h} and \eqref{eq:Ckth-product} 
imply 
\begin{equation}
    h\in C^{\alpha}.
\end{equation}
Note also that
\begin{equation}
    \frac{f(\xi, \tau)}{\xi}=\int_{0}^1h(s\xi, \tau)ds,
\end{equation}
Let us now estimate the $C^{\alpha}$ norm of $\xi^{-2}\, E_2$. Using \eqref{eq:Ckth-product}, the first term gives
\begin{align}\label{xi-2E-1}
\begin{aligned}
    &\quad\left|(k-3) \xi^{-2}{f_{\xi \xi }(\xi ,\tau)}\Big[\frac{1}{f(\xi ,\tau)+\xi }-\frac{1}{\xi}\Big]\right|_{C^\alpha} \\
    &= \left|-(k-3)\frac{f_{\xi\xi}}{f+\xi} \cdot \frac{f}{\xi^3}\right|_{C^\alpha}= 
     \left|-(k-3)\frac{f_{\xi\xi}}{f+\xi}\xi^{-2}\int_{0}^1h(s\xi, \tau)ds\right|_{C^\alpha}\\
    & \leq \left|\frac{f_{\xi\xi}}{f+\xi}\right|_{C^0}\left|\xi^{-2}\int_{0}^1h(s\xi, \tau)ds\right|_{C^\alpha} + \left|\frac{f_{\xi\xi}}{f+\xi}\right|_{C^\alpha}\left|\xi^{-2}\int_{0}^1h(s\xi, \tau)ds\right|_{C^0}.
\end{aligned}
\end{align}
The second term gives
\begin{align}\label{xi-2E-2}
\begin{aligned}
    &\quad\left|-(2k-4)\xi^{-2}{h(\xi,\tau)}\Big[\frac{h+1}{(f(\xi,\tau)+\xi)^2}-\frac{1}{\xi^2}\Big]\right|_{C^\alpha}\\
    & = \left|-(2k-4)\xi^{-2}{h(\xi,\tau)}\Big[\frac{\xi^2h-f^2-2\xi f}{\xi^2(f(\xi,\tau)+\xi)^2}\Big]\right|_{C^\alpha}\\
    &=\left|\xi^{-2}h\left[\frac{\xi^2h-f^2-2\xi f}{\xi^4(\frac{f(\xi,\tau)}{\xi}+1)^2}\right]\right|_{C^\alpha}\\
    &=\left|\xi^{-2}h\left[\xi^{-2}h-\xi^{-2}h^2-2\xi^{-2}\int_{0}^1h(s\xi, \tau)ds)\right]\frac{1}{(1+\int_{0}^1h(s\xi, \tau)ds)^2}\right|_{C^\alpha}
\end{aligned}
\end{align}
The third and fourth terms give
\begin{align}\label{xi-2E-34}
\begin{aligned}
    &\quad \left|-\frac{2hf_{\xi\xi}}{\xi^2(f+\xi)}+\frac{h^2(h+1)}{\xi^2(f+\xi)^2}\right|_{C^\alpha} \leq \left|\xi^{-2}h\frac{f_{\xi\xi}}{f+\xi}\right|_{C^\alpha} + \left|(\xi^{-2}h)^2\frac{h(\xi, \tau)+1}{(\int_{0}^1h(s\xi, \tau)ds+1)^2}\right|_{C^\alpha}
\end{aligned}
\end{align}
For estimating the fifth term, we will use the following claim.
\begin{claim}\label{average-horder}
Suppose that $\phi(\xi)\in C^{k,\alpha}([0, R])$, then we have  
\begin{align}
H\phi(\xi) = \frac{1}{\xi}\int_0^\xi \phi(\eta) d\eta \in C^{k,\alpha}([0, R]). 
\end{align}
\end{claim}

\begin{proof}
For $x>0$, change variables $t = s\xi$, we have
\begin{align}
(H\phi)(\xi)
= \frac{1}{\xi} \int_0^\xi \phi(t)\,dt
= \int_0^1 \phi(s\xi)\,ds.
\end{align}
Since $\phi(\xi)\in C^{\ell,\alpha}([0, R])$, we know $(H\phi)(0) = \phi(0)$, this implies $H\phi(\xi) = \frac{1}{\xi}\int_0^\xi \phi(\eta) d\eta \in C^{0}([0, R])$. Moreover, we can show that 
\begin{align}
\begin{aligned}
(H\phi)^{(m)}(\xi)
&= \partial_\xi^m \int_0^1 \phi(s\xi)\,ds = \int_0^1 \partial_\xi^m \big(\phi(s\xi)\big)\,ds \\
&= \int_0^1 s^m \phi^{(m)}(s\xi)\,ds
\end{aligned}
\end{align}
for all $\xi>0$ and $0\leq m\leq \ell$. Using $\phi(\xi)\in C^{\ell,\alpha}([0, R])$ again, we know that $H\phi(\xi) = \frac{1}{\xi}\int_0^\xi \phi(\eta) d\eta \in C^{\ell}([0, R])$.
For any $x,y\in[0,R]$, we compute
\begin{align}
\begin{aligned}
(H\phi)^{(\ell)}(x)-(H\phi)^{(\ell)}(y)
&= \int_0^1 s^\ell\big(\phi^{(\ell)}(sx)-\phi^{(\ell)}(sy)\big)\,ds.
\end{aligned}
\end{align}
Taking absolute values and using the $\alpha$–Hölder continuity of $\phi^{(\ell)}$, we obtain
\begin{align}
\begin{aligned}
\big|(H\phi)^{(\ell)}(x)-(H\phi)^{(\ell)}(y)\big|
&\leq \int_0^1 s^\ell
\big|\phi^{(\ell)}(sx)-\phi^{(\ell)}(sy)\big|\,ds \\
&\leq \int_0^1 s^\ell
[\phi^{(\ell)}]_{C^{0,\alpha}([0,R])}
|sx-sy|^\alpha\,ds \\
&= [\phi^{(\ell)}]_{C^{0,\alpha}([0,R])}
|x-y|^\alpha
\int_0^1 s^{\ell+\alpha}\,ds.
\end{aligned}
\end{align}
Thus, we conclude
\begin{align}
[(H\phi)^{(\ell)}]_{C^{0,\alpha}([0,R])}
\leq
\frac{1}{\ell+\alpha+1}
\,[\phi^{(\ell)}]_{C^{0,\alpha}([0,R])}.
\end{align}
Here we finish the proof of the claim.
\end{proof}

By Claim \ref{average-horder}, the fifth term gives
\begin{align}
\begin{aligned}
    &\quad \left|\xi^{-2}f_{\xi\xi}(\xi,\tau)\int_0^\xi \frac{h(\eta,\tau)(h(\eta,\tau)+1)}{(f(\eta,\tau)+\eta)^2} d\eta\right|_{C^\alpha}\\
    &= \left|\frac{f_{\xi\xi}}{f+\xi}\frac{f+\xi}{\xi}\frac{1}{\xi}\int_0^\xi \frac{h(\eta,\tau)(h(\eta,\tau)+1)}{(f(\eta,\tau)+\eta)^2} d\eta\right|_{C^\alpha}\\
    & = \left|\frac{f_{\xi\xi}}{f+\xi}\left(\int_{0}^1h(s\xi, \tau)ds+1\right)\frac{1}{\xi}\int_0^\xi \frac{h(\eta,\tau)(h(\eta,\tau)+1)}{(f(\eta,\tau)+\eta)^2}\right|_{C^\alpha}\\
    & \leq C\left|\frac{f_{\xi\xi}}{f+\xi}\left(\int_{0}^1h(s\xi, \tau)ds+1\right)\right|_{C^0}\left|\frac{h(\xi, \tau)(h(\xi, \tau)+1)}{\xi^2(\int_{0}^1h(s\xi, \tau)ds+1)^2}\right|_{C^\alpha} \\
    &+ C\left|\frac{f_{\xi\xi}}{f+\xi}\left(\int_{0}^1h(s\xi, \tau)ds+1\right)\right|_{C^\alpha}\left|\frac{h(\xi, \tau)(h(\xi, \tau)+1)}{\xi^2(\int_{0}^1h(s\xi, \tau)ds+1)^2}\right|_{C^0}.
\end{aligned}
\end{align}
The sixth term gives 
\begin{align}
\begin{aligned}
    &\quad \left|-(k-1)\xi^{-2}h(\xi,\tau)\frac{h(\xi,\tau)(h(\xi,\tau)+1)}{(f(\xi,\tau)+\xi)^2}\right|_{C^\alpha}\\
    & = \left|-(k-1)(\xi^{-2}h(\xi, \tau))^2\frac{h(\xi, \tau)+1}{(\int_{0}^1h(s\xi, \tau)ds+1)^2}\right|_{C^\alpha}.
\end{aligned}
\end{align}
By Claim \ref{average-horder}, the seventh term gives
\begin{align}
\begin{aligned}
    &\quad \left|-(n-k)\xi^{-2}f_{\xi\xi}(\xi,\tau)\int_0^\xi \frac{g_{\eta}^2(\eta,\tau)}{(g(\eta,\tau) + \sqrt{2(n-k-1)})^2} d\eta\right|_{C^\alpha}\\
    & = \left|-(n-k)\frac{f_{\xi\xi}}{f+\xi}\frac{f+\xi}{\xi}\frac{1}{\xi}\int_0^\xi \frac{g_{\eta}^2(\eta,\tau)}{(g(\eta,\tau) + \sqrt{2(n-k-1)})^2} d\eta\right|_{C^\alpha}\\
    & \leq C\left|\frac{f_{\xi\xi}}{f+\xi}(\int_{0}^1h(s\xi, \tau)ds+1)\right|_{C^0}\left|\frac{g_{\xi}^2(\xi)}{(g(\xi) + \sqrt{2(n-k-1)})^2}\right|_{C^\alpha}\\
    &+ C\left|\frac{f_{\xi\xi}}{f+\xi}(\int_{0}^1h(s\xi, \tau)ds+1)\right|_{C^\alpha}\left|\frac{g_{\xi}^2(\xi)}{(g(\xi) + \sqrt{2(n-k-1)})^2}\right|_{C^0}.
\end{aligned}
\end{align}
For the last term, we note that $g_{\xi}(0, \tau)=0$
\begin{equation}
    g_{\xi}(\xi, \tau)= \int_{0}^{\xi}g_{\eta\eta}(\eta, \tau)d\eta
\end{equation}
Hence, by this representation, bounded $C^{\alpha}$ norm of $g_{\xi\xi}$ in bounded region and Claim \ref{average-horder}
\begin{equation}
   \left|\frac{g_{\xi}}{\xi}\right|_{C^{\alpha}}\leq C
\end{equation}

last term gives
\begin{align}
    &\quad \left|\xi^{-2}(f_\xi(\xi,\tau) + 1) \frac{g_{\xi}^2(\xi,\tau)}{(g(\xi,\tau) + \sqrt{2(n-k-1)})^2} \right|_{C^{\alpha}}\\\nonumber
    &\leq   \left| \frac{f_\xi(\xi,\tau) + 1}{(g(\xi,\tau) + \sqrt{2(n-k-1)})^2} \right|_{C^{\alpha}}\left|\frac{g^2_{\xi}}{\xi^2}\right|_{C^{0}}+\left| \frac{f_\xi(\xi,\tau) + 1}{(g(\xi,\tau) + \sqrt{2(n-k-1)})^2} \right|_{C^{0}}\left|\frac{g_{\xi}}{\xi}\right|^2_{C^{\alpha}}
\end{align}
Combining above estimates, the $C^{\alpha}$ estimates of $\xi^{-2}E_2$ follows directly.
\end{proof}
\begin{lem}[$L^q_{\tilde{\rho}}$ estimates of $\xi^{-2}E_2$]\label{Lqx-2E3}
 \begin{align}
 \begin{aligned}
\|\xi^{-2}E_2\|_{L^q_{\tilde{\rho}} (Q(2R))}&\leq C\|\xi^{-2}E_2\|^{1-\theta}_{L^2_{\tilde{\rho}}(Q(2R))}\|\xi^{-2}E_2\|^{\theta}_{C^0 (Q(2R))}\\
&=C\|E_2\|^{1-\theta}_{L^2_{{\rho}}(Q(2R))}\|\xi^{-2}E_2\|^{\theta}_{C^0 (Q(2R))}. 
\end{aligned}
 \end{align}
\end{lem}
\begin{proof}
This follows from the standard interpolation inequality and the following identity
\begin{equation}
\|\xi^{-2}E_2\|_{L^2_{\tilde{\rho}}(Q(2R))}=\|E_2\|_{L^2_{{\rho}}(Q(2R))}.
\end{equation}

\end{proof}

\begin{lem}[$C^{\alpha}$-estimates of $\xi^{-2}h$]\label{Calphax-2h}We  have 
\begin{equation}
\|\xi^{-2}h\|_{C^{\alpha}(Q(R))}\leq C(k, R)\left(\|h\|_{L^{2}_{{\rho}}(Q(2R))}+\|E_2\|^{1-\theta}_{L^2_{{\rho}}(Q(2R))}\|\xi^{-2}E_2\|^{\theta}_{C^0 (Q(2R))}\right).
\end{equation}
\end{lem}
\begin{proof}
    This directly follows from Lemma \ref{Linfinity estimates} and Lemma \ref{Lqx-2E3}.
\end{proof}
\begin{lem}[Schauder-estimates]\label{Schauder-estimatesh}
     For $h\leq 0$ with $h(0, \tau)=0$, $h(0, \tau)=0$ satisfying 
    \begin{equation}
\partial_{\tau}h=\mathcal{L}_2h+E_2,
\end{equation}
we have 
\begin{equation}
 \|\xi^{-2} E_{2}\|_{C^{m, \alpha}(Q(R))}\leq C(m, k, R)
\end{equation}
and
   \begin{equation}
       \|\xi^{-2}h\|_{C^{2, \alpha}(Q(\frac{R}{2}))}\leq C(k, R)\left(\|\xi^{-2}h\|_{C^{\alpha}(Q(R))}+\|\xi^{-2}E_2\|_{C^{\alpha}(Q(R))}\right)
    \end{equation}
and
\begin{equation}
       \|\xi^{-2}h\|_{C^{m+2, \alpha}(Q(\frac{R}{2}))}\leq C(m, k, R)\left(\|\xi^{-2}h\|_{C^{m, \alpha}(Q(R))}+\|\xi^{-2}E_2\|_{C^{m, \alpha}(Q(R))}\right)
    \end{equation}
where $Q(R)=[0, R]\times [-R^2, 0]$ and $R>0$ .
\end{lem}
\begin{proof}
For 
\begin{equation}
    h(\xi, \tau)=\xi^2\tilde{h}(\xi, \tau)
\end{equation}
we have
\begin{equation}
\partial_{\tau}\tilde{h}=\tilde{\rho}^{-1}\left(
\tilde{\rho}\tilde{h}_\xi\right)_\xi
-\tilde{h}+ \tilde{E}_2,
\end{equation}
where
\begin{equation}
    \tilde{\rho}(\xi)=\xi^{k+1} e^{-\xi^2/4} \quad \text{for}\quad k\geq 1
\end{equation}
and
\begin{equation}
    \tilde{E}_2=\frac{E_2}{\xi^2}
\end{equation}
By standard Schauder theory, we obtain
  \begin{align}
       \|\tilde{h}\|_{C^{2, \alpha}(Q(\frac{R}{2}))}\leq C(k, R)\left(\|\tilde{h}\|_{C^{\alpha}(Q(R))}+\|\tilde{E}_2\|_{C^{\alpha}(Q(R))}\right),
    \end{align}
Translating back to original variable, we obtain 
  \begin{equation}
       \|\xi^{-2}h\|_{C^{2, \alpha}(Q(\frac{R}{2}))}\leq C(k, R)\left(\|\xi^{-2}h\|_{C^{\alpha}(Q(R))}+\|\xi^{-2}E_2\|_{C^{\alpha}(Q(R))}\right)
    \end{equation}
The $m$-th higher order Schauder estimates follows from a bootstrap argument and the inductive estimates  
   \begin{equation}
    \xi^{-2}h\in  C^{m, \alpha}\quad  \xi^{-2} E_{2}\in C^{m, \alpha}
\end{equation}
and
\begin{equation}
    \|\xi^{-2}h\|_{C^{m, \alpha}}+\|\xi^{-2} E_{2}\|_{C^{m, \alpha}}\leq C
\end{equation}
via \eqref{eq:Ckth-product} and a similar argument in the proof of Lemma \ref{x-2E3Calpha}.
\end{proof}
\begin{lem}[Decay estimates of derivatives of $h$]\label{decayDlh}We have 
    \begin{align}
&\|\xi^{-2}h\|_{C^{\ell}(Q(R))}\leq C(\ell, k, R)\left(\|h\|_{L^{2}_{{\rho}}(Q(2R))}+\|E_2\|^{1-\theta}_{L^2_{{\rho}}(Q(2R))}\right)^{1-\theta}.
\end{align}
\end{lem}
\begin{proof}
By Lemma \ref{Linfinity estimates} and interpolation inequality, 
we have 
   \begin{align}
      \|\xi^{-2}h\|_{C^{\alpha}(Q(R))}&\leq C(k, R)\left(\|h\|_{L^{2}_{{\rho}}(Q(2R))}+\|\xi^{-2}E_2\|_{L^q_{\tilde{\rho}}(Q(2R))}\right)\\\nonumber
      &\leq C(k, R)\left(\|h\|_{L^{2}_{{\rho}}(Q(2R))}+C(\theta)\|E_2\|^{1-\theta}_{L^2_{{\rho}}(Q(2R))}\|\xi^{-2}E_2\|^{\theta}_{C^0 (Q(2R))}\right).
    \end{align}
    By Schauder estimates in Lemma \ref{Schauder-estimatesh}, \eqref{Calphax-2E3} and the decay estimates in \eqref{hL2tdecay} and \eqref{E-L2Q},
  we have
  \begin{equation}\label{Cm+2alphax-2E3}
        \|\xi^{-2}h\|_{C^{m+2, \alpha}(Q(\frac{R}{2}))}\leq C(m, k, R)
  \end{equation}
  By interpolation estimates,  the above estimates \eqref{Cm+2alphax-2E3},  and the estimates in Lemma \ref{Calphax-2h}, and taking $\theta>0$ small enough and $m$ large enough, we have 
 \begin{align}
&\|\xi^{-2}h\|_{C^{\ell}(Q(R))}\\\nonumber
&\leq C(k,\ell, R)\left(\|h\|_{L^{2}_{{\rho}}(Q(2R))}+\|E_2\|^{1-\theta}_{L^2_{{\rho}}(Q(2R))}\|\xi^{-2}E_2\|^{\theta}_{C^0 (Q(2R))}\right)^{1-\theta}\|\xi^{-2}h\|^{\theta}_{C^{m+2}(Q(2R))}\\\nonumber
&\leq C(\ell, k, R)\left(\|h\|_{L^{2}_{{\rho}}(Q(2R))}+\|E_2\|^{1-\theta}_{L^2_{{\rho}}(Q(2R))}\|\xi^{-2}E_2\|^{\theta}_{C^0 (Q(2R))}\right)^{1-\theta}. 
\end{align}
Hence together with \eqref{Calphax-2E3} in Lemma \ref{x-2E3Calpha}, we have
  \begin{align}
        \|\xi^{-2}h\|_{C^{\ell, \alpha}(Q(\frac{R}{2}))}&\leq C(\ell, k, R)\left(\|h\|_{L^{2}_{{\rho}}(Q(2R))}+\|E_2\|^{1-\theta}_{L^2_{{\rho}}(Q(2R))}\|\xi^{-2}E_2\|^{\theta}_{C^0 (Q(2R))}\right)^{1-\theta}\\\nonumber
        &\leq C(\ell, k, R)\left(\|h\|_{L^{2}_{{\rho}}(Q(2R))}+\|E_2\|^{1-\theta}_{L^2_{{\rho}}(Q(2R))}\right)^{1-\theta}. 
  \end{align}
\end{proof}
\end{appendix}

\bibliography{reference}
\bibliographystyle{plain}

\vspace{5mm}

{\sc Panagiota Daskalopoulos, Department of Mathematics, Columbia University, New York,   10027, New York State,  United States}

{\sc Wenkui Du,  Department of Mathematics, Massachusetts Institute of Technology, Massachusetts, 02139, USA}

{\sc Natasa Sesum, Department of Mathematics, Rutgers University, Piscataway, 08854, New Jersey State,  United States}

{\sc Ziyi Zhao, School of Mathematical Sciences, Zhejiang Normal University, Jinhua, 321004, Zhejiang Province, China}

\vspace{5mm}

\emph{E-mail:} pdaskalo@math.columbia.edu, 
duwenkui15@gmail.com, natasas@math.rutgers.edu,  zyzhao@zjnu.edu.cn.

\end{document}